\documentclass{amsart}

\usepackage{amsmath, amssymb}
\usepackage{thmtools} 
\usepackage[table,dvipsnames]{xcolor}
\usepackage{tikz, tikz-cd}
\usepackage[hyphens]{url}
\usepackage{hyperref}
\hypersetup{
	colorlinks = true,
	linkcolor = {blue},
	urlcolor = {red},
	citecolor = {blue}
}
\usepackage{tabularx}
\usepackage[textsize=footnotesize]{todonotes} 
\usepackage{mathabx} 
\usepackage{enumitem}
\setlist{topsep=0.2em, itemsep=0.2em} 
\usepackage[nameinlink,capitalise]{cleveref} 
\crefalias{axiom}{enumi}
\crefname{axiom}{Axiom}{Axioms}
\Crefname{axiom}{Axiom}{Axioms}
\crefname{assumption}{Assumption}{Assumptions}
\Crefname{assumption}{Assumption}{Assumptions}
\crefname{figure}{Figure}{Figures}
\Crefname{figure}{Figure}{Figures}

\crefalias{Iaxiom}{enumi}
\crefname{Iaxiom}{Index Set Axiom}{Index Set Axioms}
\Crefname{Iaxiom}{Index Set Axiom}{Index Set Axioms}

\crefalias{Paxiom}{enumi}
\crefname{Paxiom}{Pre-HHS Axiom}{Pre-HHS Axioms}
\Crefname{Paxiom}{Pre-HHS Axiom}{Pre-HHS Axioms}

\crefalias{Aaxiom}{enumi}
\crefname{Aaxiom}{AIS Axiom}{AIS Axioms}
\Crefname{Aaxiom}{AIS Axiom}{AIS Axioms}

\crefalias{Haxiom}{enumi}
\crefname{Haxiom}{{HHS} Axiom}{{HHS} Axioms}
\Crefname{Haxiom}{{HHS} Axiom}{{HHS} Axioms}

\crefalias{HaxiomS}{enumi}
\crefname{HaxiomS}{{HHS} Axiom}{{HHS} Axioms}
\Crefname{HaxiomS}{{HHS} Axiom}{{HHS} Axioms}

\usepackage{makecell} 

\newtheorem{theorem}{Theorem}[section]
\newtheorem{proposition}[theorem]{Proposition}
\newtheorem{lemma}[theorem]{Lemma}
\newtheorem{corollary}[theorem]{Corollary}

\theoremstyle{definition}
\newtheorem{definition}[theorem]{Definition}
\newtheorem{example}[theorem]{Example}

\theoremstyle{remark}
\newtheorem{remark}[theorem]{Remark}
\newtheorem{warning}[theorem]{Warning}
\newtheorem{question}[theorem]{Question}
\newtheorem{conjecture}[theorem]{Conjecture}

\numberwithin{equation}{section}

\newcommand\bC{\mathbb{C}}

\newcommand\bH{\mathbb{H}}

\newcommand\bP{\mathbb{P}}
\newcommand\bQ{\mathbb{Q}}
\newcommand\bR{\mathbb{R}}

\newcommand\bZ{\mathbb{Z}}
\newcommand\cA{\mathcal{A}}
\newcommand\cB{\mathcal{B}}
\newcommand\cC{\mathcal{C}}
\newcommand\cD{\mathcal{D}}
\newcommand\cE{\mathcal{E}}

\newcommand\cH{\mathcal{H}}

\newcommand\cM{\mathcal{M}}

\newcommand\cO{\mathcal{O}}
\newcommand\cP{\mathcal{P}}

\newcommand\cT{\mathcal{T}}
\newcommand\cU{\mathcal{U}}

\newcommand\cY{\mathcal{Y}}
\newcommand\cX{\mathcal{X}}

\newcommand\mbU{\mathbf{U}}
\newcommand\mbV{\mathbf{V}}

\def\Thatsw{{\That^{\,\sw}}}
\def\Chat{\widehat{\C}}
\def\Chatcirc{\Chat^\circ}
\def\Ccirc{\C^\circ}

\def\Esw{E^\sw}

\def\yhat{\hat{y}}

\def\set#1#2{\{#1\mathrel{:}#2\}}

\def\Shatzap{\Shat^{\,!}}
\def\Tcirc{T^\circ}
\def\Dcirc{D^\circ}
\def\Vcirc{V^\circ}
\def\Xhat{\widehat{X}}
\def\Vhatcirc{\Vhat^\circ}
\def\Dhatcirc{\Dhat^\circ}
\def\length{\ell}
\def\gammahat{\hat\gamma}
\def\deltahat{\hat\delta}

\def\iso{\cong}

\def\H{\mathcal{H}}
\def\Ghat{\hat{G}}

\def\chat{{\hat{c}}}
\def\Gammahat{\widehat{\Gamma}}

\def\HHAT{\widehat{\H}}
\def\Hhat{\widehat{H}}
\def\pHhat{p_{\!\hat{H}}}
\def\PHhat{p_{\!N^1(\hat{H})}}
\def\That{\widehat{T}}
\def\Uhat{\widehat{U}}

\def\G{\Gamma}
\def\sset{\subseteq}
\def\PU{\mathrm{PU}}
\def\Bcirc{B^\circ}
\def\Bthick{B^\theta}
\def\Bhat{\widehat{B}}
\def\Bhatcirc{\Bhat^\circ}
\def\Thatcirc{\That^\circ}
\def\Bhatthick{\Bhat^\theta}
\def\Shat{\widehat{S}}
\def\Shatcirc{\Shat^\circ{}}
\def\Scirc{S^\circ}
\def\piorb{\pi_1}
\def\sw{{\mathrm{sw}}}
\def\back{\backslash}
\def\Z{\mathbb{Z}}
\def\R{\mathbb{R}}
\def\C{\mathbb{C}}
\def\xhat{\hat{x}}
\def\Vhat{\widehat{V}}
\def\Dhat{\widehat{D}}
\def\IndexBig{\mfS}
\def\IndexSmall{\mfS^\flat}

\newcommand{\mfS}{\mathfrak{S}}

\renewcommand{\Re}{\mathrm{Re}}
\renewcommand{\diamond}{{{\diamondsuit}}}
\renewcommand{\int}{\mathrm{int}}

\DeclareMathOperator{\Aut}{Aut}
\DeclareMathOperator{\Stab}{Stab}
\DeclareMathOperator{\diam}{diam}
\DeclareMathOperator{\Isom}{Isom}
\DeclareMathOperator{\Homeo}{Homeo}
\DeclareMathOperator{\Symp}{Symp}

\newcommand{\spc}[3]{  \makecell{  \color{teal} #1 \color{black}\vspace{0.15cm} \\  \color{violet}#2 \color{black} \vspace{0.15cm}  \\ \color{olive}#3 \color{black} }}
\newcommand{\spcbounded}[3]{  \makecell{ #2  \vspace{0.15cm}  \\ #3  }}

\newcommand{\ol}[1]{\makebox[0pt]{$\phantom{#1}\overline{\phantom{#1}}$}#1}
 
\renewcommand{\bold}[1]{\medskip \noindent {\bf #1 }\nopagebreak}

\begin{document}

\title[Hierarchical hyperbolicity, ball quotients, and cubic surfaces]{Hierarchical hyperbolicity, ball quotients, and the moduli space of smooth cubic surfaces}

\author{Daniel Allcock}

\author{Alex Wright}

\begin{abstract}
We prove that the orbifold fundamental group of the moduli space of smooth complex cubic surfaces is a hierarchically hyperbolic group.
This places it within a powerful geometric framework whose prototypical examples are mapping class groups. 
The proof uses a known incomplete complex hyperbolic metric on the moduli space, 
and applies in a broader setting which includes a number of other moduli spaces.  
\end{abstract}

\maketitle

\thispagestyle{empty}

\vfill
\setcounter{tocdepth}{1}
\tableofcontents
\vfill\vfill\vfill

\newpage
\section{Introduction}

\subsection{Main results.} 

The moduli space $\cM_{g,k}$ of Riemann surfaces with genus~$g$
and~$k$ punctures dates from the 19th century.  
The corresponding mapping class group is its  
orbifold 
fundamental
group, and has been a focus of geometric group theory for decades.
Previously our understanding of these groups 
relied on elaborate
and subtle
constructions specific to them, but starting about ten years ago they and several other major classes of groups can be understood
from the single perspective of hierarchically hyperbolic groups.
In this paper we show that this theory also applies to many 
groups from algebraic geometry, including
the fundamental group of another 19th century moduli space:

\begin{theorem}\label{T:Cubic}
The orbifold fundamental group $\piorb(\cM_{cs})$, of the moduli space $\cM_{cs}$ of smooth
complex cubic surfaces, is a hierarchically hyperbolic group (HHG).
\end{theorem}

Briefly, cubic surfaces are the subvarieties of $\C P^3$
defined by homogeneous cubic polynomials, and $\cM_{cs}$ is the space (orbifold)
that parameterizes the smooth ones up to isomorphism (or equivalently
projective equivalence).  
Our hierarchical hyperbolic structure was inspired by
a structural similarity between $\cM_{cs}$ and $\cM_{g,k}$.

Our proof does not use any properties of cubic surfaces.  
It relies
on the fact that $\cM_{cs}$
can be described as the complex $4$-ball (also called 
complex hyperbolic $4$-space), minus an arrangement of hyperplanes
with certain properties, modulo a lattice in $PU(4,1)$ \cite{AllcockCarlsonToledoSurfaces}.  
Specifically, 
\cref{T:Cubic} follows from our next result.  
See \cref{SS:VerificationForMcs} for this implication
and \cref{A:OtherEx} for many more examples.

\begin{theorem}[Main theorem]
\label{T:main}
Suppose $\H$ is a union of 
hyperplanes in the complex
$n$-ball  $B:=B^n$ 
satisfying the following.
\begin{enumerate}[label=(\Alph*)]
\item\label[assumption]{LocFin} Each point of $B$ has a neighborhood meeting only
finitely many hyperplanes.  
\item\label[assumption]{Orth}  Any two hyperplanes that meet do so orthogonally.
\item\label[assumption]{Inv}  $\H$ is invariant under a lattice $\G\sset\PU(n,1)$.
\item\label[assumption]{CuspLimit} Each cusp of $\G$ is a limit point of a $1$-dimensional
stratum. 
\end{enumerate}
Then the orbifold fundamental group $\Gammahat$ of
the complex analytic orbifold $\Gamma\back (B-\H)$ is hierarchically hyperbolic. 
\end{theorem}

\noindent

We restrict the term ``hyperplane in~$B$''  to the
hyperplanes comprising~$\H$, because we will refer to them often and never
to other hyperplanes in~$B$.
The stratification in \ref{CuspLimit} is the one 
induced by any hyperplane arrangement satisfying \ref{LocFin}--\ref{Orth}:
the codimension $k$ strata are the connected components of the set of points contained 
in exactly $k$ hyperplanes. Whenever we use fundamental groups and universal covers it will be in the orbifold sense.  
For example, $B-\H$ is a covering space of $\Gamma\back(B-\H)$, so
$\Gammahat$ contains $\pi_1(B-\H)$ as a normal subgroup, with quotient~$\Gamma$.

\medskip
We will expand on what an HHG is, but first mention
some consequences:
 solvability of the word and conjugacy problems (\cite[Corollary 7.5]{BehrstockHagenSistoHHSII} and \cite[Corollary H]{HaettelHodaPetyt}),  analogues of the Nielsen-Thurston classification  and the Tits alternative (\cite[Proposition 6.18, Theorem 7.1]{DurhamHagenSisto} and \cite[Theorem 4.1]{DurhamHagenSistoCorrection}),
 and more (see below).
All of these are new for 
$\piorb(\cM_{cs})$.  



The  lattice $\Gamma$ is allowed to be non-uniform,  in which case the theorem 
stands in sharp contrast with the fact that non-uniform lattices $\Gamma$ in $PU(n>1,1)$ are never HHGs.
This is the (known)
$\H=\emptyset$ case of \cref{R:NotHHG}:
if \ref{LocFin}, \ref{Orth} and \ref{Inv} hold  but \ref{CuspLimit} does not, 
then  $\Gammahat$ cannot be an HHG.  
But it is a 
relatively hierarchically hyperbolic group (RHHG) by
\cref{T:mainrel}.

\bold{Other examples.} A remarkable number of spaces have the form $\Gamma\back (B-\H)$; 
see for example
\cite{LooijengaSurvey} and \cite{DolgachevKondoModuli}.
\cref{A:OtherEx} lists other examples to which \cref{T:main} or \cref{T:mainrel} applies, including:
\begin{itemize}
\item the moduli space of stable cubic threefolds; 
\item certain moduli spaces of K3 surfaces with non-symplectic automorphisms; 
\item certain moduli spaces of abelian varieties with extra structure, which are open subsets of certain unitary Shimura varieties of PEL type; and 
\item the moduli spaces $\cM_{0,5}/S_5$ and $\cM_{0,6}/S_6$ of 5 resp.\ 6 points in $\bC\bP^1$.
\end{itemize}

\subsection{Hierarchical hyperbolicity.} 

We pause now to introduce hierarchical hyperbolicity for those not familiar with it. 

\bold{What is a hierarchically hyperbolic group?} A good part of the study of mapping class groups has been built on deep breakthrough results of Masur-Minsky \cite{MMI, MMII} relating to the geometry of curve complexes. These complexes encode how the strata of the Deligne-Mumford compactification $\ol{\cM}_{g,k}$ fit together, after appropriately lifting to the  universal cover of $\cM_{g,k}$.

After much work by many people  developing and applying the resulting theory, Behrstock-Hagen-Sisto found strong parallels elsewhere in  geometric group theory, and introduced the definition of a hierarchically hyperbolic group \cite{BehrstockHagenSistoHHSI,BehrstockHagenSistoHHSII}. By design, the restrictive axioms of hierarchical hyperbolicity guarantee  that many results on mapping class groups apply to all hierarchically hyperbolic groups. Fresh excitement and having the right category to work in have facilitated many further results, and many new examples have been found within geometric group theory and low dimensional topology.

This historical discussion gives us our first answer to the question: 
\begin{enumerate}
\item A hierarchically hyperbolic group is one that is geometrically similar enough
to a mapping class group that many results about mapping class groups carry over. 
\end{enumerate} 

\noindent
There are other equally high-level answers: 
\begin{enumerate} 
\item[(2)] A hierarchically hyperbolic group is one that admits a system of coordinates with highly structured redundancy, such that the coordinates are valued in (Gromov) hyperbolic spaces, and the different coordinates give information on the geometry of different parts of the group and how they fit together. 
\item[(3)] A hierarchically hyperbolic group is one that looks hyperbolic, except for  product regions whose structure and overlaps are similarly constrained. 
\end{enumerate}

\noindent
The precise definition of hierarchical hyperbolicity is intricate, but it is not truly required for our work. We state a sufficient
criterion for hierarchical hyperbolicity in \cref{S:Criteria}, deferring the definition of hierarchical hyperbolicity and the proof of the criterion to \cref{S:CritProof}. 

See \cite{WrightWhatIs} for a brief introduction to hierarchical hyperbolicity, \cite{SistoWhatIs} for a deeper introduction, and \cite{SistoNewTools} for a more advanced and up to date survey. 

\bold{Other consequences of hierarchical hyperbolicity.} 
We think of \cref{T:main} not as supplying a few isolated consequences but rather a coherent geometric theory.

To illustrate the far reaching nature of this theory, we briefly discuss the Farrell-Jones Conjecture. This conjecture, originating in \cite{FarrellJones}, concerns the algebraic K-theory and L-theory of group rings and has strong implications for the Borel and Novikov Conjectures; see for example the ICM proceedings \cite{Luck, Bartels}. The Farrell-Jones Conjecture was proven for $\pi_1(\cM_{g,k})$ in \cite{BartelsBestvina}, and for large classes of HHGs in \cite{DurhamMinskySistoCAT0}. In \cref{SS:FJ} we use \cite[Theorem C]{DurhamMinskySistoCAT0} to conclude the following.

\begin{theorem}\label{T:FJ}
The Farrell-Jones Conjecture holds for all $\Gammahat$ as in \cref{T:main}.
\end{theorem}

Other  results on hierarchically hyperbolic groups, too numerous to list here, are proven for example in
\cite{BehrstockHagenSistoHHSI,
BehrstockHagenSistoHHSII,
BehrstockHagenSistoAsDim,
DurhamHagenSisto,
AbbottBehrstockDurham,
BehrstockHagenSistoQuasiflats,
PetytSpriano,
RussellSprianoTran,
HaettelHodaPetyt,
DurhamMinskySisto,
Petyt,
DurhamMinskySistoCAT0,AzuelosHagen}. 
%
%
Some results on hierarchically hyperbolic groups require additional  assumptions, and we show our $\Gammahat$  satisfy the most common such assumptions in \cref{A:IndexSet} (note also \cref{R:ForABD,R:CPR}).

\subsection{The proof}

We now wish to discuss the proof of \cref{T:main} and the difficulties it must overcome.

\bold{The motivating example.} But let us step back for a moment and first discuss the motivating example  $\piorb(\cM_{cs})$. 
A naive comparison of  $\piorb(\cM_{cs})$ to $\piorb(\cM_{g,k})$ gives rise to the hope that the topological monodromy
$$\piorb(\cM_{cs}) \to \pi_0 (\Homeo(\Sigma_{cs}))$$ 
is an isomorphism. Here $\Sigma_{cs}$ is a cubic surface, and $\pi_0 (\Homeo(\Sigma_{cs}))$ is its topological mapping class group. However, $\piorb(\cM_{cs})$ in fact has \emph{finite} image in the topological mapping class group $\pi_0 (\Homeo(\Sigma_{cs}))$; see for example \cite[Theorem 2.2]{Shimada} or \cite[Section III.3]{Harris} together with   \cite[Theorem 1.1]{Quinn} and \cite{GabaiGayHartmanKrushkalPowell}. 
%
%
This means that any attempt to study $\piorb(\cM_{cs})$ via direct topological analogues of objects like curves on surfaces will  find little traction.

\bold{The key tool.}  We will make an analogy between $\cM_{g,k}$ and $\Gamma\back(B-\H)$, by thinking of
$\cM_{g,k}$ as the quotient of its  universal cover (Teichm\"uller space~$\cT_{g,k}$) by the mapping class group
$\pi_1(\cM_{g,k})$.
The Weil-Petersson metric on $\cM_{g,k}$ becomes analogous to the complex hyperbolic metric on $\Gamma\back (B-\cH)$.

The metric completion $\ol{\cT}_{g,k}$ of $\cT_{g,k}$ with respect to the lifted Weil-Petersson metric is known as augmented Teichm\"uller space or the Deligne-Mumford bordification. Its quotient by $\piorb(\cM_{g,k})$  is the Deligne-Mumford compactification $\ol{\cM}_{g,k}$ \cite{MasurDM}. The boundary $\ol{\cT}_{g,k}- \cT_{g,k}$ is a union of
(complex) codimension 1 components that are in bijection with (essential, non-peripheral) simple closed curves on a topological reference surface $\Sigma_{g,k}$. The curve complex of $\Sigma_{g,k}$ can be coarsely identified with a kind of quotient of $\ol{\cT}_{g,k}$ called an electrification \cite[Lemma 7.1]{MMI}.
%
%
Electrification collapses each of the boundary components to have finite diameter.

In the complex ball situation, our analogy suggests thinking of
the universal  cover of $\Gamma\back(B-\H)$ as an analogue of~$\cT_{g,k}$.  
This is the same as the universal cover 
of $B-\H$, except
with deck group~$\Gammahat$ not just its subgroup $\pi_1(B-\H)$.
In the motivating example, this
is the  universal cover of the moduli space of smooth complex cubic surfaces.  The metric completion $\Bhat$ of the universal cover, with respect to the lifted complex hyperbolic metric,
will be our key tool.  It is a CAT($-1$) space \cite{AllcockAsphericity}, got by adjoining to
the universal cover
an infinite family of (complex) codimension~1 components which meet each other ``the same way'' as the
boundary components of~$\ol{\cT}_{g,k}$.   (This uses the orthogonality hypothesis of
Theorem~\ref{T:main}.)

Although there is no known way to use the Weil-Petersson metric to prove that
$\piorb({\cM_{g,n}})$ is an HHG, the analogy between
$\Bhat$ and $\ol{\cT}_{g,k}$ provides motivation for our result that
$\Gammahat$
is an HHG. This analogy is strongest when  $\Gamma$ is
cocompact, but when $\Gamma$ has cusps the  
geometry of~$\Bhat$ near its cusps (suitably defined) has new features.

%
%
%
%

\bold{The hyperbolic spaces.} Since $\Bhat$  is CAT($-1$) it is of course (Gromov) hyperbolic, so it may surprise non-experts that $\Bhat$ cannot typically be used as one of the (Gromov) hyperbolic spaces required by the definition of hierarchical hyperbolicity. In fact we use four kinds of (Gromov) hyperbolic spaces: 

\begin{enumerate}
\item Electrifications of stratum closures in $\Bhat$, indexed by  ``stratum domains''. This includes an electrification of $\Bhat$ itself, at the top of our hierarchy. 
\item Spaces quasi-isometric to $\bR$, which record a sort of angular coordinate analogous to a Fenchel-Nielsen twist parameter. We index these by  ``angle domains'', and there is one for each codimension 1 stratum of $\Bhat$. 
\item More spaces quasi-isometric to $\bR$, which record the ``vertical'' direction as measured from each cusp, and are associated to the center of the Heisenberg group. We index these by ``center domains''. 
\item Spaces quasi-isometric to trees, which record the ``horizontal'' directions as measured from each cusp. We index these by  ``tree domains'', and in combination with center domains they witness how \ref{CuspLimit} transforms the cusp geometry of $\Gamma$ from an obstruction to hierarchical hyperbolicity into an entirely different and more compatible geometry. 
\end{enumerate}

\bold{Swaddling.}
Suppose $\hat{S}$ 
is the closure of a (complex) codimension~$1$ boundary stratum in $\Bhat$.  One can make sense of
its unit normal bundle, which is a principal $\bR$-bundle over a singular space. 
It has something like a connection, 
pulled back from the unit normal bundle of the hyperplane of~$B$ over which~$\hat{S}$ lies.  If this were flat,
then it would be easy to define an $\R$-valued function describing angular position around~$\hat{S}$.
But it is not.  So we introduce a construction called ``swaddling'', that starts with
a principal~$\R$-bundle and ``squeezes it horizontally''.
This produces either a quasiline or a finite-diameter space, and it 
is here that the ``angle function'' around~$\hat{S}$ takes values.  One might 
view this as a geometric alternative to quasi-morphism based ideas introduced in the HHG context in \cite{HagenRussellSistoSpriano} and used in 
\cite{FournierFacioMangioniSisto,HagenMartinSisto,DowdallDurhamLeiningerSisto}. 

We find another use for swaddling when examining the cusps of~$\Bhat$.  
Consider a horosphere~$\partial T_c$ centered at a cusp of~$\Gamma$; it is the boundary of a horoball $T_c$.
It is a copy of the Heisenberg group, and has the structure of a 
principal $\R$-bundle over~$\C^{n-1}$ with connection.  A lift $\partial \hat{T}_{\hat{c}}$ of  $\partial T_c$ to $\Bhat$ is a singular space but still a principal $\R$-bundle, which we
can swaddle after pulling back
the connection.  Although swaddling~$\partial T_c$ yields a finite-diameter space,
swaddling~$\partial \hat{T}_{\hat{c}}$ yields a quasiline and leads to our center domains.
In fact, $\partial \hat{T}_{\hat{c}}$ is very much like the product of one copy of~$\R$
and $n-1$ trees. Seeing the transition to this structure,
starting from Heisenberg geometry, was
a lovely moment.

To show that these swaddlings yield quasilines
rather than finite-diameter spaces, 
we  develop quantitative tools that allow us to show that a given set of swaddling
data is ``almost flat''.  This is similar to checking that a 2-cocycle is bounded.  In one case
we prove a local-to-global result that allows us to lift an easy boundedness result 
from $B$ to a more subtle one on $\Bhat$.  This result should be applicable
to many other swaddlings.



\bold{Architecture.}
We begin with background in
\cref{S:Background} and our (skippable until later) criterion
for hierarchical hyperbolicity in \cref{S:Criteria}.
\Cref{SecModelSpace} 
through \cref{S:AngleDomains} develop the
geometry of the four kinds of domains, including swaddling and
much more.    We work extensively
with  
geodesics in~$\Bhat$, because we lack tools like the curves on Riemann surfaces
used for the mapping class groups.  This
work is complicated by the fact that $\Bhat$ is not locally compact,
and the fact that geodesics can leave the ``thick part''
$\Bhatthick$, which is our  geometric model for the group $\Gammahat$.
We establish various properties of nearest-point projection maps,
the hyperbolicity of the stratum domains, and the
manner in which a horosphere in~$\Bhat$ is ``like'' a
product of~$\R$ and some trees. Unlike in some constructions of hierarchically hyperbolic groups, we cannot use tree or quasi-tree behavior in our analysis, since our stratum domains do not give quasi-trees.
In \cref{S:NTO} through \cref{S:Final} we must fit together the four very different types of domains
to build the HHG structure on~$\Gammahat$.  


\subsection{Previous work} 

Loosely speaking, $\Gamma\back \Bcirc$ might be considered a version of $\Gamma\back B$ with $\Gamma \back \cH$ drilled out, and thus our \cref{T:main} might be viewed as related  to work like \cite{GHMOSW} on drilling hyperbolic groups and manifolds. 

\cite[Theorem 1.1(ii)]{Belegradek} shows that if distinct hyperplanes of $\cH$ are disjoint, and if $\Gamma \back \cH$ is compact, then $\Gammahat$ is relatively hyperbolic. The proof uses a delicate warped product construction to build a complete metric on $\Gamma \back \Bcirc$, and the $\Gammahat$ studied are of intrinsic interest in relation to rigidity phenomena and geometric finiteness. This line of research was continued in \cite{BelegradekHruska}, and \cite[Theorem 1.4]{BelegradekHruska} gives in particular that $\Gammahat$ is relatively hyperbolic under a ``sparsity'' condition. See also \cite{BelegradekRealHyp,Minemyer} for related results. We note in \cref{R:RelHyp} that very strong assumptions are required for one of our $\Gammahat$ to be relatively hyperbolic, and in particular $\piorb(\cM_{cs})$ is not relatively hyperbolic; so the sparsity condition of \cite{BelegradekHruska} cannot hold for our key example. 

When $n>1$ and $\cH\neq \emptyset$, \cite[Theorem 1.5]{BelegradekHruska} shows that our $\Gammahat$ cannot be isomorphic to a discrete isometry group of a Hadamard manifold of pinched negative curvature, and \cite[Theorem 1.2]{AllcockCarlsonToledoPresentation} shows that it cannot be isomorphic to a lattice in a Lie group with finitely many connected components. \cite[Theorem 1.1]{AllcockCarlsonToledoPresentation} shows that the infinitely generated kernel of $\Gammahat\to \Gamma$
has a presentation with a certain form, and \cite{Looijenga} and 
\cite{AllcockBasakLooijenga} give simple finite presentations for $\pi_1(\cM_{cs})$. A recent result shows that the divisor subgroup of $\pi_1(\cM_{cs})$ is characteristic, and that $\cM_{cs}$ has no non-trivial holomorphic automorphisms \cite{BaldiFarbJavanpeykarStover}.

At a technical level our work can be compared to \cite{HagenMartinSisto}, which proves that Artin groups of large and hyperbolic type are HHGs, like us starting with CAT($-1$) geometry. 
Unlike ours, their CAT($-1$) spaces are metrized 2-dimensional simplicial complexes. 

\subsection{Questions and conjectures.} 

Our work might lead in many directions. 

\bold{Beyond hierarchical hyperbolicity.} For all that \cref{T:Cubic} goes a long way to moving the state of knowledge on $\piorb(\cM_{cs})$ towards catching up with $\piorb(\cM_{g,k})$, this only applies for questions of a certain flavor. An important and natural question that our work says nothing about is the following. 

\begin{question}
Is $\piorb(\cM_{cs})$ residually finite? 
\end{question}

%
%

There are also questions for which \cref{T:Cubic} may help a great deal, but which require additional work, such as automaticity and quasi-isometric rigidity. 

%
%
%
%
%

\bold{Beyond orthogonal arrangements.} \cite[Conjecture 1.4]{AllcockBranchedCovers} and \cite[Conjecture 1.1]{PanovPetrunin}, if true, would allow much of our analysis to proceed for naturally occurring examples where our orthogonality assumption \ref{Orth} does not hold. For example, the moduli space of smooth complex cubic threefolds
has the form $\Gamma \back (B-\cH)$ with $\cH$ non-orthogonal \cite{AllcockThreefolds,LooijengaSwierstra}. 

\begin{conjecture}
The fundamental group of the moduli space of smooth complex cubic threefolds is an HHG. 
\end{conjecture} 

We similarly conjecture that the $\Gammahat$ that appear in relation to Allcock's monstrous proposal \cite{AllcockMonstrous} and in relation to moduli of rational elliptic surfaces \cite{HeckmanLooijenga} are (R)HHGs.

%
%
%
%
%
%

\bold{Beyond ball quotients.} We next move on to the setting of type IV domains, where the complex ball is replaced by the  $SO(2,n)$ hermitian symmetric space. Remarkably, some moduli spaces, such as $\cM_{0,5}$ and $\cM_{cs}$, fit in this setup as well as the ball quotient setup (see \cref{R:M05both} and \cite{MatsumotoSasakiYoshida}), and such common examples are among the reasons to hope our techniques can be useful in the type IV setting.  
%
%
%

Here it is likely that, although some examples are (R)HHGs, many others are not. As is the case for outer automorphism groups of free groups, it may be that even examples that are not hierarchically hyperbolic can be understood using ideas from hierarchical hyperbolicity. 

Key examples of interest include the following; all are of the form $\Gamma \back (X-\cH)$ for $X$ a type IV domain and $\cH$ a (typically not orthogonal) hyperplane arrangement.

\begin{enumerate}
\item The moduli space of smooth complex Enriques surfaces, where $\cH$ is orthogonal; see \cite{HorikawaI, HorikawaII, Namikawa,AllcockPeriodLattice,DolgachevIntroToEnriques,AllcockAsphericity}. 
%
%
%
%
\item The moduli space of cubic fourfolds; see \cite{Voisin,LooijengaFourfolds, Laza}.
%
%
%
\item Moduli spaces of amply polarized K3 surfaces of degree $2d$, where  we expect some of the behavior to depend on $d$;  see \cite{PiatetskiShapiroShafarevich,Huybrechts,AlexeevEngel,GritsenkoHulekSankaran}.  And similarly in the more general
setting of lattice-polarized K3 surfaces \cite{DolgachevLatticePolarizedK3s}.
%
%
%
%
%
\item Certain spaces that arise in singularity theory, especially related to unimodal singularities; see  \cite[Section 8]{AllcockBranchedCovers} and \cite{LooijengaSimplyEllipticII, LooijengaTriangleII, LooijengaRational, LazaSingularities,Zhao}. 
\end{enumerate}

%
%
%

A great deal of work has been devoted to compactifications of many of the spaces we have discussed, for example 
\cite{LooijengaCompactifications,LooijengaCompactificationsII,HackingKeelTevelev,CasalainaMartinGrushevskyHulekLaza,Naruki,LooijengaSurvey,AlexeevEngelHan,LazaOGrady,LazaOGradyII,GallardoKerrSchaffler,FangSchafflerWu,AllcockFreitag,OdakaSpottiSun}. Just as the Deligne-Mumford compactification of $\cM_{g,k}$ is intimately related to the HHG structure on $\piorb(\cM_{g,k})$, some of these compactifications may be relevant to the geometric group theory of the relevant moduli spaces.

\bold{A symplectic mapping class group?}  
Besides its interest in algebraic geometry, our motivating group $\pi_1(\cM_{cs})$ 
arises at what might be a fertile threshold  in the theory of symplectic mapping class groups.
Using a Fubini-Study symplectic form $\omega_{FS}$ on a cubic surface $\Sigma_{cs}$, one gets a symplectic monodromy map
$$\piorb(\cM_{cs}) \to \pi_0 (\Symp(\Sigma_{cs},\omega_{FS})),$$ 
with the right side
called the symplectic mapping class group (of cubic surfaces). 
When cubic surfaces are replaced by more-complicated varieties, one does not expect 
the corresponding maps to be isomorphisms.  But there is a folklore conjecture: 

\begin{conjecture}\label{C:Folklore}
In the cubic surface case, the symplectic monodromy map 
$$\piorb(\cM_{cs}) \to \pi_0 (\Symp(\Sigma_{cs},\omega_{FS}))$$ is an isomorphism.
\end{conjecture} 
%
%

See
\cite{Seidel,EvansSpheres,LiLiWu,
BormanLiWu,Wu,EvansSMCG,Smith,Smirnov,SmirnovK3}
for related results and context. 
%
%
If true, \cref{C:Folklore}  would allow our results to apply to a symplectic mapping class group, thus providing an avenue to a Thurstonian theory of a symplectic mapping class group and the ``longstanding challenge'' of finding analogues of the Nielsen-Thurston classification in this context \cite[Section 3.2.2]{Smith}. 

%
%
%

%
%
%
%

\bold{Acknowledgments:} We warmly thank the following people for helpful conversations:  
 Jason Behrstock, Eliot Bongiovanni, Mark Hagen, Giorgio Mangioni, Jacob Russell, and Alessandro Sisto (hierarchical hyperbolicity); 
 Jonny Evans, Jun Li, and Paul Seidel (symplectic mapping class groups); 
 Michela Artebani, Igor Dolgachev, Shigeyuki Kond\=o, and Alessandra Sarti (moduli spaces);
  Matthew Stover (ball quotients); Carlos Serv\'an (mapping class groups of 4-manifolds); 
  Aaron Kim (singularity theory); and Mohith Nagaraju  (various topics). 
  AW gratefully acknowledges the support of NSF grant DMS-2142712.  Both authors gratefully
  acknowledge the support of the Simons Foundation. 

\bold{Tool use:} 
This paper contains no AI-generated text.  We used AI only late in the project,
after writing the initial draft and lecturing on our results:  
we used \hbox{OpenAI}'s ChatGPT 5.5 and 5.6 for proofreading, literature search and
expository suggestions.  It also provided the argument
in \cref{R:NotMCG}.

\section{Background on geometry }\label{S:Background}

Here we recall standard results in the  minimal level of generality we require. In addition to the many standard books, some of which we cite, there are also accessible course notes available online, such as \cite{SistoNotes, Parker, WrightNotes}, that may be helpful to anyone learning some of this material for the first time.


\subsection{Coarse geometry}

 Recall the following definitions. 

\begin{enumerate}
\item A geodesic space is a metric space with at least one geodesic joining every pair of points. A geodesic is an isometric embedding of a connected interval into the metric space. 
\item A geodesic triangle consists of three vertices and a geodesic joining each pair of vertices.
\item The triangle is called $\delta$-thin if every point on an edge has distance at most $\delta$ to a point on one of the other two edges.
\item A geodesic space is called $\delta$-hyperbolic if all its geodesic triangles are $\delta$-thin.
\item A geodesic space  is called Gromov hyperbolic if it is $\delta$-hyperbolic for some $\delta$.
\item A map $f : (X,d_X) \to (Y, d_Y)$ is called $K$-coarsely Lipschitz if, for all $x_1, x_2\in X$,
$$d_Y(f(x_1), f(x_2)) \leq K d_X(x_1, x_2)+K.$$
\item It is called a $K$-quasi-isometric embedding if additionally 
$$d_Y(f(x_1), f(x_2)) \geq d_X(x_1, x_2)/K-K.$$
\item A subset of a metric space is called $K$-dense if every point in the metric space has distance at most $K$ to a point in the subspace. 
\item A map is called a $K$-quasi-isometry if it is a $K$-quasi-isometric embedding and the image is $K$-dense. 
\item A map is coarsely Lipschitz or a quasi-isometric embedding or a quasi-isometry if there exists $K$ such that the appropriate condition holds.
\item A quasi-geodesic in a metric space is a quasi-isometric embedding of a closed connected subinterval of $\bR$.
\item A path $p:I\to X$ (where $I$ is a closed connected interval) in a metric space is called a $D$-unparametrized quasi-geodesic if there is a $D$-quasigeodesic  $q:J\to X$ and a monotone $D$-coarsely Lipschitz map $r:I\to J$ such that $p=q\circ r$. 
\end{enumerate}

The term ``unparametrized quasi-geodesic'', although quite standard, may seem somewhat misleading, as a parameterization  is required for the definition.

For any constants $K, \delta$, there is a constant $D$ such that for any two $K$-quasi-geodesics with the same endpoints in a $\delta$-hyperbolic space, every point of one  has distance at most $D$ to a point of the other \cite[page 401]{BridsonHaefliger}. Given two quasi-isometric geodesic metric spaces, one is Gromov hyperbolic if and only if the other is \cite[page 402]{BridsonHaefliger}. 

\begin{remark}\label{R:NotGeodesic}
Because of \cref{R:EnotG} below, we will not be able to entirely stay in the realm of geodesic metric spaces. This is an unimportant technical detail that can be safely ignored, but for completeness we now present one of many ways it can be dealt with. (For alternatives, see, for example, \cite{Loh,Vaisala}.) Given a metric space $(X,d)$ and a number $b> 0$, one can consider the graph $\Gamma_b(X)$ with one vertex $v_x$ for each point $x$ of $X$, and an edge between two vertices $v_x, v_y$ if $d(x,y)\leq b$. 
Under mild assumptions the map $x\mapsto v_x$ is a quasi-isometry from $X$ to $\Gamma_b(X)$ \cite[page 152]{BridsonHaefliger}. If a metric space  $(X,d)$ is not geodesic, we will say that it is Gromov hyperbolic if there is some $b$ such that $x\mapsto v_x$ is a quasi-isometry and  $\Gamma_b(X)$ is Gromov hyperbolic. 
\end{remark}

\subsection{Electrification}

We define the electrification $(X^!,d^!)$
of a metric space $(X,d)$, along a family of subsets 
$\mathcal{Y}$, as follows.
The notation ``!'' is 
to suggest something that shocks, and $X^!$ may be pronounced ``$X$-zap''.

As a set, $X^!$ is obtained from  $X$ by the following process: first, add one point ``$[Y]$'' for each $Y\in\mathcal{Y}$, and second, add one copy of the unit interval for each pair $(y,Y)$ with $y\in Y\in \cY$, with one endpoint identified with $y$ and the other with $[Y]$. We call $[Y]$ the generic point of $Y$. 

Define $d^!$ to be the largest function $X^!\times X^!\to\R$ that satisfies
the following conditions:
\begin{enumerate}
\item $d^!$ is symmetric and satisfies the triangle inequality;

\item  the restriction of $d^!$ to $X\times X$  is bounded above by 
 $d$; and

\item  the restriction of $d^!$ to each unit interval is bounded above by the usual metric on the unit interval. 
\end{enumerate}
One can check that $d^!$ is a metric,
making $(X^!,d^!)$ into a metric space.
Furthermore,  $d^!(y,[Y])=1$
whenever $y\in Y\in\mathcal{Y}$, and  $d^!(x,y)=d(x,y)$ whenever $x,y\in X$
and the right
side is at most $2$.
%
%


\begin{remark}\label{R:EnotG}
Even under favorable assumptions, it may be that $X^!$ is not a geodesic space.
%
%
However it is always true that, for any $X^!$ as above of a geodesic metric space $X$, for any $b>0$, the map $x\mapsto v_x$ of \cref{R:NotGeodesic}  is a quasi-isometry, and for our arguments there is little harm in pretending $X^!$ is geodesic. 
\end{remark}

The most important fact about electrifications we will use is the following.  For a proof, see \cite[Proposition 2.6]{KapovichRafi} and see the discussion before \cite[Proposition 2.6]{KapovichRafi} for history and context including \cite{Farb, Bowditch, MjReeves, DahmaniGuirardel, MaherSchleimer}. 

\begin{proposition}\label{P:ElectHyp}
For all $\delta\geq 0$ there exists a $\delta'\geq0$ and a $D\geq 0$ such that the electrification of a $\delta$-hyperbolic space along a collection of convex subsets is $\delta'$-hyperbolic. 

Moreover, any geodesic in the original space is a $D$-unparametrized quasi-geodesic in the electrification. 
\end{proposition}

A convex set is a set containing every geodesic segment with endpoints in the set; but note that the citations show convexity is  more than is truly required here.

\subsection{Comparison geometry}

Consider $\kappa\in \{0,-1\}$. Let $M_0$ be the Euclidean plane, and $M_{-1}$ be the hyperbolic plane. 

If $(X,d)$ is a geodesic metric space, we say that $(X,d)$ is CAT($\kappa$) if every geodesic triangle in $X$ is ``at least as thin'' as a triangle in $M_\kappa$ with the same edge lengths \cite[page 158]{BridsonHaefliger}. 

\begin{example}
A complete, simply connected Riemannian manifold of sectional curvature at most $\kappa$ is CAT($\kappa$) \cite[page 173, 193]{BridsonHaefliger}.
\end{example}

\begin{remark}
We will make frequent use of the following basic facts. 
\begin{enumerate}
\item Every CAT($-1$) space is also CAT(0) \cite[page 165]{BridsonHaefliger}. 
\item There is a unique geodesic joining every pair of points in a CAT(0) space \cite[page 160]{BridsonHaefliger}. 
\item If $(X,d)$ is CAT(0) and $C\subset X$ is convex and complete in the induced metric, then the closest point projection $$p_C: X\to C$$ is well defined and distance non-increasing \cite[page 176]{BridsonHaefliger}.
\item There exists a $\delta\geq 0$ such that every CAT($-1$) space is $\delta$-hyperbolic \cite[page 399]{BridsonHaefliger}.
\end{enumerate}
\end{remark}

We will also make use of the following, which generalizes the fact that if 3 lines in $\bR^n$ meet pairwise orthogonally then they have a point in common.

\begin{lemma}\label{L:TripleIntersection}
Suppose $A, B, C$ are non-empty complete  convex subsets of a CAT(0) space and the closest point projection of any one to any other is their intersection. Then $A\cap B\cap C\neq \emptyset$. 
\end{lemma}

\begin{proof}
$A\cap B\cap C$ is the projection of $A\cap B$ to $C$, as can be seen using $p_C(A\cap B) \subset p_C(A)\cap p_C(B)$. 
\end{proof}

We also have the following strong contraction result for projections. 
\begin{lemma}\label{L:CATBGI}
For all $\epsilon>0$ there exists a $\delta>0$ such that if
\begin{enumerate}
\item $C$ is a closed convex subset of a CAT($-1$) space $X$, and 
\item $x,y$ are points of $X$ such that the geodesic $[x,y]$ joining them is disjoint from the $\epsilon$-neighborhood of $C$, 
\end{enumerate}
then there is a path of length at most $\delta$ on the boundary of $C$ joining $p_C(x)$ and $p_C(y)$. In particular, $p_C([x,y])$ has diameter at most $\delta$. 
\end{lemma}

\begin{proof}
 The second claim is standard. To  see the first claim, it suffices to apply the fact that $p_C$ is distance non-increasing to a path obtained from the piecewise geodesic path $[p_C(x), x] \cup [x,y] \cup [y, p_C(y)]$ by replacing the (0, 1, or 2) portions outside the $\epsilon$-neighborhood of $C$ with geodesic segments.
%
%
%
%
%
\end{proof}

\subsection{Complex hyperbolic geometry}

Consider $\bC^{n+1}$ with the Hermitian form 
$$h(x,y) = -x_0 \overline{y}_0 + \sum_{i=1}^n x_i \overline{y}_i.$$
A vector $x$ is called negative if $h(x,x)<0$. 
Complex hyperbolic space $\bC\bH^n$ can be defined as the projectivization of the set 
of negative vectors. This is endowed with a natural K\"ahler metric called the complex hyperbolic metric, whose group of holomorphic isometries is $PU(n,1)$. Here $U(n,1)$ is the group of complex linear maps preserving $h$, and the linear action of $U(n,1)$ on $\bC^{n+1}$ induces the action of $PU(n,1)$ on complex hyperbolic space. The complex hyperbolic metric is complete and has sectional curvatures in $[-4,-1]$. 

Every negative vector $x$ can be scaled to have $x_0=1$, allowing $\bC\bH^n$ to be identified with the ball 
$$B=\left\{z\in \bC^n : \sum_{i=1}^n |z_i|^2<1\right\}.$$

A hyperplane in $\bC\bH^n$ is the image under a holomorphic isometry of the locus $\{[x]\in \bC\bH^n: x_n=0\}$. All hyperplanes are convex and isometric to $\bC\bH^{n-1}$. 

More generally, say $V$ is a subspace of $\bC^{n+1}$ where the form is positive definite. Then $V^\perp$ defines a convex subspace $\bC\bH(V^\perp)$ isometric to $\bC\bH^{n-\dim_\bC V}$. All hyperplanes arise in this way with $\dim_\bC V=1$. 

If $[z]$ is a point in $\bC\bH(V^\perp)$, then the projectivization of $\bC z \oplus V$ intersects $\bC\bH^n$ in a sub-ball $\bC\bH(\bC z \oplus V)$ that intersects $\bC\bH(V^\perp)$ orthogonally at the point $[z]$. There is a map from the ball to $\bC\bH(V^\perp)$ induced by the orthogonal linear projection map  $\bC^{n+1} = V \oplus V^\perp \to V^\perp$. In the above setting it maps $\bC\bH(\bC z \oplus V)$ to $[z]$. This map is the closest point projection onto $\bC\bH(V^\perp)$. See for example \cite[Section 3.1.5, Lemma 4.3.1]{Goldman}.
%

\subsection{Branched covers}\label{SS:completions}

We call a collection $\cH$ of hyperplanes an orthogonal arrangement if it is locally finite and if any two hyperplanes in $\cH$ are either orthogonal or disjoint.  We now recall the following result from \cite[Section 3]{AllcockAsphericity}. 

\begin{theorem}\label{T:BhatCatMinusOne}
Let $\cH$ be an orthogonal hyperplane arrangement in $B$, and let $\Bhat$ be the metric completion of the universal cover of $B-\cH$. Then $\Bhat$ is CAT($-1$). 
\end{theorem}

There is a natural map $\Bhat\to B$, and $\Bhat$ can be thought of as the universal branched cover of $B$ branched along $\cH$. \cref{T:BhatCatMinusOne} is so fundamental to our analysis that we will frequently use it without further comment. The following warning will become evident later in the paper, but we state it now for emphasis. 

\begin{warning}
$\Bhat$ is not locally compact. The $\Bhat$ geodesic joining two points in the universal cover of $B-\cH$ does not always stay in that universal cover. 
\end{warning}

\begin{remark}
Experts in Teichm\"uller theory may wish to loosely compare $\Bhat$ to augmented Teichm\"muller space (also known as the Deligne-Mumford bordification of Teichm\"muller space).
\end{remark}


\section{Criteria for hierarchical hyperbolicity }\label{S:Criteria}

In this section we give sufficient criteria for hierarchical hyperbolicity. These criteria have large overlap with the usual definitions and are  easy to prove, but are nonetheless convenient for our purposes. The proofs of the criteria are relegated to \cref{S:CritProof}.  Non-experts can consult \cite{WrightWhatIs, SistoWhatIs} for an introduction to hierarchical hyperbolicity if desired, but we do not assume  any prior knowledge.

Speaking colloquially, one might say that the skeleton of a hierarchically hyperbolic space  is a suitably reasonable metric space equipped with a suitably reasonable collection of maps to (Gromov) hyperbolic spaces. We isolate this core part of the definition of an HHS as follows.

\begin{definition}
We will say that a metric space is quasi-geodesic if there exists  $q$ such that every pair of points can be joined by a $q$-quasigeodesic. 
\end{definition}

\begin{definition}\label{D:preHHS}
Let $(\cX,d)$ be a quasi-geodesic space, and let $\IndexSmall$ be a set. Suppose  there exists $\delta, K\geq 0$ such that for each $U\in \IndexSmall$, we have a $\delta$-hyperbolic space $(\cC(U), d_U)$ and a $(K,K)$-coarsely Lipschitz  map 
$$\pi_U : \cX \to \cC(U).$$
 Then we say $(\cX, \IndexSmall, \pi_\bullet)$ is a {pre-HHS}. Elements of $\IndexSmall$ are called domains, and $\IndexSmall$ is called the index set. 
\end{definition}

We reserve the notation $\IndexBig$ for index sets known to have all of the structure required for hierarchical hyperbolicity, and use $\IndexSmall$ for index sets otherwise. The notation $\IndexSmall$ can be pronounced as ``Sigma flat''. 

From one point of view, the idea of hierarchical hyperbolicity is that  $\pi_\bullet$ should be a type of constrained coordinate system for $\cX$. The constraints on the coordinates appear as severe restrictions on the maps $$\pi_U\times \pi_V : \cX \to \cC(U)\times \cC(V).$$ These restrictions come in one of three forms, and accordingly any pair $(U,V)$ of domains relate to each other in one of three ways. We isolate this combinatorial part of the definition of an HHS as follows. 

\begin{definition}\label{D:Valid}
A set $\IndexBig$ is called a valid index set if it is endowed with a symmetric anti-reflexive orthogonality relation $\perp$ and a nesting partial order $\sqsubseteq$ with a unique maximal element, 
and the following hold:
\begin{enumerate}
\item\label[Iaxiom]{A:ChainsBounded} \textbf{Finite height:}  The size of chains for the nesting partial order is bounded.
\item\label[Iaxiom]{A:Coherent} \textbf{Coherence:}  If $V \sqsubseteq W$ and $U\perp W$ then $U \perp V$.
\item\label[Iaxiom]{A:Container} \textbf{Containers:} If $V \sqsubseteq W$ then the set $\cP_{V,W}$ of domains orthogonal to $V$ and nested in $W$ is either empty or there is a domain $U\sqsubsetneq W$ such that every domain in $\cP_{V,W}$ is nested in $U$. 
\end{enumerate}
If $U,V\in \IndexBig$ are not orthogonal and neither is nested in the other, we say they are asynchronous
and
write $U\pitchfork V$. (We use  ``asynchronous'' instead of the standard term ``transverse'', since the latter is potentially misleading in our examples.) 
\end{definition}

The fact that $\perp$ is anti-reflexive means that no domain is orthogonal to itself, and together with \eqref{A:Coherent} implies that a nested pair is never orthogonal.

The containers axiom indicates that, when $V$ is nested in $W$ and one restricts attention to domains nested in $W$, then the set of domains orthogonal to $V$ must be in a sense localized. We will find it useful to weaken this requirement slightly in the following way. Note that we say that a set (or tuple) of domains is orthogonal if every pair of distinct domains in the set (or tuple) is orthogonal. 

\begin{definition}\label{D:NearlyValid}
Suppose \cref{D:Valid} holds except with \eqref{A:Container}  replaced by 
\begin{enumerate}
\item[(3')]\label[Iaxiom]{A:WeakContainer} \textbf{Weak containers:} If $V \sqsubseteq W$ then the set $\cP_{V,W}$ of domains orthogonal to $V$ and nested in $W$ is either empty or there is an orthogonal set $\{U_i\}$ consisting of domains properly nested in $W$ 
 such that every domain in $\cP_{V,W}$ is nested in one of the $U_i$. Furthermore the size of orthogonal sets of domains is bounded.
\end{enumerate}
Then we say $\IndexSmall$ is a nearly valid index set. 
\end{definition}

\begin{example}\label{E:MCG}
If $\Sigma_g$ is a (2 real dimensional) closed surface of genus $g$,  a subsurface of $\Sigma_g$ is a closed submanifold $U$ with boundary such that no component of $U$ or $\Sigma_g-U$ is a disc. If  $\IndexSmall$ is the set of connected subsurfaces up to isotopy, then $\IndexSmall$ is a nearly valid index set. Here $U$ is orthogonal to  $V$ if $U\neq V$ and they can be isotoped to be disjoint; and $U$ is nested in $V$ if they are not orthogonal and $U$ can be isotoped to be a subset of $V$. If $W=\Sigma_g$, then the  orthogonal set $\{U_i\}$ required by the weak containers axiom is the set of components of $\Sigma_g-V$ together with (if $V$ is not annular) an annular subsurface for each boundary component of $V$. 
\end{example}

In what follows, we call $\cA_U$ the active region for $U$; one should imagine that $\pi_U$ is difficult to change outside of $\cA_U$.

\begin{proposition}\label{P:Criterion}
Let $(\cX, \IndexSmall, \pi_\bullet)$  be a pre-HHS with a nearly valid index set. Suppose we fix $B\geq 0$ and a subset $\cA_U\subset  \cX$  for each $U\in \IndexSmall$,  and  suppose the following hold. 
\begin{enumerate}
\item\label[Aaxiom]{A:U} \textbf{Enough projections:} For each $\kappa\geq 0$ there exists $\omega\geq0$ such that if $d(x,y)\geq \omega$ then  $d_{U}(\pi_U(x),\pi_U(y))\geq \kappa$ for some $U$.
\item\label[Aaxiom]{A:PR} \textbf{Density:} For all orthogonal tuples $T=(U_1, \ldots, U_k)$ of   domains,   
$$\pi_{U_1} \times \cdots \times \pi_{U_k}(\cA_{U_1}\cap \cdots \cap \cA_{U_k})$$
is $B$-dense in $\cC(U_1)\times \cdots \times \cC(U_k)$.
\item\label[Aaxiom]{A:BP} \textbf{Bounded projections:} If $U\pitchfork V$ or $U\sqsubsetneq V$, then $$\diam \pi_V(\cA_U)\leq B.$$
Furthermore, if $U\sqsubsetneq V$ then  $\cA_U \cap \cA_V\neq\emptyset$.
\item\label[Aaxiom]{A:SR} \textbf{Structured redundancy:}  If $U\pitchfork V$, for any $x$ we have at least one of
$$d_U(\pi_U(x), \pi_U(\cA_V)) \leq B \quad\quad\text{or}\quad\quad d_V(\pi_V(x), \pi_V(\cA_U)) \leq B.$$
 If $U\sqsubsetneq V$, for any $x,y$, if there is a geodesic from $\pi_V(x)$ to $ \pi_V(y)$ that has distance at least $B$ from $\pi_V(\cA_U)$, then $d_U(\pi_U(x), \pi_U(y))\leq B$.
\item\label[Aaxiom]{A:LL} \textbf{Progress localization:} For any $P\geq 0$ there exists  $N\geq 0$ and an $R\geq0$ such that for any $V$ and any $x,y$  with $$d_V(\pi_V(x),\pi_V(y))\leq P$$  there is a set of at most $N$ domains $U_i\sqsubsetneq V$ such that if $T\sqsubsetneq V$ is not nested in or orthogonal to one of the $U_i$ then $d_T(\pi_T(x),\pi_T(y))\leq R$. 
\end{enumerate}
Then $\cX$ is a hierarchically hyperbolic space. 
\end{proposition}

More precisely, the proof enlarges $\IndexSmall$ to a larger index set $\IndexBig$ and shows that $(\cX, \IndexBig, \pi_\bullet)$ is an HHS. The domains of $\IndexBig- \IndexSmall$ can be regarded as ``dummy domains''; their coordinate functions are constant and their addition can be viewed as a minor technical point. In this case the dummy domains are exactly the orthogonal subsets of $\IndexSmall$ of size greater than 1, and in the analogy to \cref{E:MCG} they should be compared to disconnected surfaces.

From the coordinate system point of view, speaking very roughly, \eqref{A:U} indicates there are enough coordinates to distinguish points, \eqref{A:PR} indicates orthogonal coordinates are independent, \eqref{A:BP} indicates that $\cA_U$ is collapsed to a point in $\cC(V)$,  \eqref{A:SR}   indicates  that the $\pi_V$ coordinate is largely redundant off the corresponding active region $\cA_V$, and  \eqref{A:LL} indicates a sort of local finiteness for the coordinate system.  

Even when an assumption of \cref{P:Criterion} corresponds to an axiom of hierarchical hyperbolicity, all of which have standard names, we have chosen a different name to more clearly indicate when we are referring to an assumption and when we are referring to an axiom.

We also have a variant for groups. Recall that a group action on a metric space is called geometric if the group acts by isometries and the action is properly discontinuous and cocompact. 

\begin{proposition}\label{P:CriterionGroups}
Suppose $(\cX, \IndexSmall, \pi_\bullet)$ satisfies the assumptions of \cref{P:Criterion}, $\cX$ is proper, and $G$ is a group that acts on $\cX$ geometrically. Suppose $G$ also acts on the index set $\IndexSmall$, via a map $g\mapsto g^{\diamond}\in \Aut(\IndexSmall)$,  preserving the orthogonality and nesting relations, and there are finitely many orbits of orthogonal tuples. 
Suppose that  for each $U\in \IndexSmall, g\in G$ there is an isometry $g(U): \cC(U) \to \cC(g^\diamond U)$, and
$$g \cA_U = \cA_{g^\diamond U} \quad\quad\text{and}\quad\quad  g(U)\circ \pi_U = \pi_{g^\diamond U}  \circ g   .$$ 
Finally, suppose that for each $U\in \IndexSmall$ and $g,h\in G$ we have $(hg)(U)=h(g^\diamond U)\circ g(U).$ 
Then $G$ is a hierarchically hyperbolic group. 
\end{proposition}

\section{The model space and statement of the goal }
\label{SecModelSpace}

In this section we will define a proper metric space $\Bhatthick$ on which $\Gammahat$ acts properly and cocompactly. It will serve as a sort of geometric model for $\Gammahat$. Our eventual goal in this paper is to apply \cref{P:CriterionGroups} with $\cX=\Bhatthick$ and $G=\Gammahat$.

Let $s_\theta \gg_{n,\cH, \Gamma} 1$ be a very large constant depending on $n, \cH, \Gamma$; see \cref{R:constants} for details. Having fixed $s_\theta$, 
for each cusp $c\in\partial B$ of $\G$ 
choose a closed horoball $T_c$, centered at $c$ and small
enough that every hyperplane coming within  $s_\theta$  of $T_c$ passes through $c$ (see for example \cite[page 21]{BelegradekHruska}). 
The notation
``$T$'' is for consistency with the next paragraph, where it indicates
``tubular''.
The $T_c$ are the only horoballs of $B$ that we will consider,
so a horoball resp.\ horosphere (of $B$) will always
mean $T_c$ resp.\ $\partial T_c$ for some cusp~$c$.
We may suppose without loss that 
the family of horoballs obtained this way is $\G$-invariant,
and every pair of them have distance at least $s_\theta$ from each other.

\begin{remark}
For what comes, it is helpful to note that, since no hyperplane can be contained entirely in a horoball, and since the arrangement is locally finite,  there are only finitely many $\Gamma$-orbits of hyperplanes. 
\end{remark}

We may
choose a very small constant $0<r_\theta\ll_{n,\cH, \Gamma, s_\theta} 1$ depending on $n,\cH, \Gamma, s_\theta$. We define
$T_H$ to be the closed $r_\theta$-neighborhood of any hyperplane $H$ of~$B$. (Recall that we have specialized ``hyperplane of~$B$'' to
mean a member of $\H$.)

\begin{remark}\label{R:constants}
We require the following of our choices of $s_\theta$ and $r_\theta$. 
\begin{enumerate}
\item It is convenient to assume, especially for \cref{L:GettingToStratum}, that any hyperplane that comes within some large distance of one of the horoballs $T_c$ actually enters the $T_c$. A weaker assumption is used in \cref{L:UhatcProjection}. We also assume that the projection of any horoball to any other horoball has diameter at most 1, for \cref{L:GettingToStratum}, and that the projection of a horoball to any stratum closure not entering the horoball has diameter at most 1, again for  \cref{L:GettingToStratum}.
\item\label{ConstantsA1} \cref{S:StratumDomains} will show that each point of $B$ not in one of the $T_c$ has a small neighborhood in which $\cH$ with certain properties. The size of this neighborhood, call it $\xi$, depends on $n, \cH, s_\theta$. We require, especially for \cref{SS:Uniformity}, that $r_\theta$ times a constant depending on $n$ is  smaller than $\xi$. A weaker assumption is used in \cref{L:EasyCaseOfAngleLipschitz,L:NaiveProjectionReasonableNearHhat}. 
\item It would be nice to choose the $T_H$ thin enough to miss each
other except near orthogonal intersections.  This is
not possible, because two hyperplanes may 
approach each other at a cusp. However we will assume $r_\theta$ is small enough that if $x$ and $y$ are on hyperplanes that don't intersect, and neither $x$ nor $y$ is on a horoball, then the distance from $x$ to $y$ is at least $1000 n r_\theta$. 
\end{enumerate}
In an attempt to minimize the number of different constants used in the paper, we sometimes pick particular new constants as a function of $n$ and $r_\theta$, even when such a precise choice is not required. 
\end{remark}

 Note that we say two hyperplanes are parallel if they are disjoint but share a point at infinity. Our assumptions imply in particular the following. 
\begin{enumerate}
\item If $H,H'$ are parallel, then 
$T_H\cap T_{H'}\sset T_c$, where the cusp $c\in\partial B^n$
is the common limit
point of $H$ and $H'$ in~$\partial B^n$.
\item  If $H,H'$ are neither
parallel nor orthogonal, then
$T_H\cap T_{H'}=\emptyset$.
\end{enumerate}

%
%

Define $\Bthick$ (``$B$thick'')
as the closure in $B$ of the
complement of the union of
all $T_c$ and~$T_H$.

Define $\Bhatthick$ as its universal cover, 
which is also the universal cover of $\G\back\Bthick$. Equip $\Bhatthick$ with its lifted intrinsic path metric, and note that since $\G\back\Bthick$ is compact we have that $\Bhatthick$ is a proper metric space. 

%
We will show that $\Bhatthick$ is a Hierarchically Hyperbolic Space (HHS) using our criteria for hierarchical hyperbolicity with $\cX=\Bhatthick$. Recall that these criteria require a set $\IndexSmall$ of domains, and for each domain~$U$ 
a Gromov hyperbolic space $\cC(U)$ and a coarsely Lipschitz map 
$\pi_U : \Bhatthick \to \cC(U)$.

We will use four types of domains:
``stratum domains'', ``angle domains'', ``center domains''
and ``tree domains''.
These domains and their associated
spaces and projections are defined below.

\begin{remark}\label{R:extension}
Each of our $\pi_U: \Bhatthick\to \cC(U)$ will be defined on a larger subset $\Bhat^U$, with $$\Bhatthick \subset \Bhat^U \subset \Bhat,$$
giving a map $$\pi_U : \Bhat^U \to \cC(U)$$ whose restriction to $\Bhatthick$ is $\pi_U: \Bhatthick\to \cC(U)$. The metric on $\Bhat^U$ will always be the intrinsic path metric.  Details will be provided in \cref{R:extensionstrata,R:extensiontree,R:extensioncenter,R:extensionangle}. 
In every result about one of the $\pi_U$ we will indicate which version we mean, but the distinction is largely a technical detail. 
\end{remark}

\section{Stratum domains and the branched cover }\label{S:StratumDomains}

In this section we describe the ``stratum domains'' required
for our HHS structure on $\Bhatthick$.
The main  ingredients in the definition are the universal branched
cover $\Bhat$ of $B$ over~$\H$ and various of its subspaces.  
These are also essential for
other parts of the HHS structure.

We will consistently use hats to indicate objects
associated to the universal cover of $\Bcirc=B-\H$.
For example, we define $\Bhatcirc$ as this universal
cover,
$\Bhat$ as its metric completion, and $\HHAT\sset\Bhat$ as the preimage of~$\H$.   
The deck group of $\Bhatcirc\to\Gamma\backslash\Bcirc$
has already been named  $\Gammahat$.
It is easy to see that $\Bcirc$ deformation retracts to
$\Bthick$, so we may regard $\Bhatthick$ as the preimage
of $\Bthick$ in $\Bhatcirc\subseteq\Bhat$.  The three spaces
$\Bhatthick\sset\Bhatcirc\sset\Bhat$ lie over the
three spaces $\Bthick\sset \Bcirc\sset B.$ We think of the natural map 
$b:\Bhat\to B$ as  a branched
cover.    

\subsection{Local models}\label{SS:Local}
To give a detailed local description,
we begin by constructing a special chart  around any
fixed $x\in B$.
Suppose 
exactly $k$ hyperplanes of~$B$ pass through~$x$.
Their intersection is a $B^{n-k}$. The intersection of
any $k-1$ of them is a $B^{n-k+1}$, and we consider the
$B^1$ therein that passes through~$x$ and is orthogonal to
$B^{n-k}$.  We complete these $B^1$'s to a ``set of axes''
by choosing $n-k$ additional $B^1$'s inside $B^{n-k}$,
that are mutually orthogonal and pass through~$x$.  
We now have $n$ many $B^1$'s. The closest point projection to each $B^1$, being the projectivization of a linear map, is holomorphic. 
Nearest-point projection to the product of the $B^1$'s identifies $B$
with an open subset of $(B^1)^n$. (If one takes the $B^1$'s to be the standard coordinate directions through the origin of the ball, then this is the standard inclusion of the ball in $\bC^n$ into the product of the $n$ balls in $\bC^1$.) 
We order the factors so that the first $k$ of them
are the first~$k$ that we specified.  
We index the hyperplanes through~$x$ by $\set{i}{1\leq i\leq k}$,
such that the $i$th one is
the preimage of 
$(B^1)^{i-1}\times\{x\}\times(B^1)^{n-i}$.

In each $B^1$ we choose a closed disk centered at~$x$, 
and define~$V\sset B$ as the preimage of their product.
By choosing the disks
small enough, we may suppose that $V$ projects surjectively 
(hence biholomorphically)
to their product.  We may also
suppose that $V$ is disjoint from $g V$ when $g\in \Gamma - \Stab_\Gamma(x)$,
and no hyperplanes intersect $V$
except those containing~$x$.   Next we identify each disk
biholomorphically with the closed unit disk $D\sset\C$,
such that $x$ corresponds to the origin.  Equipping
$D$ with its standard Euclidean metric, 
and $D^n$ with the product metric, we then have
\begin{enumerate}
    \item
        $V$ is a closed neighborhood of~$x$, identified
        with $D^n$ via a map which is biholomorphic 
        and bi-Lipschitz.
    \item
        The point $x$ corresponds to $(0,\dots,0)$.
    \item
        Whenever $1\leq i\leq k$, the intersection
        of $V$ with the $i$th hyperplane through~$x$ 
        corresponds to 
        $\{(y_1,\dots,y_{i-1},0,y_{i+1},\dots,y_n)\}\sset D^n$.
    \item
        $V\cap\H$ corresponds to the union of these
        $k$ many copies of $D^{n-1}$ in~$D^n$.
\end{enumerate}

\begin{remark}\label{R:localgeodesics}
These coordinates have many properties beyond those guaranteed by their bi-Lipschitz and biholomorphic properties. For example, their definition via nearest-point projection implies that the segments in $V$ which are constant in all but one coordinate and are a radial segment in the remaining coordinate are geodesics. These geodesics give the shortest paths to the relevant hyperplane. 
\end{remark}

%
%

Our next step is to 
define $\Vcirc$ as $V-\H$ and work out its preimage in~$\Bhatcirc$.
Writing 
$\Dcirc$ for $D-\{0\}$, $\Vcirc$ corresponds to 
$$(\Dcirc)^k\times D^{n-k}\sset D^n.$$
So $\pi_1(\Vcirc)\iso\Z^k$.
It is easy
to see that $\pi_1(\Vcirc)\to\pi_1(\Bcirc)$ is injective, by using winding number.
Therefore the preimage of $\Vcirc$ in $\Bhatcirc$ is a
disjoint union of copies of the universal cover~$\Vhatcirc$ of~$\Vcirc$.
The $\pi_1(\Bcirc)$-stabilizers of these copies are the conjugates
of  $\pi_1(\Vcirc)$.  Because $\Bhat$ was defined
as the metric completion of $\Bhatcirc$, the inclusion $\Vhatcirc\to\Bhatcirc$
extends to a natural map from the metric 
completion~$\Vhat$ of $\Vhatcirc$ to~$\Bhat$.  
It is easy to see
that this is injective, and that the closures of distinct components
of $b^{-1}(\Vcirc)$
are disjoint.  Therefore $b^{-1}(V)$ is a disjoint union of
copies of $\Vhat$.
(We work with closed neighborhoods because using
open neighborhoods would complicate 
this statement.)

This reduces the study of the branching above~$x$ to understanding~$\Vhat$ and
its projection to~$V$.  For this we use the bi-Lipschitz property of our
identification between $V$ and~$D^n$.
Because the metric on $\Vhatcirc$ is defined by taking infima of
lengths of paths, the bi-Lipschitz identification $\Vcirc\iso(\Dcirc)^k\times D^{n-k}$
lifts to a bi-Lipschitz identification of $\Vhatcirc$ with
$(\Dhatcirc)^k\times D^{n-k}$.  Here 
$\Dhatcirc\sset\C$ is the universal cover of~$\Dcirc$, namely the closed
left half-plane, with $\Dhatcirc\to\Dcirc$ being the exponential
map. 
Here
$\Dhatcirc$ does not have the Euclidean metric, but rather its
pullback under this map.
The deck group $\Z^k$ consists of the translations
 by integral multiples of~$2\pi i$ in the $\Dhatcirc$ factors.
The bi-Lipschitz identification $\Vhatcirc\iso(\Dhatcirc)^k\times D^{n-k}$
extends to an identification of their metric completions.  
Therefore
$\Vhat$ is bi-Lipschitz to $\Dhat^k\times D^{n-k}$, where $\Dhat$ is the metric completion
of~$\Dhatcirc$.  
It is easy to see that
$\Dhat=\Dhat^\circ\cup\{-\infty\}$, where 
the distance between $-\infty$ and any given $z\in\Dhatcirc$
is $\exp(\Re(z))$.
The map $\Dhat\to D$
is still the exponential map, with the convention $\exp(-\infty)=0$.

This completes our description of $b^{-1}(V)\to V$, namely:
$V$ is bi-Lipschitz to~$D^n$, with $x$ identified with~$(0,\dots,0)$,
such that each component of
$b^{-1}(V)$ is identified with $\Dhat^k\times D^{n-k}$.  
The projection to~$D^n$ applies the exponential map to all
$\Dhat$ factors.  Each component of $b^{-1}(V)$
contains exactly one preimage
of~$x$, corresponding to  
$(-\infty,\dots,-\infty,0,\dots,0)$ with $k$ many~$-\infty$'s.
The $\pi_1(\Bcirc)$-stabilizers
of the preimages of~$x$ are the conjugates of
$\pi_1(\Vcirc)\iso\Z^k$.

\subsection{Uniformity}\label{SS:Uniformity}
The neighborhoods $V$ cannot be chosen of uniformly large size even if $x$ is assumed to not lie in a horoball, because $x$ can lie arbitrarily close to a hyperplane not containing it. 

For each $\Vhat$ centered at a point $x\in \Bhat$, let $\Vhat_0\subset \Vhat$ be a small ball centered at $x$, so that the ball centered at $x$ with radius 100 times bigger is contained in $\Vhat$. Assume additionally that the $\Vhat$ chosen in this way are small enough that the implicit bi-Lipschitz constant in the statement ``$\Vhat$ is bi-Lipschitz to $\Dhat^k\times D^{n-k}$'' is at most 2. This can be achieved by noting that as the size of the neighborhood $V$ decreases, the implicit bi-Lipschitz constant in the statement ``$V$ is bi-Lipschitz to $D^n$'' goes to 1. 

A compactness argument shows that there is a finite collection of $\Gamma$ orbits of $V$ as above such that
each $y$ not in a horoball is contained in a $\Vhat_0$ of some $\Vhat$ covering an element of one of those finitely many $\Gamma$ orbits of $V$. Let $\xi/100$ be the min of the radii of the $\Vhat_0$ in this construction. This gives the constant $\xi$ in \cref{R:constants}. 

\begin{remark}\label{R:NudgeToThick}
Suppose $x\in \Bhat$ is not in any horoball. Keeping in mind that \cref{R:constants} assumes $r_\theta \ll \xi$, we can see that there is a $x'\in \Bhatthick$ with $d(x,x')\leq 8 \sqrt{n} r_\theta$. To see this, first find one of the chosen $\Vhat_0$ containing $x$. The associated $\Vhat$ is identified with $\Dhat^k\times D^{n-k}$, and proceed as follows.

First, if $x$ is very close to a horoball, push it radially outwards away from the horoball to get a point $x''$ with $d(x,x'') \leq 4 \sqrt{n} r_\theta$ and so that $x''$ has distance at least $4 \sqrt{n} r_\theta$ from every horoball. (Compare to the proof of \cref{L:UhatcProjection} for details.)

Then we find $x'$ by modifying the first $k$ coordinates of $x''$, if required, so that each has distance in $\Dhat$ more than $2 r_\theta$ from the center of $\Dhat$. Using that the bi-Lipschitz constant is at most 2, this can produce a point $x'$ with $d(x', x'') \leq 4 \sqrt{n} r_\theta$ such that $x'$ is not within $r_\theta$ of any hyperplane intersecting $\Vhat$. The triangle inequality gives that $x'$ is not in the interior of any horoball, so $x'\in \Bhatthick$.
\end{remark}

\subsection{Stratification}\label{SS:Stratification}

The intersections of hyperplanes in~$B$ define a stratification
of~$B$.  Namely, the codimension-$k$ strata are the connected
components of the set of points in~$B$ that lie in exactly
$k$ hyperplanes. 
This
induces a  stratification of $\Bhat$: the strata are
the
connected components of the preimages of strata of $B$.  
In keeping with the use of~$^\circ$ to suggest ``with hyperplanes removed'',
we will usually use notation like $\Scirc$ resp.\ $\Shatcirc$ for strata in~$B$ 
resp.\ $\Bhat$, and $S$ resp.\ $\Shat$ for their closures.
Our local model makes clear that the restriction of $b$ to
each stratum of~$\Bhat$ is a covering map to a stratum of~$B$.
It also shows that the closure~$\Shat$
of a codimension-$k$ stratum is a component of 
the fixed point set of a subgroup
$\Z^k\sset\Gammahat$.  Since 
$\Bhat$ is CAT($-1$), fixed-point sets are convex, so $\Shat$
is the full fixed-point set of this~$\Z^k$.
Therefore it is convex and CAT($-1$).

Just as we reserved the word hyperplane (of~$B$) to mean a component
of~$\H$, when we speak of hyperplanes (of~$\Bhat$) we mean 
the closures of codimension 1 strata of $\Bhat$.  These
lie over the hyperplanes of~$B$.
 
\subsection{Horoballs}\label{SS:horoballs}

We also need the notion of a horoball of~$\Bhat$.
Recalling that we reserved the word horoball (of~$B$) to
mean $T_c$ for some cusp $c$ of~$\Gamma$, 
a horoball of~$\Bhat$ will mean a 
component of the preimage
of a horoball of~$B$.  
A horosphere of $\Bhat$ will mean the boundary
of such a horoball.
Above each $T_c$ lie infinitely many horoballs of $\Bhat$;
when we have some particular one in mind, 
we will usually
denote it $\That_{\hat{c}}$, thinking of $\hat{c}$ as a lift of the cusp $c$.
See \cref{S:CuspDomains} for a detailed analysis
of these horoballs; for now we only need their convexity, which we will prove presently in \cref{L:ConvexityOfThatc}.

\begin{lemma}[Nearest-point projection near a horoball of~$\Bhat$]\label{L:UhatcProjection}
    Suppose $\That_{\hat{c}}$ is a horoball of $\Bhat$
    and $\Uhat_{\hat{c}}$ is the set of points of $\Bhat$ at distance${}<1$
    from it.  Then  the nearest-point projection map
    $p_{\That_{\hat{c}}}:\Uhat_{\hat{c}}\to\That_{\hat{c}}$ is well-defined and enjoys the 
            following property.  For each $r>0$ there is a constant
            $k<1$, such that if $\deltahat$ is a path in $\Uhat_{\hat{c}}$
            that stays at least $r$ away from $\That_{\hat{c}}$, 
            then
            $\length(p_{\That_{\hat{c}}}\circ\deltahat)\leq k\length(\deltahat)$.
\end{lemma}

\begin{proof}
    Write $T_c$ for the horoball of $B$ over which $\That_{\hat{c}}$
    lies, $U_c$ for the set of points of $B$ at distance${}<1$
    from it, and $p_{T_c}:U_c\to T_c$ for the nearest-point
    projection map.  Suppose $\xhat\in\Uhat_{\hat{c}}$ and set $x=b(\xhat)\in U_c$. 
    For our first claim, it is enough to show
    that there is a unique path from $\xhat$ into~$\That_{\hat{c}}$
    with length${}\leq d(x,T_c)$.  Existence will be a byproduct
    of the uniqueness argument.  For any such path $\gammahat$,
    $b\circ\gammahat$ must be a path from $x$ to~$T_c$
    with length${}\leq d(x,T_c)$.  There is only one:
    the geodesic 
    $\gamma=\overline{x\,p_{T_c}(x)}$.  
    By the definition of $T_c$ and \cref{R:constants}, every hyperplane of $B$ that meets $U_c$
    passes through~$c$.  Therefore $\gamma$ lies in the 
    same stratum~$\Scirc$ of~$B$ as~$x$.  So $\gammahat$ must lie in the
    same stratum~$\Shatcirc$ of~$\Bhat$ as~$\xhat$, which necessarily
    lies over~$\Scirc$.  Since $\Shatcirc\to\Scirc$ is a covering 
    space, the lift of
    $\gamma$ starting at~$\xhat$ exists and is unique.  This is the
    unique possibility for~$\gammahat$, proving the well-definedness of
    $p_{\That_{\hat{c}}}:\Uhat_{\hat{c}}\to\That_{\hat{c}}$.

    For our second claim, we start with the corresponding ``unhatted''
    statement, which holds in complex hyperbolic geometry just
    as in real hyperbolic geometry.
    Namely, for every $r>0$ there is a constant $k<1$, such that
    if $\delta$ is a path in $U_c$ that stays at least $r$ away
    from~$T_c$, then $\length(p_{T_c}\circ\delta)\leq k\length(\delta)$.
    From this follows the claim in the special case that
    $\deltahat$ lies in a metric ball that misses~$\HHAT$.  
    Concatenating such paths treats the case
    that $\deltahat$ misses $\HHAT$.  Then taking limits 
    treats the full claim, since every path $\deltahat$ in $\Uhat_{\hat{c}}$ is
    a limit of paths in $\Uhat_{\hat{c}}-\HHAT$ 
    whose lengths approach $\length(\deltahat)$.
\end{proof}

\begin{lemma}[Convexity of horoballs of~$\Bhat$]\label{L:ConvexityOfThatc}
Every horoball of $\Bhat$ is convex.
\end{lemma}

\begin{proof}
    Suppose some horoball $\That_{\hat{c}}$ were nonconvex, so
some geodesic $\gammahat$ in $\Bhat$ has endpoints in 
$\That_{\hat{c}}$ but does not lie entirely within $\That_{\hat{c}}$.
Recall that geodesics in a  CAT($-1$) space depend
continuously on their endpoints \cite[page 160]{BridsonHaefliger}.  
Moving one endpoint toward the other, inside $\That_{\hat{c}}$, homotopes
    $\gammahat$ to a constant geodesic.  So some intermediate geodesic
    exits  $\That_{\hat{c}}$ and then returns, but also remains at distance less than 1 from it.
    Using it in place of $\gammahat$, we may suppose without
    loss that $\gammahat$ has these properties.

    Now we apply \cref{L:UhatcProjection}: $\gammahat$ lies in $\Uhat_{\hat{c}}$,
    and the fact that it leaves $\That_{\hat{c}}$ implies 
    $\length(p_{\That_{\hat{c}}}\circ\gammahat)<\length(\gammahat)$.
    Since $p_{\That_{\hat{c}}}\circ\gammahat$ has the same endpoints as $\gammahat$,
    the latter cannot be a geodesic, contrary to hypothesis.
\end{proof}

\begin{remark}
An alternative but slightly less elementary way to prove \cref{L:ConvexityOfThatc} is to prove $\That_{\hat{c}}$ really is a horoball in the usual sense of \cite[page 267]{BridsonHaefliger}, by lifting the flow towards the cusp, and then to use that horoballs in CAT(0) spaces are convex \cite[page 271]{BridsonHaefliger}. 
\end{remark}
%
%
%

\subsection{Stratum domains}\label{SS:stratumdomains}

Now we can define the stratum domains.  They are indexed by
the strata  $\Shatcirc$ in~$\Bhat$ for all $\Shatcirc$ of positive dimension; we choose not to define stratum domains for 0-dimensional strata.  If $\Shat$ is the closure
of a $d$-dimensional
stratum $\Shatcirc$ , then the Gromov hyperbolic
space $\cC(\Shatcirc)$ associated to $\Shatcirc$ is defined as the electrification $\Shatzap$ of $\Shat$ along
\begin{enumerate}
\item the closures of the $(d-1)$-dimensional strata that lie in~$\Shat$, and
\item the nonempty intersections of $\Shat$ with horoballs of $\Bhat$.
\end{enumerate}
\cref{L:StratDomainsHyp} below
shows that $\Shatzap$ is Gromov hyperbolic. 

An HHS structure also
requires a coarsely Lipschitz
projection map $\pi_{\Shat}:\Bhatthick\to\cC\Shat$, which we take to be the composition
$$
\pi_{\Shat}:\Bhatthick\sset\Bhat\to\Shat\sset\Shatzap.
$$ 
The 
map $\Bhat\to\Shat$ is
closest-point projection, which
makes
sense because $\Bhat$ is CAT($-1$) and $\Shat$ is convex.  
This map and the two inclusions are distance non-increasing,
so
$\pi_{\Shat}$ is certainly
coarsely Lipschitz.

\begin{remark}\label{R:extensionstrata}
Continuing \cref{R:extension}, we note that the true domain of this map is $\Bhat^{\Shat} = \Bhat$.
\end{remark}

\begin{lemma}[Gromov hyperbolicity of stratum domains]\label{L:StratDomainsHyp}
If $\Shat$ is any stratum closure in~$\Bhat$,
    then $\Shatzap$ 
    is
Gromov hyperbolic.
\end{lemma}

\begin{proof}
We observed above that
    $\Shat$ is CAT($-1$), hence Gromov hyperbolic.
    By \cref{P:ElectHyp}, it is enough to show that each of the
    subsets of~$\Shat$ along which we electrify to get~$\Shatzap$
    is convex.  
    First, all stratum closures are convex in~$\Bhat$,
    so those that lie in  $\Shat$ are convex there.
    Second, by \cref{L:ConvexityOfThatc}, all  horoballs of~$\Bhat$
    are convex in~$\Bhat$. So the horoballs that meet $\Shat$ do so
    in convex subsets.
\end{proof}

\begin{remark}\label{R:CombinatorialModel}
If the union of the hyperplanes in $\Bhat$ is connected and either $n\geq 3$ or $\Gamma$ is cocompact, one can show that the ``top-level'' hyperbolic space $\cC(\Bhatcirc)$ is quasi-isometric to the intersection graph of hyperplanes of $\Bhat$. (If $n=2$ and $\Gamma$ is non-uniform, one should add vertices corresponding to the cusps.) 
%
%
%

 As recalled and reinterpreted in \cite[Section 7]{AllcockCarlsonToledoSurfaces}, Libgober \cite{Libgober} showed that in the case of the moduli space of cubic surfaces $\Gammahat$ is generated by a finite set of elements that each stabilize a hyperplane in $\Bhat$, such that the union of the hyperplanes thus obtained is connected. Keeping in mind there is only one orbit of hyperplanes \cite[Theorem 7.21]{AllcockCarlsonToledoSurfaces}, this shows that the union of the hyperplanes of $\Bhat$ is connected in this case. 

Thus, a more combinatorial and geometric top-level hyperbolic space is available in the cubic surfaces case, which is more closely analogous to classical curve graphs (or curve complexes). 
\end{remark}

\begin{remark}\label{R:HoroballCapStratProj}
If $\Shat$ enters a horoball $\That_{\hat{c}}$, then the closest point projection of $\Shat$ to $\That_{\hat{c}}$ is $\Shat \cap \That_{\hat{c}}$, since $\Shat$ is the intersection of fixed point sets of isometries preserving $\That_{\hat{c}}$.

Additionally, the projection of $ \That_{\hat{c}}$ to $\Shat$ is $\Shat \cap \That_{\hat{c}}$, because, as just established, projection to $\That_{\hat{c}}$ preserves $\Shat$ and because this projection is distance non-increasing. 
\end{remark}

\begin{remark}\label{R:OrthIntersection}
Additionally using that commuting elements preserve each others fixed point sets, one can see that if $\Shat$ and $\Shat'$ are intersecting stratum closures, then $p_{\Shat}(\Shat')=\Shat\cap \Shat'$.
\end{remark}

\section{Swaddling and curvature }\label{S:Swaddling}

In this section, which is largely self-contained, we develop a sort of theory of $\bR$-bundles, first over arbitrary spaces and then in a more geometric setting. 

\subsection{Swaddling} 

We start by introducing swaddling, which coarsely identifies all the different fibers, and give a simple condition under which the result is a quasi-line (a space quasi-isometric to $\bR$). 

\begin{definition}\label{D:SwaddlingData}
Consider a map $E\to X$ of sets, and assume we specify an action of $\bR$ on $E$ which preserves fibers and acts simply transitively on each fiber. (So, this is a principal $\bR$-bundle but without topology.) Assume that for each pair $(x,y)\in X^2$ we have an $\bR$-equivariant  map  $$m_{x,y}: E_x\to E_y$$ from the fiber $E_x$ over $x$ to the fiber $E_y$ over $y$. We assume $m_{y,x}=m_{x,y}^{-1}$ for all $x,y\in X$, which implies that $m_{x,x}$ is the identity for all $x\in X$. We call this data $(E\to X, m_{\bullet, \bullet})$ swaddling data. 
\end{definition}

\begin{definition}\label{D:SwaddledSpace} 
Given a choice of swaddling data, we define the swaddled space $E^{\sw}$ as follows. For each 
   pair $x,y \in X$ of distinct points, consider a strip $S_{x,y}=[0,1]\times \bR$, equipped with the usual action of $\bR$ on the second coordinate by translation. We fix a $\bR$-equivariant identification between $\{0\} \times \bR$ and $E_x$, and also between $\{1\} \times \bR$ and $E_y$, such that via these identifications $m_{x,y}((0,t)) = (1,t)$. We define $E^{\sw}$ to be the disjoint union of all the fibers $E_x$ and all the strips $S_{x,y}$ with the specified identifications between fibers and edges of strips. 

We equip $\Esw$ with the obvious path metric, which can be defined  formally as follows: Each point of $\Esw$ has a neighborhood which is either an open interval cross $\bR$ or the join of half-open intervals cross $\bR$. We equip these neighborhoods with the usual product metric, and note that this gives an open cover of $\Esw$ with a choice of metric on each open set in the cover such that these choices agree on overlaps. We equip $\Esw$ with the associated path metric, as in \cite[page 32]{BridsonHaefliger}. (The details of this metric will not be important for us.)
\end{definition}

There is a natural inclusion map $E\to E^\sw$. If $x,y$ are distinct points of~$X$, the distance between their fibers in $E^\sw$ is 1.
This is the sense in which we
think of swaddling as ``weakly'' identifying fibers
with each other.
This compression of~$E$, 
in all directions except the fiber direction (``head-to-toe''), is 
the source of the word \emph{swaddled}.

\begin{definition}\label{D:AlmostTrivial}
For $M \geq 0$, say the swaddling data $(E\to X, m_{\bullet, \bullet})$ is $M$-almost trivial if  for all $x,y,z\in X$, the $\bR$-equivariant map
$$ m_{z,x} \circ m_{y,z}\circ m_{x,y}: E_x \to E_x,$$
which must be translation by some real number, is a translation of size at most $M$.
\end{definition} 

\begin{lemma}\label{L:Swaddled}
Suppose swaddling data $(E\to X, m_{\bullet, \bullet})$ is $M$-almost trivial for some $M\geq 0$. Then the inclusion of any fiber 
into the swaddled space $E^\sw$ is a quasi-isometry.   
\end{lemma}

Note that each fiber, being canonically identified with $\bR$ up to translations, has the usual metric on $\bR$. In particular, we get that $E^\sw$ is a quasi-line (a space quasi-isometric to $\bR$).

\begin{proof}
Let us consider the fiber of $x\in X$. The inclusion of the fiber of $x$ is distance non-increasing,
and every point of $E^\sw$ lies within distance $\frac32$ of this
fiber. So it suffices to build a quasi-inverse  $r$ to this inclusion (see for example \cite[Proposition 5.1.10]{Loh}).

Consider the disjoint union of the fibers of $E\to X$, which is a $\frac12$-dense subset of $E^\sw$. We define a map $r$ from this subset to the fiber of $x$ as follows: For any point $e$ of $E$, lying over a point $y\in X$, we define $r(e) = m_{y,x}(e)$. 

To see that $r$ is coarsely Lipschitz, it suffices to consider two points of $E$ that differ by a map $m_{y,z}$. These have distance at most 1 in $E^\sw$, and we must show their images have uniformly bounded distance in the fiber of $x$. However the assumption exactly gives that their images have distance at most $M$. 
\end{proof} 

Above $E$ and $X$ are just sets, but in many cases they will come with metrics, and we note the following metric compatibility result, whose proof is left as an exercise. 

\begin{lemma}\label{L:MetricCompatibilityOfSwaddling} 
Suppose $E$ is equipped with a path metric, and there exists $\epsilon>0$ and $C>0$ such that for all $e,f\in E$ over points $x,y\in X$ we have that $d(e,f)\leq \epsilon$ implies that the unique  $t$ such that 
$$m_{x,y}(e) = t \cdot f$$
has $|t|\leq C$. Then the inclusion $E\to E^\sw$ is coarsely Lipschitz. 
\end{lemma}
%
%

\subsection{Other points of view (not required)}\label{SS:OtherPOV}

Given a central extension 
$$1 \to Z \to G \to Q \to 1$$
any choice of section $s:Q\to G$ gives rise to a $Z$-valued 2-cocycle $c$ on $Q$ defined by 
$$c(q_1, q_2) = s(q_1 q_2)^{-1} s(q_1) s(q_2).$$
The class $[c]\in H^2(Q, Z)$ is the extension class of the central extension, and gives rise to the well-known bijective correspondence between central $Z$-extensions up to equivalence and $H^2(Q, Z)$. Especially when $Z=\bZ$, the extension class is also known as the Euler class. 

 Assume $Z$ is equipped with a suitable metric, for example the word metric on a finitely generated group or the usual metric on $\bR$. A 2-cocycle $c$ is called bounded if $c(Q,Q)\subset Z$ is bounded, and an element of $H^2(Q, Z)$ is called bounded if it can be represented by a bounded 2-cocycle. A central extension is called bounded if its extension class is bounded. 

If $Z=\bR$, then we can take $E=G$ and $X=Q$ in the discussion of swaddling. Since $Z=\bR$, we can arrange for  $s(q^{-1}) = s(q)^{-1}$. Under these assumptions,  the section $s$ determines swaddling data as follows: For every $x,y\in X^2$, there is a unique $q\in Q$ such that $y=qx$, and we set $m_{x,y}$ to be left multiplication by  $s(q)$ restricted to ${E_x}$.

The map 
$$m_{q_2 q_1\cdot x_0, x_0} \circ m_{q_1\cdot x_0, q_2 q_1\cdot x_0} \circ m_{x_0, q_1\cdot x_0} : E_{x_0} \to E_{x_0}$$
is  translation by  $c(q_2, q_1)\in Z$, and  $c$ is bounded if and only if the  swaddling data is almost trivial.

Under appropriate assumptions one can produce central extensions from the swaddling setup. Even though these assumptions do not always hold,  \cref{L:Swaddled} should be seen as closely related to Gersten's result \cite[Theorem 3.1]{GerstenBounded} that bounded central $\bZ$-extensions are quasi-isometric to direct products $\bZ\times Q$. See also \cite{FrigerioSisto,FournierFacioMangioniSisto} for more results related to bounded central extensions. 

Quasi-lines can also be built from quasi-morphisms \cite[Lemma 4.15]{AbbottBalasubramanyaOsin}, and, following \cite{HagenRussellSistoSpriano}, this has become a valuable tool in the theory of hierarchical hyperbolicity \cite{HagenMartinSisto,DowdallDurhamLeiningerSisto,FournierFacioMangioniSisto,Tao}. Conversely quasi-morphisms can be defined using appropriate actions on quasi-lines \cite[Section 4.1]{Manning}. 

\cite[Proposition 2.9]{FournierFacioMangioniSisto} characterizes boundedness of central extensions in terms of quasi-homomorphisms. We can write every element of $G$ uniquely as $z\cdot s(q)$ for some $z\in Z$ and $q\in Q$. Consider the map 
$$\phi: G\to Z, \quad\quad \phi(z\cdot s(q))=z,$$
%
%
%
and compute, for $g_i=z_i s(q_i)$,
$$\phi(g_1g_2)- \phi(g_1)- \phi(g_2)  
= c(q_1,q_2).
$$
Thus $c$ is bounded if and only if $\phi$ is a quasi-homomorphism.

\begin{remark}\label{R:NotJustQuasiLines}
All quasi-lines are by definition quasi-isometric to each other. In our case we think of results like \cref{L:Swaddled} as inseparable from the definition of the space that turns out to be a quasi-line, and keep in mind that this definition provides tools to build and study maps to the quasi-lines that are produced. In particular, in our applications the space $E$ is often quite large -- sometimes not even admitting a natural cobounded group action -- and the fact that there is a natural map from the large space $E$ to the quasi-line $E^\sw$ is helpful for our analysis.
%
%
%
%
\end{remark}

\subsection{$L$-bounded curvature}

We now study swaddling data which can be thought of as coming from parallel transport along geodesics in a negatively curved base. 

\begin{definition}
We say swaddling data $(E\to X, m_{\bullet, \bullet})$ is geometric if $X$ is equipped with a CAT($-1$) metric, and, whenever $x,y, z\in X$ with $y$ on the geodesic from $x$ to $z$, we have $m_{x,z} = m_{y,z}\circ m_{x,y}.$
\end{definition}

\begin{definition}
Given a triangle with vertices $x,y,z$ in a CAT($-1$) space $X$, and swaddling data $(E\to X, m_{\bullet, \bullet})$, we say the monodromy of the triangle with vertices $x,y,z$ is the amount by which 
$$ m_{z,x} \circ m_{y,z}\circ m_{x,y}: E_x \to E_x$$
translates. This is a real number that, up to sign, depends only on $\{x,y,z\}$. 
\end{definition}

\begin{definition}\label{D:Bounded}
Given swaddling data $(E\to X, m_{\bullet, \bullet})$ with $X$ a CAT($-1$) space, and a non-negative number $L$, we say that a triangle in~$X$ 
    is {\it $L$-bounded} if the monodromy around it
has magnitude bounded above by $L$ times the area of
the comparison triangle in~$\bH^2$.

The swaddling data itself is called $L$-bounded  if the same is true for all triangles. 
\end{definition}

\begin{definition}\label{D:trianglesubdivision}
Given a triangle $T$ in a CAT($-1$) space with vertices $x, y, z$ and a point $m$ on the edge of $T$ opposite $y$, the splitting of $T$ determined by $m$ is the two triangles with vertices $y, m, x$ and $y, m, z$ respectively. 

A subdivision of $T$ is a collection of triangles obtained starting from $T$ by iteratively replacing a triangle $T'$ in the collection with the two triangles obtained by splitting $T'$. 
\end{definition}

In the next  subsection we will prove the following. 

\begin{proposition}\label{P:Subdividing}
Suppose that $(E\to X, m_{\bullet, \bullet})$ is geometric swaddling data, and fix $L\geq 0$. If every point in $X$ has a neighborhood such that every triangle in this neighborhood admits a subdivision into $L$-bounded triangles, then  $(E\to X, m_{\bullet, \bullet})$ is $L$-bounded. 
\end{proposition} 

\subsection{Subdivisions of triangles in CAT($-1$) spaces.}

The purpose of this subsection is to prove \cref{P:Subdividing}. 

The advantage of the  definition of subdivision we use is the following. 

\begin{lemma}[Super-additivity of areas]\label{L:SuperAdd}
Given a triangle $T$ in a CAT($-1$) space and a subdivision of $T$ into triangles $T_1, \ldots, T_k$,  the area of the comparison triangle for $T$ is at least the sum of the areas of the comparison triangles for the $T_i$. 
\end{lemma} 

Here is some intuition for this statement, speaking in very imprecise terms: The area of a triangle gets smaller when curvature is more negative. The CAT($-1$) space can have more negative curvature than the hyperbolic plane. The subdivision can see more of the negative curvature beyond what is guaranteed by CAT($-1$), and so allows for a smaller area. 

\begin{proof}
    Consider two triangles in $\bH^2$, with a common edge
    and disjoint interiors.  Write $a,w$ for the vertices
    at the end of their common edge, and suppose the
    sum of the angles at~$w$ is at least $\pi$.  Write
    $v,x$ for the remaining vertices. (See \cref{F:SuperAdd}.)  Then the sum
    of the areas of these triangles is bounded above by the
    area of a triangle in $\bH^2$ whose edge lengths are
    $d(a,v)$, $d(v,w)+d(w,x)$ and $d(x,a)$. This is an immediate corollary of Alexandrov's Lemma \cite[page 25]{BridsonHaefliger}, as was noted in \cite[Lemma 4.5]{ChowdhuryHuRomneyTsou}. 

    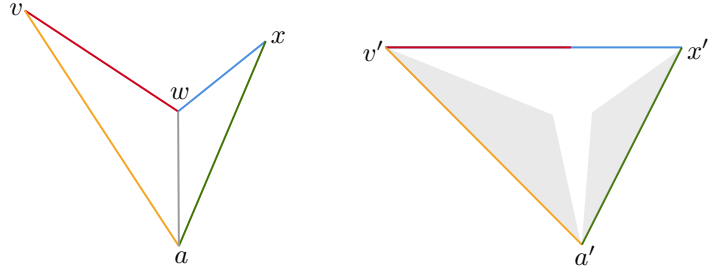
\begin{figure}[h]
\tikzset{every picture/.style={line width=0.75pt}} 
\begin{tikzpicture}[x=0.55pt,y=0.55pt,yscale=-1,xscale=1]
\draw [color={rgb, 255:red, 208; green, 2; blue, 27 }  ,draw opacity=1 ]   (43.5,41) -- (148.5,110) ;
\draw [color={rgb, 255:red, 74; green, 144; blue, 226 }  ,draw opacity=1 ]   (148.5,110) -- (208.5,62) ;
\draw [color={rgb, 255:red, 74; green, 144; blue, 226 }  ,draw opacity=1 ]   (290.5,66.25) -- (417.5,66.25) -- (493.5,66.25) ;
\draw [color={rgb, 255:red, 208; green, 2; blue, 27 }  ,draw opacity=1 ]   (290.5,66.25) -- (417.5,66.25) ;
\draw [color={rgb, 255:red, 245; green, 166; blue, 35 }  ,draw opacity=1 ]   (43.5,41) -- (149,202) ;
\draw [color={rgb, 255:red, 65; green, 117; blue, 5 }  ,draw opacity=1 ]   (149,202) -- (208.5,62) ;
\draw [color={rgb, 255:red, 245; green, 166; blue, 35 }  ,draw opacity=1 ]   (290.5,66.25) -- (425,201) ;
\draw [color={rgb, 255:red, 65; green, 117; blue, 5 }  ,draw opacity=1 ]   (493.5,66.25) -- (425,201) ;
\draw [color={rgb, 255:red, 155; green, 155; blue, 155 }  ,draw opacity=1 ]   (148.5,110) -- (149,202) ;
\draw  [draw opacity=0][fill={rgb, 255:red, 155; green, 155; blue, 155 }  ,fill opacity=0.22 ] (405.13,112.38) -- (423.95,198.68) -- (290.5,66.25) -- cycle ;
\draw  [draw opacity=0][fill={rgb, 255:red, 155; green, 155; blue, 155 }  ,fill opacity=0.21 ] (431.93,110.65) -- (492.99,67.26) -- (424.81,201.36) -- cycle ;
\draw (141,91.4) node [anchor=north west][inner sep=0.75pt]    {$w$};
\draw (144,203.4) node [anchor=north west][inner sep=0.75pt]    {$a$};
\draw (31,34.4) node [anchor=north west][inner sep=0.75pt]    {$v$};
\draw (210,53.4) node [anchor=north west][inner sep=0.75pt]    {$x$};
\draw (273,59.4) node [anchor=north west][inner sep=0.75pt]    {$v'$};
\draw (495,58.4) node [anchor=north west][inner sep=0.75pt]    {$x'$};
\draw (418,199.4) node [anchor=north west][inner sep=0.75pt]    {$a'$};
\end{tikzpicture}
\caption{A schematic of the proof of \cref{L:SuperAdd}.}
\label{F:SuperAdd}
\end{figure}

%
%


Keeping in mind the triangle inequality for Alexandrov angles \cite[page 10]{BridsonHaefliger} and \cite[page 161,  Proposition 1.7(4)]{BridsonHaefliger} this shows the result when $k=2$, and the result follows by induction on $k$. 
\end{proof}

The definition of subdivision is flexible enough to allow the following.

\begin{lemma}[Alexandrov patchwork]\label{L:patchwork}
    Let $X$ be a CAT$(-1)$ metric space, and let $\cU$ be an open cover of $X$. Then every triangle $T$ in $X$ has a subdivision into triangles each of which is contained in an element of $\cU$. 
\end{lemma}

 \begin{proof}[Sketch of proof of \cref{L:patchwork}]
Suppose
    $\Delta$ is a geodesic triangle with vertices $a$, $v$ and~$x$.
    Consider the surface~$S$ swept out by the $\overline{aw}$ as $w$ moves
    along~$\overline{vx}$.  
    
    For each point $s$ of~$S$ pick a $\delta_s>0$ such that the ball of radius $2\delta_s$ about $s$ is contained in an element of $\cU$. 
    By compactness, the covering of~$S$ by the open balls of radius $\delta_s$ about all the $s\in S$ has a finite subcover. It follows that there is some $\lambda>0$ such that every point $s$ of $S$ has that the ball of radius $\lambda$ about $s$ is contained in an element of $\cU$. 
    
    So every geodesic triangle in~$X$ whose three vertices lie
    within $\lambda$ of some point of~$S$ is contained in a ball which is contained in an element of $\cU$. (Here we use that balls are convex.)
    The Alexandrov patchwork argument, used for example in \cite[page 199-200]{BridsonHaefliger} and \cite[pages 102-104]{AlexanderKapovitchPetrunin},
    subdivides~$\Delta$ into such triangles.
\end{proof}

\begin{proposition}\label{P:CombiningBoundedTriangles}
Fix geometric swaddling data. If a triangle $T$ can be subdivided into $L$-bounded triangles, then $T$ is $L$-bounded. 
\end{proposition}

\begin{proof}
We call the triangles in the subdivision ``small triangles''. 
\cref{L:SuperAdd} gives that the area of the comparison triangle for $T$ is at least the sum of the areas of the comparison triangles for small triangles. The monodromy of $T$ is the sum of the monodromies of the small triangles (after making appropriate orientation conventions).   The proposition follows immediately.
\end{proof}

\begin{proof}[Proof of \cref{P:Subdividing}]
This follows immediately from \cref{L:patchwork} and \cref{P:CombiningBoundedTriangles}. 
\end{proof}

\subsection{Pullbacks} 

Our goal is now to show that frequently $L$-bounded curvature is stable under pullbacks. 

\begin{definition}\label{D:piecewisegeodesic}
Suppose $\hat{X}$ and $X$ are CAT($-1$) spaces. Then we will say a map $b:\hat{X} \to X$ is piecewise geodesic if every geodesic segment in $\hat{X}$ can be broken into finitely many sub-geodesic segments that each map isometrically to their image in $X$. 
\end{definition} 

Note that any piecewise geodesic map is 1-Lipschitz. 

\begin{lemma}\label{L:DefinePullback}
Let $b:\hat{X} \to X$ be a piecewise geodesic map of CAT($-1$) spaces. Suppose $(E\to X, m_{\bullet, \bullet})$ is geometric swaddling data. Then there is a unique family of maps $\hat{m}_{\bullet, \bullet}$ such that $(b^*(E)\to \Xhat, \hat{m}_{\bullet, \bullet})$ is geometric swaddling data for the pullback, and such that if $x,y\in \hat{X}$ is such that the geodesic from $x$ to $y$ maps isometrically to its image in $X$, then $\hat{m}_{x,y}=m_{b(x), b(y)}$. 
\end{lemma}

By definition, the fiber of the pullback bundle $b^*(E)$ over a point $x\in \Xhat$ is naturally identified with the fiber of $E$ over $b(x)$, and it is via such identifications that the equality $\hat{m}_{x,y}=m_{b(x), b(y)}$ should be interpreted. 

\begin{proof}
For any $x, y\in \Xhat$, there exists a sequence $x=x_0, x_1, \ldots, x_k=y$ of points on the geodesic $[x,y]$ such that each geodesic $[x_i, x_{i+1}]$ maps isometrically to a geodesic in $X$. We define $$\hat{m}_{x,y} = m_{b(x_{k-1}), b(x_k)} \circ \cdots \circ m_{b(x_0), b(x_1)}.$$
Because the original swaddling data  $(E\to X, m_{\bullet, \bullet})$ is geometric, this definition is stable under adding extra points to the sequence. Taking common subdivisions shows that $\hat{m}_{x,y}$ is well-defined. Uniqueness is clear. 
\end{proof}

\begin{definition}
We define the swaddling data $(b^*(E)\to \Xhat, \hat{m}_{\bullet, \bullet})$ produced by \cref{L:DefinePullback} to be the pullback swaddling data. 
\end{definition} 

\begin{definition}\label{D:piecewisetriangular}
Suppose $\hat{X}$ and $X$ are CAT($-1$) spaces. Then we will say a map $b:\hat{X} \to X$ is piecewise triangular if every triangle in $\hat{X}$ can be subdivided, in the sense of \cref{D:trianglesubdivision}, into (smaller) triangles such that the restriction of $b$ to each edge of one of these (smaller)  triangles is an isometric embedding. 
\end{definition} 

Note that piecewise triangular implies piecewise geodesic, since any geodesic segment is the edge of a triangle. We now have our main goal for this section. 

\begin{corollary}\label{C:PullBackQuasiLine} 
Suppose that $(E\to X, m_{\bullet, \bullet})$ is geometric swaddling data that is $L$-bounded, and suppose that $b:\Xhat\to X$ is piecewise triangular. Then the pullback swaddling data $(b^*(E)\to \Xhat, \hat{m}_{\bullet, \bullet})$ is $L$-bounded. In particular, it is almost trivial, and hence the swaddled space $b^*(E)^\sw$ is a quasi-line. 
\end{corollary} 

\begin{proof}
It follows from \cref{P:CombiningBoundedTriangles} and the definitions that $(b^*(E)\to \Xhat, \hat{m}_{\bullet, \bullet})$ is $L$-bounded, and it is immediate from the definitions that $L$-bounded implies almost trivial.
The final claim, that the swaddled space is a quasi-line, is provided by \cref{L:Swaddled}. 
\end{proof} 

\section{Cusp domains }\label{S:CuspDomains}

Our goal in this section is to continue to fulfill our promise
made in \cref{SecModelSpace}, by defining more of the domains
in the HHS on $\Bhatthick$, and the projections from $\Bhatthick$ to them.
The domains defined in this section, called  ``center domains'' and ``tree domains'', 
arise from the horoballs of $\Bhat$. Each horoball will give rise to $n-1$ tree domains and 1 center domain. 

In \cref{SecModelSpace} we introduced the horoballs
$T_c\sset B$ centered at the cusps $c$
of~$\Gamma$.  As usual, we add a
superscript $^\circ$ to indicate ``minus the hyperplanes'',
i.e. $\Tcirc_c= T_c\cap\Bcirc$.  
By $\partial\Tcirc_c$ we mean $(\partial T_c)^\circ=\partial T_c\cap\Bcirc$,
not $\partial(\Tcirc_c)$.

We continue to write $\That_{\hat{c}}$ for a horoball
of~$\Bhat$ lying over~$T_c$, where $\chat$ indicates
one of the many cusps of~$\Bhat$ that lie over~$c$.
We attach a superscript~$^\circ$ as before, i.e.
$\Thatcirc_\chat$ and $\partial\Thatcirc_\chat$ are
the intersections of $\That_\chat$
and  $\partial\That_\chat$ with~$\Bhatcirc$.

We will show that $\partial\That_\chat$ is a type of $\bR$-bundle over a product of $n-1$ Gromov hyperbolic spaces $\Chat_i$. The $\Chat_i$ are quasi-isometric to electrifications of trees, hence the term ``tree domain''. A swaddling of $\partial\That_\chat$ is related to the center of the Heisenberg group, hence the term ``center domain''. 
Formal definitions require some background about
the geometry of  a complex hyperbolic horosphere.

\subsection{The Heisenberg group and horoballs}\label{SS:HeisenbergGroup}

Here we follow the discussion in \cite[Section 3]{AllcockIsoperimetric}, but see also \cite[Chapter 4]{Goldman} and \cite[Section 4]{Parker} for details. 

Let $H^{2n-1}$ denote the $(2n-1)$-dimensional Heisenberg group, and let $Z=Z(H^{2n-1})$ denote its center. This center is isomorphic to $\bR$, and the quotient $H^{2n-1}/Z$ is isomorphic to $\bR^{2n-2}$ as a Lie group.

Start by supposing
$c$ lies in~$\partial B^n$ and $T_c$ is a horoball
with center~$c$. Consider the bounding horosphere~$\partial T_c$.
The $\PU(n,1)$-stabilizer of $\partial T_c$ is the semidirect 
product of the $(2n-1)$-dimensional Heisenberg group~$H^{2n-1}$ by
the unitary group $U(n-1)$.  The factor $H^{2n-1}$ acts simply
transitively on~$\partial T_c$, with the $U(n-1)$ factor being
any chosen point stabilizer.  
Under the action of the center of $H^{2n-1}$,  $$Z=Z(H^{2n-1})\iso\R,$$ $\partial T_c$ becomes
a principal $\R$-bundle.  The base of this bundle is the set of
$Z$-orbits, which can be identified with the set of complex
lines in $\bP(\C^{n,1})$ that pass through~$c$ and meet~$B^n$.
Another way to think of the base is as a torsor for
$H^{2n-1}/Z\iso\C^{n-1}$.  
We write $\C^{n-1}$ and not $\R^{2n-2}$,
because the induced action of~$U(n-1)$ identifies the quotient group
with complex Euclidean space~$\C^{n-1}$ (up to a scale factor).
The group structure of the central extension
$$
1\to Z\to H^{2n-1}\to\C^{n-1}\to 1
$$
is
determined by the 
commutator map $H^{2n-1}\times H^{2n-1}\to Z$.  This 
carries the same information as 
the antisymmetric bilinear form $\C^{n-1}\times\C^{n-1}\to Z$
that it induces.
By $U(n-1)$ symmetry, this bilinear form must be the
imaginary part of the Hermitian inner product on~$\C^{n-1}$, up
to a constant factor.  

Having described  $H^{2n-1}$ as a group, next we examine
it as a Riemannian manifold.  We already know
that each fiber is isometric to~$\R$,
and the base to~$\C^{n-1}$. 
Additional information
is carried by the field of hyperplanes orthogonal to the $Z$-orbits, 
in the tangent spaces of points
of $\partial T_c$.  
Usually one imagines the fiber direction as vertical, so
these hyperplanes and the vectors they contain
 are called \emph{horizontal}.  
 
A smooth path in $H^{2n-1}$ is called horizontal if
its tangent vectors are.
The horizontal hyperplanes form an Ehresmann connection.  In particular,
given
any
smooth path~$\gamma$ in the base, and any preimage in $H^{2n-1}$ of its
starting point, there is a unique horizontal lift~$\tilde\gamma$ of~$\gamma$ starting there.
If $\gamma$ is a loop, then the endpoints of $\tilde\gamma$ lie in
a single fiber.  So there is a unique
element of~$Z$ that sends the starting point
to the endpoint, called the \emph{monodromy} around~$\gamma$.

The monodromy 
can be expressed in terms of the curvature
of the connection.  This is a $2$-form on the base,
taking values in the Lie algebra of the structure group,
and  describes
``the monodromy around infinitesimal rectangles''.  
In our setting,
the structure group is~$\R$, so the curvature is an ordinary
$2$-form on (the real manifold underlying)~$\C^{n-1}$.

To find the
monodromy around a curve~$\gamma$,  integrate the curvature form over  
any disk with boundary~$\gamma$.
By $U(n-1)$ symmetry, the only possibility for the curvature form
is the imaginary
part of the Hermitian form, up to a scale factor.
This makes the monodromy
easy to understand: choose any orthonormal
basis for~$\C^{n-1}$ and write $\gamma=(\gamma_1,\dots,\gamma_{n-1})$ with respect to it.
Then, up to a global factor,
the monodromy around~$\gamma$
is the sum of the signed areas enclosed by the $\gamma_i$.

For later use we record the following. In this lemma $T_c$ may be any horoball of $B$ (there is only one up to $PU(n,1)$).

\begin{lemma}\label{L:ConstInCn1IsAlmostParallel}
There is a constant $C$ depending only on $n$ such that the following holds. 
Let $\alpha:[0,1]\to \partial T_c$ be a $C^1$ path. Let $v_0$ be a unit horizontal tangent vector at $\alpha(0)$, and let $\ol{v}_0$ be its image in $\bC^{n-1}$. Let $v_1$ be the parallel translate of $v_0$ to $\alpha(1)$ along $\alpha$ and let $v_1'$ be the horizontal lift to $\alpha(1)$ of $\ol{v}_0$. Then the distance between $v_1$ and $v_1'$ is at most $C$ times the length of $\alpha$. 
\end{lemma}

Here, when $\ol{v}_0$ is viewed as in $\bC^{n-1}$, we are identifying the $\bC^{n-1}$ as being simultaneously the tangent space to every point of the $\bC^{n-1}$ that is $\partial T_c$ mod center. 

\begin{proof}
There is a vector field $V_0$ on $\partial T_c$ whose value at a point is the horizontal lift of $\ol{v}_0$. 

The action of $H^{2n-1}$ on $\partial T_c$ induces an action by translations on $\partial T_c$ mod center. Since the $H^{2n-1}$ action is transitive, this shows that $V_0$ is invariant under $H^{2n-1}$. Since the Levi-Civita connection $\nabla$ is invariant under isometries, we get that $\nabla V_0$ is invariant under $H^{2n-1}$. In particular, the operator norm of $\nabla V_0$ is constant on $\partial T_c$.

 This gives the result because the difference between $v_1$ and $v_1'$ can be computed using the integral of $\nabla V_0$ applied to $\alpha'(t)$. 
\end{proof}

\subsection{The hyperplane arrangement}\label{SS:HyperplaneArrangement}

We will also need to understand the hyperplanes of~$B$ that contain~$c$.
Each is the orthogonal complement of a vector~$r\in\C^{n,1}$ that
satisfies $r^2>0$ and $r\perp c$, and we note the following, where ``hyperplane'' of $\bC^{n-1}$ means an affine complex codimension 1 subspace of $\bC^{n-1}$, and where we temporarily suspend our standing convention that a hyperplane of $B$ always means a hyperplane in $\cH$.  

\begin{lemma}\label{L:HinTcBasic}
The map $H\mapsto (H \cap \partial T_c)/Z$ induces a bijection between hyperplanes of $B$ containing $c$ and hyperplanes of $\partial T_c/Z \simeq \bC^{n-1}$. Furthermore, 
\begin{enumerate} 
\item two such hyperplanes $H_1$ and $H_2$ are orthogonal in $B$ if and only if the associated hyperplanes  in $\bC^{n-1}$ are orthogonal, and 
\item if $v$ is a normal vector to the hyperplane associated to $H$, then any horizontal lift of $v$ is orthogonal to $H$. 
\end{enumerate}
\end{lemma}

Since this is standard, we offer only a partial sketch. Say $r\in \bC^{n,1}$ is a positive vector. Note that $\partial T_c\cap r^\perp$
is a horosphere in $r^\perp$.  
The $\PU(n,1)$-stabilizer of $r$ 
is a copy of $\PU(n-1,1)$, and $c$'s stabilizer therein is
a copy of $H^{2n-3}\rtimes U(n-2)$.  Because the center of $H^{2n-3}$ is also the center of $H^{2n-1}$ here, it follows that each
hyperplane through~$c$ meets $\partial T_c$ in the union of fibers over
a hyperplane in~$\C^{n-1}$. We leave it to the reader to check orthogonality either using symmetry or by working in coordinates. The final claim follows because the whole horosphere is orthogonal to any geodesic to $c$, and the geodesic direction and the center direction are the only directions not seen in $\bC^{n-1}$. 

Now we specialize to our situation.  Namely, $c$ is a cusp of~$\Gamma$,
and $T_c$ is the horoball, centered there, that we defined in \cref{SecModelSpace}.
Consider the components of $\H$ that meet~$c$.  By the previous
paragraph, they correspond to a union of  hyperplanes in~$\C^{n-1}$.
Since $\H$ is an orthogonal arrangement, so is this  arrangement in~$\C^{n-1}$.

By \cref{CuspLimit}, $c$ is a limit of some $1$-dimensional stratum
of~$B$.  This stratum is the intersection of $n-1$ hyperplanes of~$B$,
all containing~$c$.  By $U(n-1)$ symmetry, we may suppose that the
corresponding  hyperplanes are parallel to the coordinate hyperplanes
in~$\C^{n-1}$.

Now consider another hyperplane of~$B$ that passes
through~$c$.  The corresponding  hyperplane must be orthogonal
to every one of these $(n-1)$  hyperplanes that it meets.  Therefore
it is parallel to one of them and orthogonal to the rest.  So the 
hyperplanes fall into $n-1$ parallelism classes.  Each class can be
parameterized by a subset $\Sigma_i$ of~$\C$.  

At a formal level, we
regard~$i$ as varying over the set of
parallelism classes of hyperplanes incident to~$c$, 
rather than $\{1,\dots,n-1\}$.
Therefore $\Gamma$ acts naturally on the set of pairs $(c,i)$. However,
we will continue to write $i=1,\dots,n-1$ in our concrete
descriptions.
We have proven most of:

\begin{lemma}
    \label{LemAffineArrangement}
    For $i=1,\dots,n-1$ there  are lattices 
    $\Lambda_i\subset\C$ and $\Lambda_i$-invariant
    discrete subsets~$\Sigma_i$ of~$\C$, 
     such that the following holds:
    $\H\cap\partial T_c$ is the union of the fibers over the  arrangement
    $$
    \bigcup_{i=1}^{n-1}
    \bigl(\C^{i-1}\times\Sigma_i\times\C^{n-1-i}\bigr)
    \subseteq\C^{n-1}
    .$$ 
\end{lemma}

\begin{proof}
The lemma is vacuous if $n=1$, so we can assume $n\geq 2$. 
    We have already proven everything except the existence of the~$\Lambda_i$
    and the discreteness of the~$\Sigma_i$.  The latter follows from the
    local finiteness of~$\H$.  Next,
    since $\Gamma$ is a lattice in $\PU(n,1)$, and $c$ is a cusp of~$\Gamma$,
    the $\Gamma$-stabilizer of~$T_c$ 
    is a lattice in the $\PU(n,1)$-stabilizer of~$T_c$, i.e. in 
    $H^{2n-1}\rtimes U(n-1)$ (for example because the image of the horosphere must be compact in the quotient). 
    By a variation on Bieberbach's Theorem, 
    this has a finite index
    normal subgroup which 
    is a lattice in $H^{2n-1}$ \cite{Auslander}. 
    
    Some pair of elements of this lattice in $H^{2n-1}$ has non-zero commutator, showing that the lattice contains a nonzero element of $Z$ and hence a lattice in $Z$. 
    The quotient of the lattice in $H^{2n-1}$ by its intersection with $Z$ is a lattice $\Lambda$ in
    $H^{2n-1}/Z=\C^{n-1}$.  Translation by $\Lambda$ preserves each parallelism class, hence
    acts on each set $\Sigma_i$.  The action on any given~$\Sigma_i$ is
    by a discrete subgroup of~$\C$, which we write~$\Lambda_i$.  
    Each $\Lambda_i$ is a lattice in~$\C$, or else $\Lambda$ could not be
    a lattice in~$\C^{n-1}$.
\end{proof}

It is easy to write down the complement of this arrangement,
namely
   $$\bC^{n-1} - \bigcup_{i=1}^{n-1}\bigl(\C^{i-1}\times\Sigma_i\times\C^{n-1-i}\bigr) = \prod_{i=1}^{n-1} \C_i^\circ,$$
where $\C_i^\circ=\C-\Sigma_i$.

\subsection{Pulling back to the universal cover}\label{SS:PullBackToHoroball}

Next, we pull all these constructions back to $\That_{\hat{c}}$.
First we observe that $\Thatcirc_{\hat{c}}$ is the universal cover
of~$\Tcirc_c$.  This uses
the fact that $\Tcirc_c\to\Bcirc$
induces an injection on~$\pi_1$, as can be seen by noting that closest point projection provides a one-sided inverse. It follows that $\partial \Thatcirc_{\hat{c}}$ is the universal cover
of~$\partial\Tcirc_c$.  Next,
the action of $Z\iso\R$ on~$\partial\Tcirc_c$
lifts to one on $\partial\Thatcirc_{\hat{c}}$. (Pull back the vector field
describing the former action, and then integrate it.)
This makes
$\partial \Thatcirc_{\hat{c}}$ into a principal $\R$-bundle.  Since its
total space is simply connected, its base is too (use the long exact sequence of the fiber bundle).

Writing $\Chatcirc_i$ for the universal cover of $\Ccirc_i$,
it follows that $\partial\Thatcirc_{\hat{c}}$ is the pullback of
$\partial \Tcirc_c\to\prod\C_i^\circ$ to the universal cover
$\prod\Chatcirc_i$ of its base.  

\subsection{Tree domains}\label{SS:TreeDomains}

We can now define the tree domains.  For each $i$, write
$\Chat_i$ for the metric completion of $\Chatcirc_i$. Note that $\Chat_i$ is CAT($0$). 

Each $i$, viewed as a parallelism class, is a tree domain; thus we get $n-1$ tree domains from each cusp $\hat{c}$. The associated metric space will be $\cC(i) = \Chat_i$.

The projection $\partial\Thatcirc_{\hat{c}}\to\prod\Chatcirc_i$ extends
to a map $\partial\That_{\hat{c}}\to\prod\Chat_i$.  The action of $\bR$  extends to $\partial\That_{\hat{c}}$, and the quotient of this action is $\prod\Chat_i$.
%
%

The HHS projection $\pi_{\chat,i}:\Bhatthick\to\cC(i)$
is the inclusion $\Bhatthick\to\Bhat$, followed by
closest point projection to $\That_{\chat}$, followed
by the projection map $\partial\That_{\chat}\to\prod\Chat_i$, followed by projection
to the $i$th factor.  

\begin{remark}\label{R:extensiontree}
Continuing \cref{R:extension}, we note that the true domain of this map is $\Bhat^{\chat,i}=\Bhat- \mathrm{int}(\That_{\hat{c}})$.
\end{remark}

The next lemma shows that $\cC(i)$ 
is Gromov hyperbolic;  its projection $\pi_{\chat,i}$
is 
coarsely  Lipschitz because it is a composition of coarsely Lipschitz maps.  

\begin{lemma}
    \label{LemTreeDomainsGromovHyperbolic}
    Suppose $\chat$ is a cusp of~$\Bhat$, and suppose $i$
    is a parallelism class of hyperplanes of~$\Bhat$ that are incident to~$\chat$.
    Then the tree domain $\cC(i)$ is Gromov hyperbolic.
\end{lemma}

We remark that the parallelism classes here have the obvious meaning:
two hyperplanes of~$\Bhat$, that are incident at~$\chat$, are
parallel just if their images in~$B$ are parallel as hyperplanes
containing~$c$.  In particular, parallel hyperplanes are
disjoint, and non-parallel hyperplanes meet.

Before we prove the lemma, we set up some of the notation that will also be used later. Let $G_i$ be the graph in $\C_i^\circ$ which is the union
of the boundaries of the Voronoi cells around the
points of $\Sigma_i$.  The size of the cells
is uniformly bounded, by the $\Lambda_i$-invariance
of $\Sigma_i$.  

We write $\Ghat_i$ for
the preimage of $G_i\subseteq\Ccirc_i$ in $\Chatcirc_i$.
Because $G_i\to\Ccirc_i$ is a homotopy equivalence, 
$\Ghat_i$ is the universal cover of~$G_i$.  In particular,
it is a tree.  

\begin{proof}
    As a tree, $\Ghat_i$ is Gromov hyperbolic.
    By a cell of $\Chat_i$ we mean a component of the preimage
    of a Voronoi cell in~$\C_i$.  Its boundary 
    is a copy of~$\R$ in the
    tree~$\Ghat_i$. So \cref{P:ElectHyp} establishes that the electrification $\Ghat_i^!$ of $\Ghat_i$ along these copies of $\R$ is Gromov hyperbolic. Thus it suffices to show that $\Chat_i$ is quasi-isometric to $\Ghat_i^!$.   

    There is a natural map $\Chat_i \to \Ghat_i^!$ that is the identity on $\Ghat_i$ and maps the interior of each cell to the corresponding generic point of $\Ghat_i^!$. There is also a natural map $\Ghat_i^! \to \Chat_i $ that again is the identity on $\Ghat_i$ and maps each generic point to the corresponding singular point of $\Chat_i$. These maps are coarsely inverse to each other and coarsely Lipschitz, proving that they are quasi-isometries. 
\end{proof}

The proof of \cref{LemTreeDomainsGromovHyperbolic} explains the terminology ``tree domains'', and gives an alternative way to think of the associated hyperbolic spaces as electrifications of trees. 

\subsection{Center domains}\label{SS:CenterDomains}

We view $\partial\That_\chat$ as a type of $\bR$-bundle over $\prod \Chat_i$. The center domain associated to $\That_\chat$ will be defined as a swaddling of $\partial\That_\chat$. We will use \cref{L:Swaddled} to see that this swaddling is a quasi-line, and it will serve to measure distance in the center direction in the horoball.

The projection map to this quasi-line is defined as the composition 
of the inclusion $\Bhatthick\to\Bhat$, followed by nearest-point
projection to $\partial\That_\chat$, followed by the natural
map $\partial\That_\chat\to\partial\Thatsw_\chat$.  This will be 
coarsely Lipschitz because it is a composition of coarsely Lipschitz maps, given that \cref{L:Thatswlip} below shows that $\partial\That_\chat\to\partial\Thatsw_\chat$ is coarsely Lipschitz. 

\begin{remark}\label{R:extensioncenter}
Continuing \cref{R:extension}, we note that the true domain of this map is $\Bhat^{\That_\chat}=\Bhat-\mathrm{int}(\That_{\hat{c}})$.
\end{remark}

\begin{remark}[Optional motivation]\label{R:tauMotivation}
Before we continue, we want to discuss the intuition and some ideas that encounter technical complications, to better motivate how we proceed. 

The base $\prod \Chat_i$ has a sort of spine, defined as $\prod \hat{G}_i$. This product of trees can be viewed as a type of Lagrangian cube complex. Because it is Lagrangian, one can show that the restriction of $\partial\That_\chat\to \prod \Chat_i$ becomes, in a certain metric sense, a trivial bundle, so the preimage of $\prod \hat{G}_i$ can be viewed as $\bR\times \prod \hat{G}_i$. This stands in sharp contrast to the extremely non-trivial bundle $\partial T_c\to \bC^{n-1}$. 

Although it is tempting to use a deformation retract $\prod \Chat_i^\circ$ to $\prod \hat{G}_i$ to make use of the triviality over the spine, one quickly encounters technical difficulties related to the non-Lipschitz nature of this retract and the fact that bounded diameter loops in $\prod \Chat_i^\circ$ can enclose arbitrarily large area. For this reason we abandon the spine and work with all of $\partial\That_\chat$. 

In the most naive approach to working with all of $\partial\That_\chat$, one might consider paths in the base $\prod \Chat_i$ that are geodesics in each factor. However, despite the fact that each $\Chat_i$ is both CAT($0$) and Gromov hyperbolic, geodesic triangles in $\Chat_i$ can enclose arbitrarily large area. 

In conclusion, pushing entirely onto the spine causes problems, and not pushing at all onto the spine causes problems, and this motivates the hybrid approach we follow now. 
\end{remark}

Swaddling $\partial\That_\chat$ requires choosing a family of $\R$-torsor
isomorphisms between pairs of fibers.  We do this by defining the
\emph{swaddling path} $\tau_{xy}$ 
for any pair of points
$$x=(x_1,\dots,x_{n-1}), y=(y_1,\dots,y_{n-1})\in\prod\Chat_i.$$
It is the path $(\tau^{(1)},\dots,\tau^{(n-1)})$, where $\tau^{(i)}$
is the path in~$\Chat_i$ defined as follows (see \cref{F:gammai}).
Begin with 
the geodesic $\overline{x_i y_i}$.  If it enters and leaves
a cell~$C$
of~$\Chat_i$, at points of $\partial C$ whose angular separation
is less than~$2\pi$ (from the perspective of the center of~$C$), then
we replace that portion of $\overline{x_i y_i}$ by the portion of 
$\partial C$  between them.

 \begin{figure}[h]
    \centering
    \includegraphics[width=0.85\linewidth]{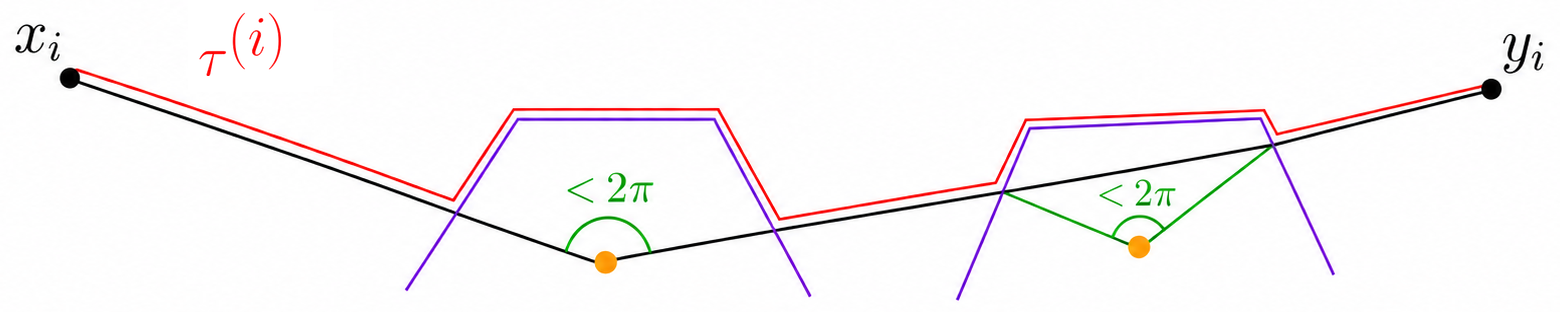}
    \caption{The definition of $\tau^{(i)}$. }
    \label{F:gammai}
\end{figure}

Doing this simultaneously, for every 
cell~$C$, yields a piecewise geodesic path
in~$\Chat_i$.  We define $\tau^{(i)}$ as the
parameterization of this path  by $[0,1]$, proportionally to
arclength.

One can almost think of it as a path in the tree $\hat{G}_i$, except that it will
enter the cells containing its endpoints, 
and will 
cut corners,
through the centers of cells, whenever this provides
a ``significant'' shortcut.

Let $q$ denote the natural map $\prod \Chat_i \to \prod \bC_i$. Each swaddling path projects to a concatenation of line segments 
in~$\prod\C_i$.  Consider each of these segments~$\sigma$ in turn.  
The fibers of $\partial\That_\chat$ over its endpoints are
identified with the fibers of $\partial T_c$ over the
endpoints of $q(\sigma)\subseteq\prod \bC_i$,
since $\partial\That_\chat\to\prod\Chat_i$ is the pullback of
$\partial T_c\to\prod\C_i$.  Furthermore,  the Ehresmann connection
on
$\partial T_c\to\prod\C_i$ trivializes over $q(\sigma)$.
This provides an identification of the fibers of $\partial T_c$
over the endpoints of~$q(\sigma)$, and therefore
an identification of the fibers of $\partial\That_\chat$
over the endpoints of~$\sigma$.  Concatenating these identifications,
as $\sigma$ varies over the segments comprising the swaddling path,
yields an $\R$-torsor isomorphism from
the fiber of~$\partial\That_\chat$ over $x$ to the fiber over~$y$.

We apply the swaddling construction to the family of $\R$-torsor
isomorphisms induced by all swaddling paths, obtaining the
swaddled space $\partial\Thatsw_\chat$. The horoball gives a center domain, which we also call $\That_{\hat{c}}$, and we define its associated metric space $\cC(\That_{\hat{c}})$ to be $\partial\Thatsw_\chat$.  

\begin{lemma}\label{L:Thatswlip}
The inclusion  $\partial\That_\chat\to\partial\Thatsw_\chat$ is coarsely Lipschitz. 
\end{lemma}

\begin{proof}
We provide a sketch only. 

Consider two points of $\partial\Thatcirc_\chat$ joined by a path $\alpha$ in $\partial\Thatcirc_\chat$ of length at most 1. Assume $\alpha$ is parametrized by $[0,1]$. By an approximation argument 
%
%
and \cref{L:MetricCompatibilityOfSwaddling}, it suffices to show that $\alpha(1)$ is a uniformly bounded central translation of  $\hat{m}_{x, y}(\alpha(0))$, where $x$ and $y$ are the images of $\alpha(0)$ and $\alpha(1)$ in $\prod_i \Chat_i$. 

The connection 1-form of $T_c\to \bC^{n-1}$ is bounded because it is invariant under the stabilizer of $T_{c}$. Hence the connection 1-form of $\Thatcirc_{\hat{c}} \to \prod \Chatcirc_i$ is also bounded. Let $\beta$ be the horizontal lift, starting at $\alpha(0)$, of the image of $\alpha$ in $\prod \Chat_i$. Note that $\beta(1)$ and $\alpha(1)$ are in the same fiber of the map to $\prod \Chat_i$, and hence differ by translation of some $t\in \bR$. We note that $t$ is bounded by a constant times the length of $\alpha$, because it is the integral over $\alpha$ of the connection $1$-form. 
%
%
%

It now suffices to show that $\beta(1)$ is a uniformly bounded central translation of  $\hat{m}_{x, y}(\alpha(0))$. Consider the image of $\beta$ in a factor $\Chat_i$, and consider also the geodesic in $\Chat_i$ joining its endpoints. This loop encloses bounded area in $\Chat_i$, for example by \cite[page 426]{BridsonHaefliger}. Additionally,   the $\Chat_i$ geodesic and the $\tau^{(i)}$ path enclose bounded area in $\Chat_i$, because of the definition of the $\tau$ path and the fact that a bounded length geodesic in a $\Chat_i$ can only go through boundedly many Voronoi cells. The result follows because horizontal lifts along different paths differ by a central translation proportional to the sum of the signed areas enclosed in the $\Chat_i$. 
\end{proof}

\begin{theorem}\label{CenterQuasiLine}
    For any cusp $\chat$ of~$\Bhat$, the metric space  $\cC(\That_{\hat{c}})=\partial\Thatsw_\chat$
    is a quasi-line and hence Gromov hyperbolic.
\end{theorem}

We write $A$ for the largest of the areas of the Voronoi cells of~$\C_i$,
over all~$i$.

\begin{lemma}
    \label{LemDegenerateSwaddlingTriangle}
    Suppose $x,z\in\Chat_i$ and $b\in\overline{xz}$.  Then the 
    area enclosed by the triangle $\tau_{xb}^{(i)}\tau_{bz}^{(i)}\tau_{zx}^{(i)}$ is at most $A$. 
\end{lemma}

\begin{proof}
    The essential case is when there is a cell~$C$ of~$\Chat_i$, whose
    interior contains $b$ but not $x$ or~$z$. See \cref{F:aby}.

\begin{figure}[h]
    \centering
    \includegraphics[width=0.85\linewidth]{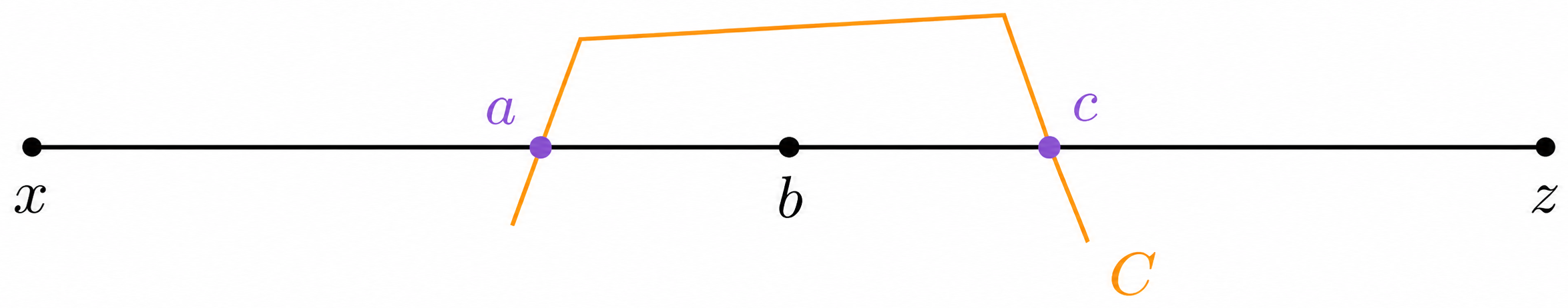}
    \caption{The proof of \cref{LemDegenerateSwaddlingTriangle}}
    \label{F:aby}
\end{figure}
    
    We define $a$ and $c$ as the
    points where $\partial C$ meets $\overline{xb}$ and $\overline{bz}$
    respectively.   Then $\tau^{(i)}_{xz}$ coincides with $\tau^{(i)}_{xb}$ between
    $x$ and~$a$, and with $\tau^{(i)}_{bz}$ between $c$ and~$z$.
    So $\tau^{(i)}_{xb}\tau^{(i)}_{bz}\tau^{(i)}_{zx}$ encloses the same area
    as $\tau^{(i)}_{ab}\tau^{(i)}_{bc}\tau^{(i)}_{ca}$.

    Now, $\tau^{(i)}_{ab}$ is $\overline{ab}$, $\tau^{(i)}_{bc}$ is 
    $\overline{bc}$, and
    $\tau^{(i)}_{ac}$ is either   $\overline{ac}$
    (if the angular separation between $a$ and~$c$ is${}\geq2\pi$)
    or the portion of $\partial C$ between $a$ and~$c$ (otherwise).
    In the former case, $\tau^{(i)}_{ac}$ coincides with $\tau^{(i)}_{ab}\tau^{(i)}_{bc}$,
    so the triangle they form encloses no area.  In the latter case,
    the triangle they form is bounded by $\partial C$ and
    $\overline{ab}\cup\overline{bc}=\overline{ac}$.
    Because the angular separation between $a$ and $c$ is${}<2\pi$,
    the area of this region is at most that
    of the Voronoi cell of~$\C_i$ over which $C$ lies.  So it
    is at most~$A$.

    The argument simplifies in the various degenerate cases.
    Namely, if $x$ resp.\ $z$ lies in the interior of~$C$, then set $a=x$
    resp.\ $c=z$, and modify the argument slightly. 
    The other possible degeneration
    is when  $b$ does not lie in the
    interior of any cell. Then it lies in~$\Ghat_i$, and 
    $\tau^{(i)}_{xz}=\tau^{(i)}_{xb}\tau^{(i)}_{bz}$, so the triangle encloses no area.
\end{proof}

\begin{lemma}
    \label{LemLargeAnglesCase}
    Suppose $x,y,z\in\Chat_i$, and $a\in\overline{xz}-\{x,z\}$, such
    that $\overline{xa}$ and $\overline{za}$ make angle${}\geq2\pi$ at~$a$.
    Then either $\overline{xa}\subseteq\overline{xy}$ or
    $\overline{za}\subseteq\overline{zy}$.
\end{lemma}

\begin{proof}
    Because $\Chatcirc_i$ is locally Euclidean, two geodesics whose 
    union is a geodesic
    make angle~$\pi$.  Therefore $a\notin\Chatcirc_i$, so
    $a$ is a ramification point of~$\Chat_i$.  
    The rays emanating from~$a$ form a copy of~$\R$,
    and
    $\overline{ax}$ and $\overline{az}$ represent two of these rays.
    
    Let $R$ be their angle bisector, i.e. the geodesic emanating from~$a$,
    whose angles with $\overline{ax}$ and $\overline{az}$ are equal.
    It might happen that $R$ eventually hits another branch point of~$\Chat_i$
    (in which case we stop it there, so $R$ is a segment), 
    or not (in which case $R$ is a ray).

    In either case, $R$ cuts $\Chat_i$ into two components, one containing~$x$
    and the other~$z$.  Therefore $R$ meets one of the other edges
    $\overline{yx},\overline{yz}$ of the triangle $xyz$.  Without loss of generality,
    suppose it meets $\overline{yz}$. See \cref{F:R}. 
    \begin{figure}[h]
    \centering
    \includegraphics[width=0.35\linewidth]{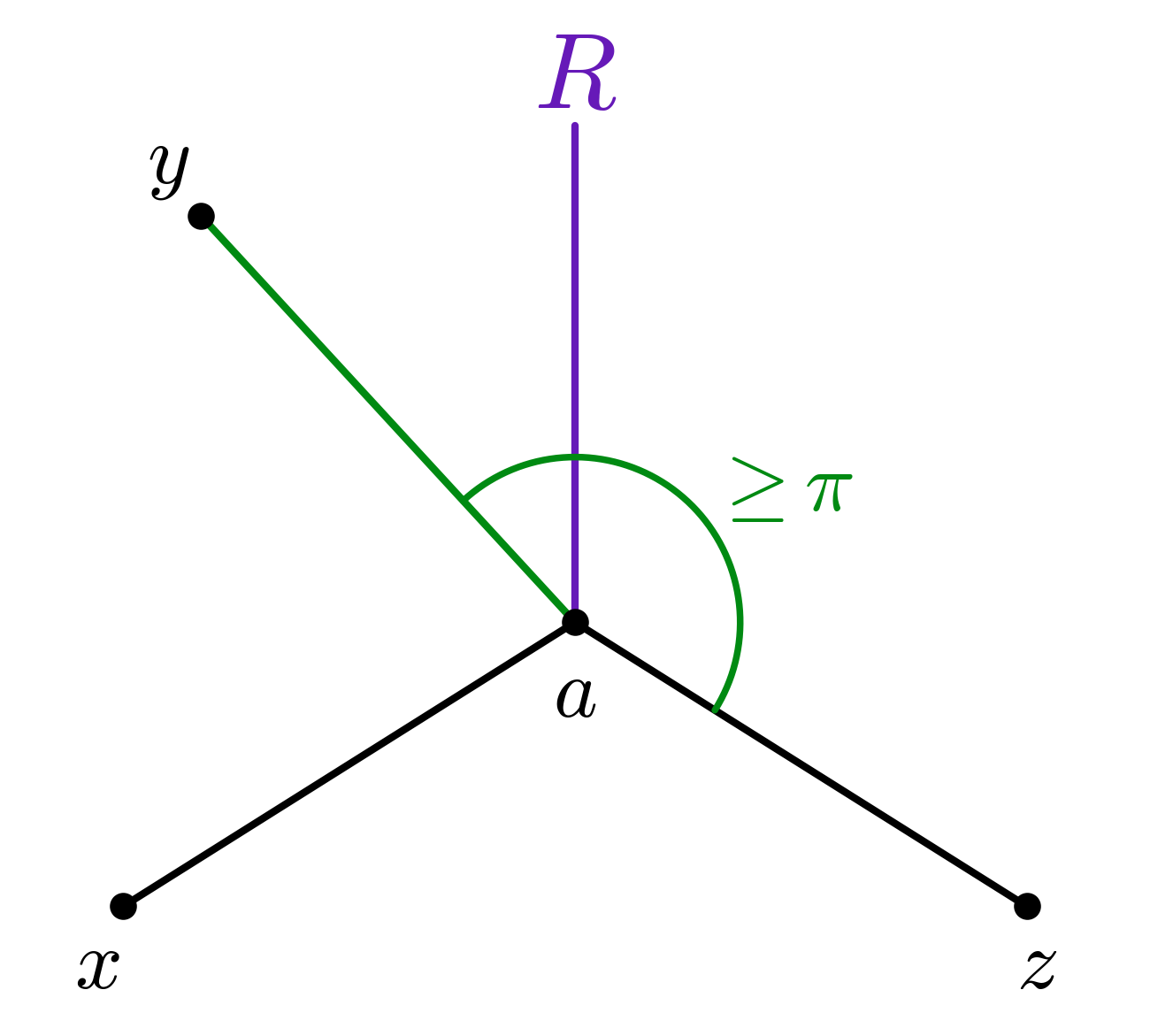}
    \caption{The proof of \cref{LemLargeAnglesCase}}
    \label{F:R}
\end{figure}
Then the angle at~$a$, between $\overline{ay}$ and $\overline{az}$,
    is at least as large as the angle between $R$ and $\overline{az}$, which 
    is at least~$\pi$.  Therefore their concatenation is a geodesic,
    in fact the geodesic $\overline{yz}$, which therefore contains $\overline{az}$.
\end{proof}

\begin{proof}[Proof of \cref{CenterQuasiLine}]
    We must show that the monodromy, around any triangle of swaddling paths
    in $\prod\Chat_i$,
    is uniformly bounded.   Because of the product structure, it is enough
    to show this for each $\Chat_i$ separately.  We claim that for any
    $x,y,z\in\Chat_i$, the area enclosed by $\tau^{(i)}_{xy}\tau^{(i)}_{yz}\tau^{(i)}_{zx}$
    is at most $9A$.  We define $x'$ as the point of $\overline{xy}\cap\overline{xz}$
    furthest from~$x$, and similarly for $y'$ and~$z'$. See \cref{F:primes}.  By construction,
    any pair of edges, of the geodesic triangle $x'y'z'$, meet only at their
    common vertex.  Also, the points $x,x',y',y$ occur in this order 
    along $\overline{xy}$ (possibly with coincidences), and similarly for
    $\overline{yz}$ and $\overline{zx}$.

    \begin{figure}[h]
    \centering
    \includegraphics[width=0.5\linewidth]{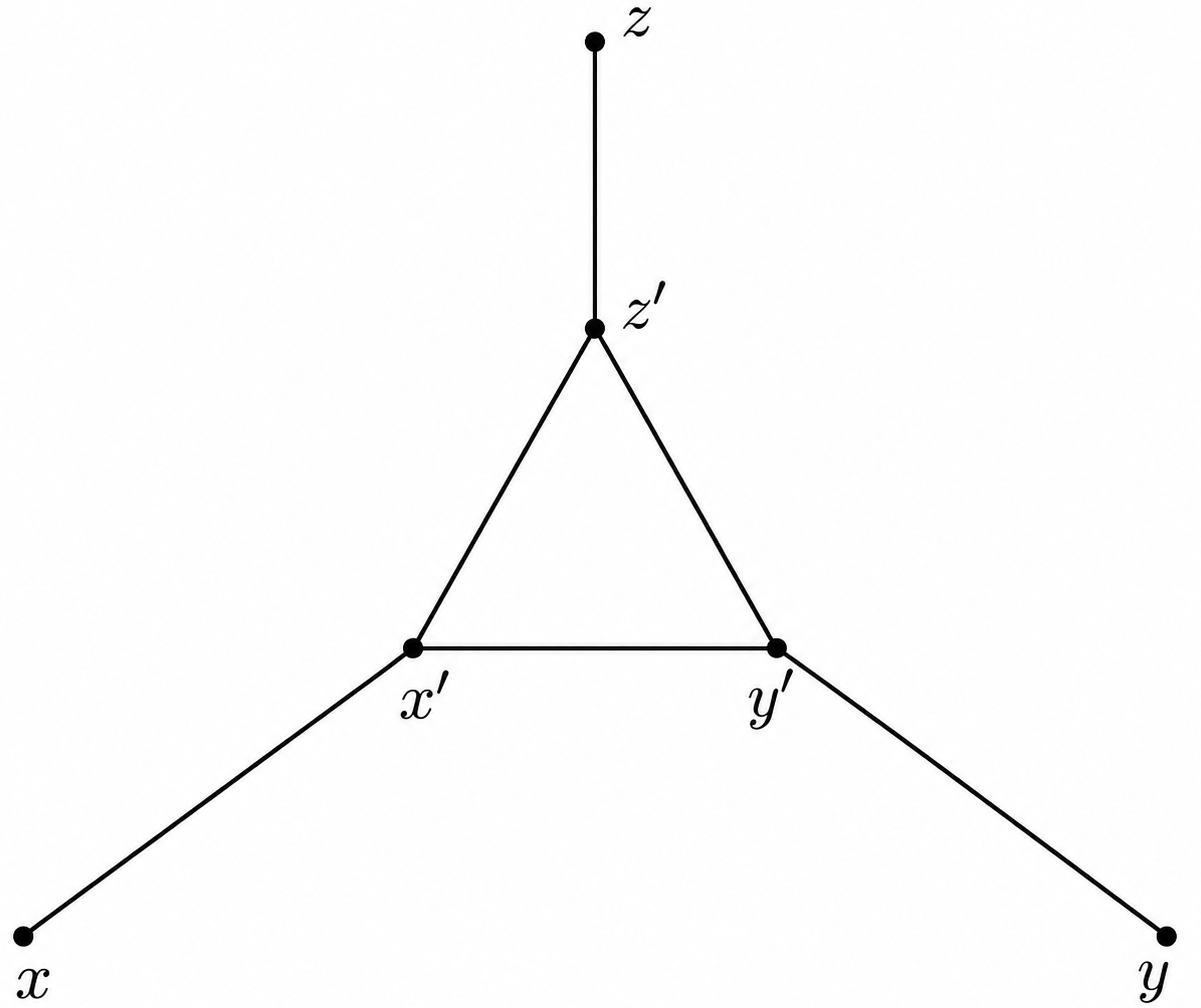}
    \caption{The definition of $x', y', z'$}
    \label{F:primes}
\end{figure}

    By six applications of \cref{LemDegenerateSwaddlingTriangle},
    the signed area enclosed by $\tau^{(i)}_{xy}\tau^{(i)}_{yz}\tau^{(i)}_{zx}$
    differs by${}\leq 6A$ from the signed area enclosed by the nonagon
    $$
    \tau^{(i)}_{xx'}\tau^{(i)}_{x'y'}\tau^{(i)}_{y'y}\cdot\tau^{(i)}_{yy'}\tau^{(i)}_{y'z'}\tau^{(i)}_{z'z}
    \cdot\tau^{(i)}_{zz'}\tau^{(i)}_{z'x'}\tau^{(i)}_{x'x}.
    $$
    This nonagon is the triangle $\tau^{(i)}_{x'y'}\tau^{(i)}_{y'z'}\tau^{(i)}_{z'x'}$
    with bigons $\tau^{(i)}_{x'x}\tau^{(i)}_{xx'}$, $\tau^{(i)}_{y'y}\tau^{(i)}_{yy'}$ and
    $\tau^{(i)}_{z'z}\tau^{(i)}_{zz'}$ attached at its corners.  Each bigon
    encloses zero area, because the swaddling paths comprising it 
    are each other's reversals.  
    So it
    is enough to show that $\tau^{(i)}_{x'y'}\tau^{(i)}_{y'z'}\tau^{(i)}_{z'x'}$
    encloses  area${}\leq3A$.   

    Having no further need for $x,y,z$, we will suppress the primes
    from $x',y',z'$.  Repeating a known fact in this notation:
        ($\star$)~any pair of edges, of the geodesic triangle $xyz$, 
        meet only at their common
        vertex.
    If two of $x,y,z$ coincide, then ($\star$) forces all three to coincide.  In this case,
    $\tau^{(i)}_{xy}\tau^{(i)}_{yz}\tau^{(i)}_{zx}$ degenerates to a point, so it
    encloses no area.  So we may suppose the three vertices are distinct.

    We claim that for every  $a\in\overline{xz}-\{x,z\}$,
    $\overline{ax}$ and $\overline{az}$ make angle${}<2\pi$.  Otherwise,
    \cref{LemLargeAnglesCase} would lead to violation of~($\star$).  From this follows
    a key fact: if $\overline{xz}$ enters and exits a cell~$C$ at points
    of~$\partial C$, then their angle (viewed from the center of~$C$) is${}<2\pi$.
    Therefore, in the definition of $\tau^{(i)}_{xz}$, the portion of $\overline{xz}$
    that passes through~$C$ is replaced by a portion of~$\partial C$.
    Applying the same reasoning to the other cells meeting~$\overline{xz}$,
    and similarly for the other two edges,
    we conclude:
    $\tau^{(i)}_{xy}\tau^{(i)}_{yz}\tau^{(i)}_{zx}$ encloses no portion of any cell,
    except perhaps for cells whose interiors contain at least one of
    $x$, $y$ or~$z$.  So it suffices to show that the enclosed 
    portion of such
    a cell has area${}\leq A$.

    We examine such a cell~$C$, writing $c$ for its center. If one of
    the vertices is~$c$, say $x=c$, then consider the geodesics
    from~$x$ to~$\partial C$, in the directions of $\overline{xy}$ and
    $\overline{xz}$. For this and the next cases, see \cref{F:ThreeCorners}, starting on the left for this case.
        \begin{figure}[h]
    \centering
    \includegraphics[width=0.75\linewidth]{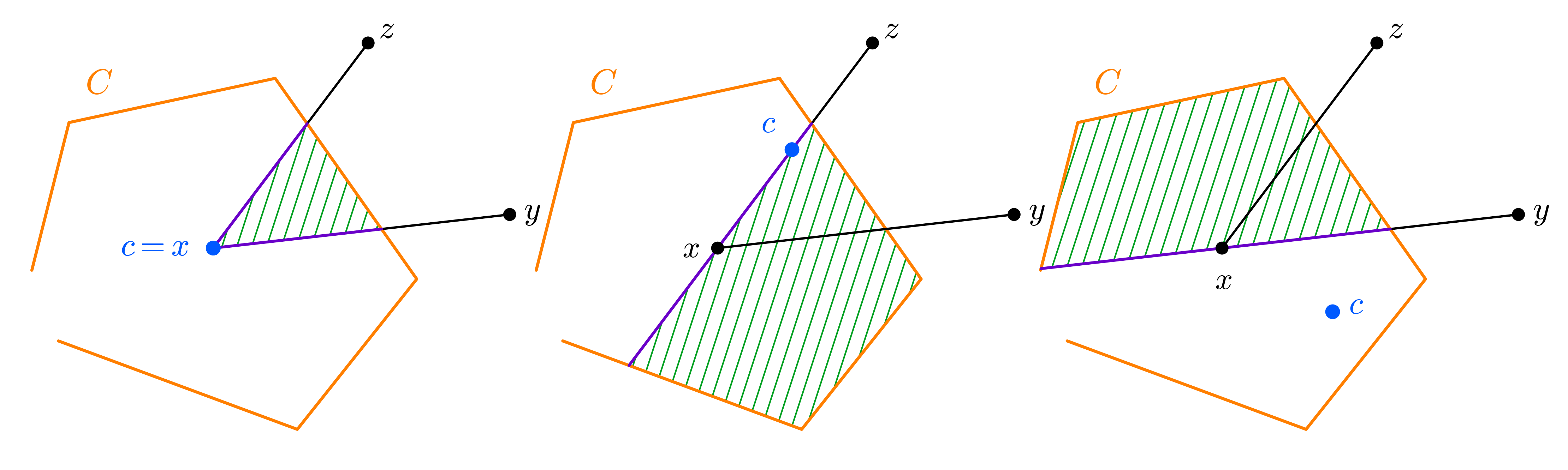}
    \caption{Three cases for the conclusion of the proof of \cref{CenterQuasiLine}. In all cases the shaded region has area at most $A$.}
    \label{F:ThreeCorners}
\end{figure}
    These make angle less than~$\pi$; otherwise,
    $\overline{xy}\cup\overline{xz}$ would be a geodesic, namely
    $\overline{yz}$, which would contradict~($\star$).  The portion of~$C$
    enclosed by $\tau^{(i)}_{xy}\tau^{(i)}_{yz}\tau^{(i)}_{zx}$ lies in the region
    bounded by these two geodesics and $\partial C$.  Since their angle
    is less than~$\pi$, this region's area is${}\leq A$.

    Now suppose $c$ is distinct from  $x,y,z$, but lies in one of the
    edges of the geodesic triangle, say~$\overline{xz}$.  
    The angle between $\overline{cx}$ and $\overline{cz}$ is less than~$2\pi$,
    or otherwise
    \cref{LemLargeAnglesCase} would yield a contradiction of~($\star$).  Consider
    the geodesics from~$c$ to~$\partial C$ in these two directions, and
    argue as in the previous paragraph.

    Finally suppose no edge of the triangle contains~$c$.  Of the edges of the triangle, choose one which comes
    closest to~$c$, say~$\overline{xy}$.  If it is ``radial'',
    meaning that it extends to a geodesic that hits~$c$, then its end near~$c$
    lies in another edge of the triangle, which is 
 non-radial and comes as close to~$c$ as $\overline{xy}$ does.  
    Using it instead,
    we may suppose
    without loss that $\overline{xy}$ is not radial.  Being non-radial,
    $\overline{xy}\cap C$ extends to a geodesic between two 
    points of~$\partial C$,
    that misses~$c$.  The portion of $C$ enclosed by
    $\tau^{(i)}_{xy}\tau^{(i)}_{yz}\tau^{(i)}_{zx}$ lies in the region bounded
    by this geodesic and~$\partial C$, which has area${}\leq A$.
\end{proof}

\subsection{Enough projections}\label{SS:HoroballEP}

We end this section with a statement that will be used later in the proof of the enough projections assumption of \cref{P:Criterion}.

\begin{lemma}\label{L:HoroballProductMetric}
For each $\kappa\geq 0$ there exists $\omega\geq0$ such that if $x, y\in \partial \That_{\hat{c}}$ and $d(x,y)\geq \omega$ then  $d_{U}(\pi_U(x),\pi_U(y))\geq \kappa$ for some cusp domain $U$ associated to $\That_{\hat{c}}$.
\end{lemma}

Here we use that the $\pi_U$ are in fact defined on all of $\partial \That_{\hat{c}}$, because $\Bhat^U$ is the complement of the interior of $\That_{\hat{c}}$. Note that it ultimately does not matter if $d(x,y)$ denotes the $\Bhat$ metric or the intrinsic path metric on $\partial \That_{\hat{c}}$ here because of \cref{L:BoundaryHorosphere}, but for concreteness let us say it denotes the later for now.

\begin{proof}
We make use of the map $\partial \That_{\hat{c}} \to \prod \Chat_i$.
Assume $x, y\in \partial \That_{\hat{c}}$ have that $d_{U}(\pi_U(x),\pi_U(y))\leq \kappa$ for all cusp domains $U$ associated to $\That_{\hat{c}}$. We must give some upper bound $\omega$ on $d(x,y)$.

Since the images of $x$ and $y$ in each factor $\Chat_i$ are close, the length of the swaddling path from $x$ to $y$ is bounded. We can use the swaddling data above to ``parallel translate'' $x$ to a point $x'$ in the same fiber of $y$. The swaddling data ultimately comes from horizontal lifts, so the distance from $x$ to $x'$ is bounded. Thus it suffices to bound the distance from $x'$ to $y$. 

But $x'$ and $y$ are in the same fiber, and their distance in that fiber is bounded by \cref{L:Swaddled} and the assumption. 
\end{proof}

\section{Angle domains }\label{S:AngleDomains}

The angle domains will be in bijection with the hyperplanes $\Hhat$ in~$\Bhat$, but to avoid confusion with the associated strata we define them to be indexed by the unit normal bundles $N^1(\Hhat)$ to the hyperplanes. 

The idea of our angle domains is simple, but takes 
work  to set up.
There is a function on $\R^3-\R$
which indicates angular position around~$\R$.  Usually one
thinks of it as taking values in a circle.  This function
lifts to one on the universal cover of $\R^3-\R$, taking
values in a copy of~$\R$.  Our goal is to construct a
similar function on $\Bhat-\Hhat$, for each hyperplane 
$\Hhat\sset\Bhat$.  Our ``angle'' function will not take
values in a copy of~$\R$, but rather in a quasi-line, i.e.
a space quasi-isometric to~$\R$.
The hyperbolic spaces associated to angle domains in our HHS structure 
will be the quasi-lines associated
to hyperplanes of~$\Bhat$, 
and the projections to them will be their ``angle'' functions.
The quasi-lines are defined intrinsically. In particular,
the $\Gammahat$-stabilizer
of~$\Hhat$ acts on
$\Hhat$'s quasi-line.

\begin{remark}\label{R:MinProjNotCan} 
These domains are analogous to annular domains for mapping class groups, and they can also be compared, roughly speaking, to a Fenchel-Nielsen twist parameter for Teichm\"uller space. In the context of Teichm\"uller space one observes that attempts to measure twisting or annular subsurface projections are subtle and sometimes difficult, and there are non-equivalent definitions with different advantages (see \cite[Appendix A]{AranaHerreraWright} for some discussion). Experts on hierarchical hyperbolicity will also recall \cite{Mangioni} as a compelling explanation that projections to annular curve graphs are not at all as canonical as one might have initially expected. Thus, even though the hyperbolic spaces for our angle domains seem very canonical, it is not a surprise that there is more than one way to define the projections to them; and in fact we exploit this flexibility to choose a definition which avoids some technical issues. 
\end{remark}

\subsection{Unit normal bundles.} 
Fix a hyperplane $\Hhat$ of $\Bhat$ and write
$\pHhat$ for nearest-point projection to it,
which makes sense because  $\Bhat$ is CAT($-1$) and
$\Hhat$ is convex.
Suppose $\ell>0$ and  $\gammahat:[0,\ell]\to \Bhat$ is a geodesic with
$\gammahat(0)\in\Hhat$.  We say $\gammahat$ is orthogonal to  $\Hhat$ if
$\pHhat\circ\gammahat(t)$ is constant. 
We define the unit normal bundle 
$N^1(\Hhat)$ to be the set of germs of such geodesics, meaning that
two geodesics
are considered equivalent if they coincide on
some interval $[0,\varepsilon]$ with~$\varepsilon>0$.
There is a surjective projection map 
$$\PHhat:\Bhat-\Hhat\to N^1(\Hhat),$$  sending each $\yhat$ to 
the germ of the 
geodesic $\overline{\pHhat(\yhat)\,\yhat}$.

The pointwise stabilizer of $\Hhat$ in $\pi_1(\Bcirc)\subset \Gammahat$ is a subgroup isomorphic to $\bZ$. By covering space theory (applied after removing hyperplanes and before taking metric completions), this corresponds to a branched cover $B'\to B$ with deck group $\bZ$. Let $H'$ be the image of $\Hhat$ in $B'$, and as usual let $H$ be the image of $\Hhat$ in $B$. Note that $B'-H'$ is the universal cover of $B-H$. 

Exactly the same considerations as above apply to 
the hyperplane $H'$ of~$B'$, yielding
the unit normal bundle $N^1(H')$ of $H'$ in~$B'$. The pointwise stabilizer of $H$ in $\Isom^+(B)$ is a copy of $S^1$, which rotates around $H$. This $S^1$ action on $B$ lifts via covering space theory to a $\bR$ action on $B'$ which fixes $H'$ pointwise. The induced action on $N^1(H')$ shows that $N^1(H')$ is a principal $\bR$-bundle. Our goal in this subsection is to show the following. 

\begin{lemma} 
    \label{LemNormalBundleIsPullback}
$N^1(\Hhat)$ is a $\bR$-bundle, and is the pullback of $N^1(H')$ via the map $\Hhat \to H'$. 
\end{lemma}

%
%

\begin{proof}
Let $H$ be the image of $\Hhat$ in $B$. 

For each $p\in H$ there is some $\epsilon$ such that any hyperplane coming within $\epsilon$ of $p$ in fact contains $p$. If $y$ is within $\epsilon$ of $p$, then the geodesic from $y$ to $H$ lies in a single stratum (except  the endpoint in $H$). Any such geodesic segment can be lifted to $\Bhat$ and $B'$ to geodesics that end on $\Hhat$ or $H'$. 

All germs in $N^1(\Hhat)$ and $N^1(H')$ arise in this way. This analysis shows in particular that there is a surjective map from $N^1(\Hhat)$ to the pullback of $N^1(H')$. This map sends the germ of a geodesic at a point of $\Hhat$ to the pair of that point of $\Hhat$ and the pushforward germ of a geodesic based at the image of that point in $H'$. 

We can see this map is injective as follows. If not, it must be because there are two germs $w, w'$ in $N^1(\Hhat)$ with the same basepoint $x\in \Hhat$ that map to the same point of $N^1(H')$. In particular, they must differ by an element $\gamma$ of the deck group of $\Bhat\to B'$ that fixes $x$. If $x$ lies in $k$ hyperplanes, its stabilizer in the subgroup of $\pi_1(\Bcirc)\subset \Gammahat$ corresponding to $B$ is a $\bZ^k$. The different coordinates of $\bZ^k$ correspond to different hyperplanes containing $x$. However, as can be seen using $B$, every hyperplane through $x$ other than $\Hhat$ contains all germs of geodesics based at $x$ and orthogonal to $\Hhat$. So actually  $w$ must be fixed by the $\bZ^{k-1}$ corresponding to the $k-1$ hyperplanes other than $\Hhat$. Applying the $\bZ$ corresponding to $\Hhat$ produces germs with different image in $N^1(H')$, so this gives the result. 

The topology and principal $\bR$-bundle structure on $N^1(\Hhat)$ can now be obtained by pulling back. 
\end{proof}

The map from $H'$ to $H$ is in fact an isomorphism, and we can think of $H'$ as a copy of $\bC\bH^{n-1}$.

Equip $N^1(H)$ with the Levi-Civita connection. Viewed as a 1-form on $N^1(H)$ with values in $\bR$ (the Lie algebra of $S^1$), it can be pulled back to a connection on $N^1(H')$, via the natural map  $N^1(H')\to N^1(H)$.

The isometry group of $B'$ fixes $H'$ as a set and acts transitively on $H'$. The curvature form of $N^1(H')$ is invariant under this action.  

\begin{remark}
In fact, symmetry shows that the curvature 2-form is a multiple of the imaginary part of the Hermitian form on the hyperplane, and it is also possible to compute the curvature form directly using \cite[pages 77, 301]{KobayashiNomizu}. There is also a soft argument to show that the curvature cannot be zero, as a consequence of a certain central extension $P(U(1,n-1)\times S^1) \to PU(1,n-1)$ not splitting. 
\end{remark}

\subsection{Preparatory lemmas.}

We also note the following at the level of $B$, where we also have a map $p_{N^1(H)}: B-H \to N^1(H)$. 

\begin{lemma}\label{L:DerivativeEqualsProj}
Suppose $\alpha(t), t\in [0,1]$ is a path with $\alpha(0)\in H$ and $\alpha'(0)$ non-zero and orthogonal to $H$. Then 
$$\lim_{t\to 0^+} p_{N^1(H)}(\alpha(t)) = \frac{\alpha'(0)}{\|\alpha'(0)\|}.$$
\end{lemma}

For a proof, compare to \cite[Theorem 5.25, Proposition 5.26, Corollary 6.39]{LeeIRM}. 

\subsection{The hyperbolic spaces.}
We start by recalling the following. 

\begin{lemma}\label{BoundedIntegralsOverTriangles}
Let $\eta$ be a closed 2-form on $\bC\bH^{n-1}$ which is invariant under the group of holomorphic isometries of $\bC\bH^{n-1}$. Then there exists $L$ such that for any geodesic triangle $T$ in $\bC\bH^{n-1}$, the absolute value of the integral of $\eta$ over any filling of $T$ is at most $L$ times the area of the comparison triangle. 
\end{lemma}

\begin{proof}
Because $\eta$ is invariant under a transitive action by isometries, it is bounded. Recall that any geodesic triangle in complex hyperbolic space can be filled by a disc whose area is at most the area of the comparison triangle in $\bH^2$ \cite[page 426]{BridsonHaefliger} (and since the exponential map is smooth, the disc can be taken to be smooth on its interior). The result follows. 
%
%
%
\end{proof}

Parallel translation along geodesics defines geometric swaddling data for $N^1(H')$. Since holonomy around a triangle can be computed as the integral of curvature on a filling disc (see for example \cite[page 41]{MarsdenMontgomeryRatiu}), keeping in mind \cref{D:Bounded} we get the following. 

\begin{corollary}\label{C:Lbounded}
The geometric swaddling data for $N^1(H')$ is $L$-bounded for some $L$. 
\end{corollary}

We now wish to pull back from $H'$ to $\Hhat$. To that end, recall the definition of a piecewise triangular map from \cref{D:piecewisetriangular}, and note the following. 

\begin{lemma}\label{L:PTri}
Let $S$ be a ball, and $\Shat$ the metric completion of the universal cover of an orthogonal arrangement in $S$. 
Then the map $\Shat\to S$ is piecewise triangular. 
\end{lemma}

It follows in particular that it is piecewise geodesic in the sense of \cref{D:piecewisegeodesic}. 

\begin{proof}
To start with, consider an open cover of $\Shat$ consisting of, for each point $p$ of $\Shat$, a ball centered at $p$ small enough to be disjoint from all hyperplanes not containing $p$. \cref{L:patchwork} gives that it suffices to consider a triangle $T$ in such a ball $U$. Let $\Delta$ denote the intersection of all the hyperplanes through $p$. We subdivide $T$ as follows, as illustrated in \cref{F:PTri1} for the first three steps.
\begin{enumerate}
    \item If a vertex~$a$  of~$T$ lies in~$\Delta$, 
    then we subdivide the opposite edge into segments, each
    of which lies in a single stratum (except perhaps for its
    endpoints).  
    We subdivide~$T$ along the geodesics from
    $a$ to these points.  Each of the resulting triangles has edges that each map to a geodesic in $S$. 
            \begin{figure}[h]
    \centering
    \includegraphics[width=0.75\linewidth]{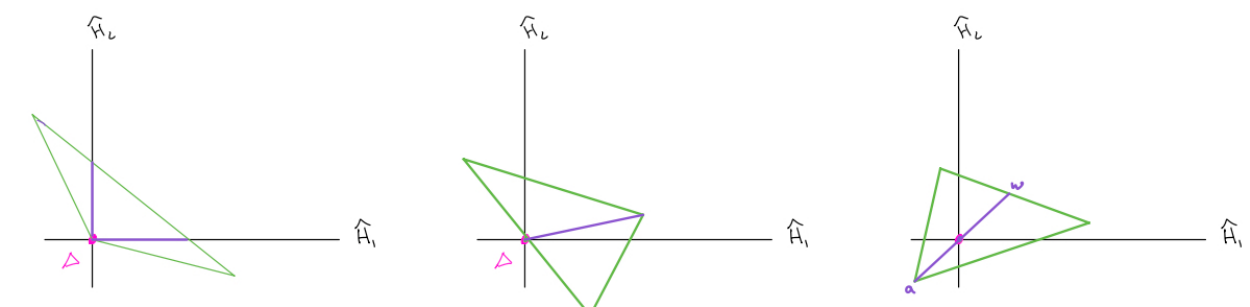}
    \caption{The proof of \cref{L:PTri}}
    \label{F:PTri1}
\end{figure}

    \item If no vertex lies in~$\Delta$, but some point $w$ of an edge
    of $T$ does, 
    then we subdivide~$T$ along the geodesic from $w$ to the opposite
    vertex.  The previous case applies to both parts of the subdivision.

    \item If $\overline{aw}$ meets $\Delta$, for some vertex $a$ of~$T$
    and $w$ in the opposite edge, then we subdivide $T$ along $\overline{aw}$.
    The previous case applies to both parts of the subdivision.

    \item If none of the previous cases apply, then the set~$W$ swept
    out by the $\overline{aw}$ misses~$\Delta$, where $a$ is a vertex 
    of~$T$ and $w$ varies over its opposite edge.  
    We apply the
    Alexandrov patchwork argument from the proof of \cref{L:patchwork}.
    The required neighborhoods, of the points of~$W$, exist by 
    an inductive argument.
    \end{enumerate}
This concludes the proof. 
\end{proof}

Keeping in mind \cref{LemNormalBundleIsPullback}, we get that \cref{L:DefinePullback} provides geometric swaddling data for $N^1(\Hhat)$. \cref{C:Lbounded} and \cref{C:PullBackQuasiLine} thus give that the swaddled space is a quasi-line: 

\begin{corollary}\label{C:AngleQuasiLine}
$N^1(\Hhat)^\sw$ is a quasi-line. 
\end{corollary}

For the angle domain $U=N^1(\Hhat)$, define $\cC(U) = N^1(\Hhat)^\sw$.
We also note the following. 

\begin{lemma}\label{L:N1inclusion}
The inclusion of $N^1(\Hhat)$ into $N^1(\Hhat)^\sw$ is coarsely Lipschitz.
\end{lemma}

Here we equip $N^1(\Hhat)$ with the path metric obtained by pulling back the natural invariant metric on $N^1(H')$ (the lifted Sasaki metric).

\begin{proof}
This is left to the reader; it is very similar to, but easier than, \cref{L:Thatswlip}. 
\end{proof} 

\begin{lemma}\label{L:StabCobounded}
Let $\Hhat_1, \ldots, \Hhat_k$ be hyperplanes that all intersect, and let $U_1, \ldots, U_k$ be the associated angle domains. Then any orbit map for the action of $\bZ^k$ on $\cC(U_1)\times \cdots \times \cC(U_k)$  gives a quasi-isometry from $\bZ^k$ to $\cC(U_1)\times \cdots \times \cC(U_k)$, with uniform constants.
\end{lemma}

Let $x$ be a generic point of the intersection of the $\Hhat_i$, i.e. $x\in (\cap_{i=1}^k \Hhat_i)^\circ.$
The $\bZ^k$ is the stabilizer of any such $x$ for the deck group of $\Bhatcirc \to \Bcirc$ and is described in \cref{SS:Local}. 

\begin{proof}
Pick $x$ as above.
For each $1\leq i\leq k$, the fiber of $N^1(\Hhat_i)$ at $x$ is, by definition, a copy of $\bR$. The result follows from the fact that the $\bZ^k$ quotient of the product of these copies of $\bR$ is a compact torus (see the discussion in \cref{SS:Local}). 
\end{proof}

\subsection{The naive projection map.}
We start our discussion with the map $$p_{N^1(\Hhat)}:\Bhat -\Hhat\to N^1(\Hhat).$$ Given a point $x$ of $\Bhat -\Hhat$, we consider its closest point projection $p_{\Hhat}(x)$ to $\Hhat$, and we recall that by definition $p_{N^1(\Hhat)}(x)$ is the germ of the geodesic from $p_{\Hhat}(x)$ to $x$. 

This map is not coarsely Lipschitz, nor is its composition with the inclusion $N^1(\Hhat)\to N^1(\Hhat)^\sw$. In fact, any tiny neighborhood of any point of $\Hhat$, minus $\Hhat$, will map to a coarsely dense subset of $N^1(\Hhat)^\sw$. It is possible, with perhaps more effort than one would expect, to show that when one restricts to $\Bhatthick$ the map to $N^1(\Hhat)^\sw$ is coarsely Lipschitz, but for more than one reason we will choose to circumvent the required analysis for this. What is easy to prove is the following. 

\begin{lemma}\label{L:EasyCaseOfAngleLipschitz}
For every $C, \epsilon>0$ there exists a $D$ such that the following is true. Suppose $x, x'\in \Bhat$ are joined by a path in $\Bhat$ of length at most 1 that does not come within $\epsilon$ of $\Hhat$. Suppose that $p_{\Hhat}(x)$ and $p_{\Hhat}(x')$ have distance at most $C$ from the complement of all horoballs in $\Bhat$. Then $p_{N^1(\Hhat)}(x)$ and $p_{N^1(\Hhat)}(x')$ have distance at most $D$ in $N^1(\Hhat)^\sw$.

The same result applies if $x,x'$ are joined by a geodesic of any length that does not come within $\epsilon$ of $\Hhat$
\end{lemma}

Here we assume that the projections to $\Hhat$ are not deep in a horoball, and we are implicitly using that $N^1(\Hhat)$ is naturally a subset of $N^1(\Hhat)^\sw$, and we are following our  convention that whenever we refer to a horoball of $\Bhat$ we mean a lift of one of the carefully chosen horoballs $T_c$ of $B$, which are based at cusps of $\Gamma$ and are chosen in particular to be disjoint from each other. 

\begin{proof}
Pick $0<a<\min(\epsilon, r_\theta)$ such that in the set of points at most distance $C+2$ from the complement of the horoballs, disjoint hyperplanes have distance at least $2a$ from each other. Consider the projection of $x$ and $x'$ to the $a$ neighborhood of $\Hhat$. Since this neighborhood is convex, the projection of the path has length at most 1. 

Along this path, the $p_{N^1(\Hhat)}$ image can only change at uniformly bounded speed, because the path stays exactly $a$ away from $\Hhat$. Indeed, $p_{N^1(H)}$ is uniformly Lipschitz on the set of points exactly distance $a$ from $H$, and one can lift this using \cref{LemNormalBundleIsPullback}  to obtain the corresponding fact for $p_{N^1(\Hhat)}$ as long as one restricts to points only boundedly deep in horoballs. 

It follows that $p_{N^1(\Hhat)}(x)$ and $p_{N^1(\Hhat)}(x')$ have uniformly bounded distance from each other in $N^1(\Hhat)$, and \cref{L:N1inclusion} gives the  first result. 

The second result can be derived from the first using  \cref{L:CATBGI}.
%
%
\end{proof}

\subsection{A different projection defined on the boundary of a horoball.} 
The rest of this section requires the setup from \cref{S:CuspDomains}. The naive projection above rests on the most natural path from a point $x\in \Bhat- \Hhat$ to $\Hhat$, namely the geodesic to the closest point projection. There is sometimes a different option with subtle technical advantages, and we explore this now. 

Let $\That_{\hat{c}}$ be a horoball that $\Hhat$ enters, and let $x$ be a point in $\partial \That_{\hat{c}} - \Hhat$. There is some $i$ so that $\Hhat$ corresponds to a point $q_{\Hhat}$ in $\Chat_i-\Chatcirc_i$. Inside of $\prod_j \Chat_j$ we can define a path which starts at the image of $x$, is constant in all but the $i$-th coordinate, and is the geodesic to $q_{\Hhat}$ in the $i$-th coordinate. We then consider the horizontal lift of this path starting at $x$. It is a path in $\partial \That_{\hat{c}}$ from $x$ to a point of $\Hhat \cap \partial \That_{\hat{c}}.$

\cref{L:HinTcBasic} gives that the germ of such a path at its endpoint defines a point of $N^1(\Hhat)$. This is our alternative projection.

\begin{lemma}\label{L:NaiveProjectionReasonableNearHhat}
For all $\epsilon>0$ there exists $C_\epsilon>0$ such that for any point of $\partial \That_{\hat{c}}$ whose distance to $\Hhat$ is in $(0,\epsilon)$, the naive and alternative projections to $N^1(\Hhat)^\sw$ are within $C_\epsilon$ of each other.  
\end{lemma}

\begin{proof}
We first prove the result for some $\epsilon>0$, and then deduce from that that it holds for all $\epsilon>0$. 

The image in $\Chat_i$ of the path chosen above  consists of segments between singularities, and the final segment is the part of the path after it hits the penultimate singularity and before it gets to $q_{\Hhat}$. If $\epsilon$ is small enough, the whole path is equal to its final segment. 

The result for a sufficiently small $\epsilon$ follows by a compactness argument, since one can project these paths to $B$, where up to the group action we get a compact set of paths, and then use \cref{L:DerivativeEqualsProj}, keeping in mind \cref{L:N1inclusion}. 
%
%
%

Now, suppose the result is known for a small $\epsilon_0>0$, and let $\epsilon>\epsilon_0$ be arbitrary. \cref{L:EasyCaseOfAngleLipschitz} gives that the naive projection changes by only a bounded amount as one moves on $\partial \That_{\hat{c}}$ from distance $\epsilon$ to distance $\epsilon_0$ from $\Hhat$. The alternative projection does not change at all if one follows a horizontal lift of a geodesic in $\prod \Chat_i$. This gives the result.  
\end{proof}

\begin{remark}
It does not matter if one uses the $\tau$ path in the $i$-th coordinate instead of the geodesic, since they give the same germ at $\Hhat$. 
\end{remark}

The alternative projection has the following property.

\begin{proposition}\label{P:TreeAngleBGI}
There is a constant $F$ such that the following holds. 
Let $\That_{\hat{c}}$ be a horoball, let $\Chat_i$ be one of its factors, and let $\Hhat$ be a hyperplane corresponding to  a point $q_{\Hhat}$ in $\Chat_i$. Let $U$ be the angle domain corresponding to $\Hhat$. If $x,y\in \partial \That_{\hat{c}}- \Hhat$ and the geodesic between their images $x_i, y_i$ in $\Chat_i$ does not pass through $q_{\Hhat}$, then the alternative projections of $x$ and $y$ to $N^1(\Hhat)^\sw$ are within $F$ of each other. 
\end{proposition}

The proof requires some preparation.  Start by noting that a neighborhood of $q_{\Hhat}$ is isometric to a neighborhood of the special point in the metric completion of the universal cover of $\bR^2-\{(0,0)\}$.
The space of geodesics leaving $q_{\Hhat}$ is parametrized by $\bR$, and there is a natural notion of angle between two geodesics leaving $q_{\Hhat}$ valued in $[0, \infty)$.

\begin{lemma}\label{L:rhoAngleTree}
Suppose $w$ and $w'$ are elements of $N^1(\Hhat)$ lying over possibly different points of $\partial \That_{\hat{c}}$ that both define the same outgoing germ at $q_{\Hhat} \in \Chat_i$.  Then the distance between the images of $w$ and $w'$  in $N^1(\Hhat)^\sw$ is universally bounded. 
\end{lemma} 

\begin{proof}
By an approximation argument, it suffices to assume $w$ and $w'$ lie over points of the generic stratum of $\Hhat$. Join the basepoints of $w$ and $w'$ in $\Hhat\cap \partial \That_{\hat{c}} $ with a path $\alpha$ in the generic stratum of $\Hhat \cap \partial \That_{\hat{c}}$. 

For $s\geq 0$, let $\alpha^{(s)}$ denote the image of $\alpha$ under flow for time $s$ towards $\hat{c}$. Thus, the length of $\alpha^{(s)}$ goes to zero as $s$ goes to infinity. Let $w_s$ and $w_s'$ be the parallel translates of $w, w'$ for time $s$ along the geodesics towards $\hat{c}$. 

By definition, in $N^1(\Hhat)^\sw$, we know that $w_s$ is distance at most 1 from $w$ and that $w_s'$ is distance at most 1 from $w'$. So it suffices to show that $w_s$ and $w_s'$ are bounded distance apart in $N^1(\Hhat)^\sw$.

A slight variant of \cref{L:ConstInCn1IsAlmostParallel} gives that $w_s'$ and the parallel translate of $w_s$ along $\alpha^{(s)}$ are very close in $N^1(\Hhat)$ when $s$ is large.

One can consider the loop formed by $\alpha^{(s)}$ and the geodesic between its endpoints. One can approximate  $\alpha^{(s)}$ with a piecewise geodesic path with a fine mesh and close to the same length, and consider the geodesic polygon formed by this piecewise geodesic path and the geodesic between the endpoints of $\alpha^{(s)}$. Since this polygon has small perimeter, it can be subdivided into triangles whose total comparison area is small; compare to \cite[page 426]{BridsonHaefliger}. Because the swaddling data defined on $N^1(\Hhat)$ is $L$-bounded for some $L$ (by \cref{C:PullBackQuasiLine}), the total monodromy of the geodesic polygon is small, so by approximation we see that the parallel translate along $\alpha^{(s)}$ is close to the parallel translate along the geodesic between its endpoints. This gives the result. 
%
%
%
\end{proof}

\begin{proof}[Proof of \cref{P:TreeAngleBGI}]
Consider the geodesics from $q_{\Hhat}$ to $x_i$ and $y_i$. If the angle at $q_{\Hhat}$ between these two geodesics was at least $\pi$, their union would be a geodesic, contradicting the assumption. Thus we get that the angle at $q_{\Hhat}$ is less than $\pi$. A short argument using \cref{L:rhoAngleTree} then shows that the alternative projections are close. 
\end{proof}

\begin{corollary}\label{C:AltLip}
For all $\delta>0$ there exists $Q>0$ such that
this alternative projection is a $Q$-coarsely Lipschitz map from the set of points of $\partial \That_{\hat{c}}$ at least $\delta$ away from $\Hhat$, with its intrinsic path metric, to $N^1(\Hhat)^\sw$.
\end{corollary}

\subsection{A modification of the naive projection map.} 
We wish to modify the naive projection to $N^1(\Hhat)^\sw$ when the closest point  projection to $\Hhat$ lands in a horoball $\That_{\hat{c}}$. We will do this by using the alternative projection applied to the entry point to the horoball. The following helps us to understand the entry point. 

\begin{lemma}\label{L:EntryPoint}
There exists a universal constant $C$ such that the following holds. 
Let $x, y$ be points of $\Bhat$, and let $\That_{\hat{c}}$ be a horoball. Suppose $x\notin\That_{\hat{c}}$ and $y\in \That_{\hat{c}}$. Let $z$ be the point where the geodesic from $x$ to $y$ enters $\That_{\hat{c}}$. Then the distance from $z$ to the closest point projection of $x$ to $\That_{\hat{c}}$ is at most $C$. 
\end{lemma}
%
%

For a proof, compare to \cite[Lemma 2.3]{ParkkonenPaulin}, and note also that more precise bounds are possible.
%
%

\begin{lemma}\label{L:BoundaryHorosphere}
For all $R>0$ there exists an $R'>0$ such that for any $\That_{\hat{c}}$, if two points of $\partial \That_{\hat{c}}$ are within distance $R$ using the $\Bhat$ metric, then they are within distance $R'$ using the intrinsic path metric on $\partial \That_{\hat{c}}$.
\end{lemma}

We leave it to the reader to prove \cref{L:BoundaryHorosphere} by starting with a geodesic in $\Bhat$ and pushing it radially outwards from $\hat{c}$ until it gives a path on $\partial \That_{\hat{c}}$. 

We can now give our definition of $$\pi_{N^1(\Hhat)} : \Bhatthick \to  N^1(\Hhat)^\sw.$$
Fix an $\epsilon$ large depending on \cref{L:EntryPoint} and \cref{L:BoundaryHorosphere} in ways that will become clear shortly.  
Consider $x\in \Bhatthick$. If the closest point projection of $x$ to $\Hhat$ is not in a horoball, we define $\pi_{N^1(\Hhat)}(x)$ to be the naive projection. Similarly, if the projection is in a horoball, and the entry-point is within $\epsilon$ of $\Hhat$, we define $\pi_{N^1(\Hhat)}(x)$ to be the naive projection. Otherwise, we look at the entry point $z$ where the geodesic from $x$ to $p_{\Hhat}(x)$ enters $\That_{\hat{c}}$. We then define $\pi_{N^1(\Hhat)}(x)$ to be the alternative projection of $z$. 

\begin{remark}\label{R:extensionangle}
Continuing \cref{R:extension}, we note that the true domain $\Bhat^{N^1(\Hhat)}$ of this map is the complement in $\Bhat$ of the open $r_\theta$ neighborhood of $\Hhat$ union all the interiors of all horoballs that $\Hhat$ enters. (Here $r_\theta$ could be replaced with anything smaller, but some positive constant is required for \cref{L:AngleLipschitz}. As a warning, we emphasize that $\Bhat^{N^1(\Hhat)}$ is not convex, and as a rule arguments with angle domains require extra care.) 
\end{remark}

It remains to show the following. 

\begin{lemma}\label{L:AngleLipschitz}
$\pi_{N^1(\Hhat)}:\Bhat^{N^1(\Hhat)} \to \cC(N^1(\Hhat)) $ is coarsely Lipschitz.
\end{lemma}

\begin{proof}
Consider points $x,y$ that can be joined by a path in $\Bhat^{N^1(\Hhat)}$ of length at most 1; it suffices to give a uniform bound on the distance between their images. 

If the closest point projections $p_{\Hhat}(x)$ and $p_{\Hhat}(y)$ are not in a horoball, this follows from \cref{L:EasyCaseOfAngleLipschitz}. If  one of the closest point projections is not in a horoball and the other is in a horoball, it again follows from \cref{L:EasyCaseOfAngleLipschitz}, given that the relevant entry point lies close to $\Hhat$.

So assume that both of $p_{\Hhat}(x)$ and $p_{\Hhat}(y)$ are in a horoball $\That_{\hat{c}}$. Note that $p_{\That_{\hat{c}}}(x)$ and $p_{\That_{\hat{c}}}(y)$ are bounded distance from each other, and \cref{L:EntryPoint} gives that these points are bounded distance from associated entry points. 

If these points are reasonably close to $\Hhat$, the result follows from \cref{L:NaiveProjectionReasonableNearHhat} and \cref{L:EasyCaseOfAngleLipschitz}. If these points are not close to $\Hhat$ it follows from \cref{C:AltLip} and \cref{L:BoundaryHorosphere}.
\end{proof}

\subsection{Results for later use} 

\begin{lemma}\label{L:PreUniqueness}
For any $\Vhat$ as in \cref{SS:Uniformity} and any $C>0$ there exists a $D>0$ such that the following holds. 
Suppose $y,z\in \Bhatthick \cap \Vhat$. Suppose also that $d_U(\pi_U(y), \pi_U(z))<C$ for all angle domains $U$ associated to hyperplanes passing through $\Vhat$. Then the distance between $y$ and $z$ in the intrinsic $\Bhatthick$ metric is at most $D$. 
\end{lemma}

\begin{proof}
 On an appropriate pre-compact fundamental domain for the $\bZ^k$ action on the set of points in $\Vhat$ at least $r_\theta$ away from each hyperplane, the $\pi_U$ are uniformly coarsely constant by  \cref{L:AngleLipschitz}. The result now follows from  \cref{L:StabCobounded} and a compactness argument.
\end{proof}

The following is an immediate consequence of \cref{L:NaiveProjectionReasonableNearHhat} and the definition of $\pi_U$. 

\begin{corollary}\label{C:AngleProjOnHoroBoundary}
Let $U$ be an angle domain, let $\Hhat$ be the corresponding hyperplane, and suppose $\Hhat$ enters a horoball $\That_{\chat}$. For any $x\in \partial \That_{\chat}\cap \Bhat^U$, the alternative projection of $x$ to $\cC(U)$ is uniformly close to $\pi_U(x)$. 
\end{corollary}

For later use we also record the following consequence of \cref{P:TreeAngleBGI} and the definition of $\pi_{N^1(\Hhat)}$.

\begin{corollary}\label{C:TreeAngleBGI}
There exists $C, D_0>0$ such that for any horoball $\That_{\hat{c}}$ and any $x,y\in \partial \That_{\hat{c}}$, the number of hyperplanes $\Hhat$ intersecting $\That_{\hat{c}}$ whose angle domain $U$ does not have $d_U(\pi_U(x), \pi_U(y))< D_0$ is at most $C d_{\partial \That_{\hat{c}}}(x,y)+C$, where $d_{\partial \That_{\hat{c}}}$ is the intrinsic path metric. 
\end{corollary}

We define ``does not have $d_U(\pi_U(x), \pi_U(y))< D_0$'' to include the possibility that $\pi_U(x)$ or $\pi_U(y)$ are not defined, even when we extend the definition of $\pi_U$ to all of $\partial \That_{\hat{c}} - \Hhat$ in the natural way using the alternative projection.

\begin{proof}
We claim that $d_U(\pi_U(x), \pi_U(y))< D_0$ for all angle domains $U$ except those corresponding to singular points on the geodesic joining the images of $x$ and $y$ in one of the $\Chat_i$. Indeed, suppose $U$ is such an angle domain, corresponding to a hyperplane $\Hhat$. By \cref{C:AngleProjOnHoroBoundary}, it suffices to prove the result for the alternative projection instead of $\pi_U$, so \cref{P:TreeAngleBGI} gives the claim. 

The result then follows by uniform separation of the singular points in $\Chat_i$. 
\end{proof}

\section{Nesting and orthogonality }\label{S:NTO}

We now have a set of domains $\IndexSmall$ consisting of stratum domains, angle domains, and (center and tree)  cusp domains, as well as associated maps $\pi_U : \Bhatthick \to \cC U$ to Gromov hyperbolic spaces. Much of the work so far can be summarized  using \cref{D:preHHS} by noting that we have proven the following. 

\begin{proposition}\label{P:pre}
$(\Bhatthick, \IndexSmall, \pi_\bullet)$ is a pre-HHS.
\end{proposition}

\begin{proof}
For stratum domains, hyperbolicity was checked in  \cref{L:StratDomainsHyp} and coarse Lipschitzness  in the discussion immediately before. For center domains, hyperbolicity was checked in \cref{CenterQuasiLine}, and coarse Lipschitzness  at the start of \cref{SS:CenterDomains}. For tree domains, hyperbolicity was checked in \cref{LemTreeDomainsGromovHyperbolic}, and coarse Lipschitzness  in the discussion immediately before. For angle domains, hyperbolicity was checked in \cref{C:AngleQuasiLine}, and coarse Lipschitzness in \cref{L:AngleLipschitz}.
\end{proof}

In this section we define the nesting and orthogonality relations and prove that $\IndexSmall$ is a nearly valid index set. Before proceeding, it is helpful to recall the notation for the different types of domains in $\IndexSmall$, which we have set up to allow more concise descriptions of these relations. 

\bold{Stratum domains} are associated to a stratum $\Shatcirc$ of $\Bhat$. These strata are typically not closed subsets of $\Bhat$. We do not define stratum domains for 0-dimensional strata. 

\bold{Angle domains} are associated to the unit normal bundle $N^1(\Hhat)$ of a hyperplane $\Hhat$. 

\bold{Center domains} are associated to horoballs $\That_{\hat{c}}$.

\bold{Tree domains} are associated to factors of the boundary of the horoball modulo the center direction. For each such factor, we formally define the associated domain to be the set of all hyperplanes that enter the horoball and are constant in that factor.

\begin{warning}\label{W:warning1}
It is crucial to keep in mind that the terms ``nested'' and ``orthogonal'' describe the map $\pi_U \times \pi_V$ and not  geometric objects in $\Bhat$ associated to $U$ and $V$. See  \cref{W:warning2} for examples of how otherwise these terms can cause confusion.
\end{warning}

\subsection{The definition of nesting and orthogonality}

The nesting relation is defined as follows. 
\begin{enumerate}
\item A stratum domain is nested in a second stratum domain if the first stratum is contained in the closure of the second. 
\item An angle domain is nested in a stratum domain if the closure of the stratum intersects but is not contained in the hyperplane associated to the angle domain. 
\item A center domain is nested in a stratum domain if the stratum enters the associated horoball. 
\item A tree domain is nested in a stratum domain if the stratum enters the associated horoball and is not constant in the associated factor.
\item An angle domain is nested in a tree domain if the hyperplane associated to the angle domain enters the horoball associated to the tree domain and is constant in the factor associated to the tree domain. 
\end{enumerate}

\begin{remark}\label{R:limitednesting}
Angle domains and center domains are always nest-minimal, meaning they have nothing properly nested in them. Only angle domains can be nested in tree domains. Tree domains are never nest minimal.
\end{remark}

 \begin{remark}
 Strata are in general not closed and are in general not cobounded. But strata can in exceptional cases be cocompact, and a stratum domain is nest minimal if and only if  the stratum is cocompact. (Keep in mind that $0$-dimensional strata do not give domains, so in particular do not give nest-minimal domains.)
 %
 %
 %
 \end{remark}

The orthogonality relation is defined as follows. 
\begin{enumerate}
\item A stratum domain is orthogonal to an angle domain if the stratum is contained in the associated hyperplane. 
\item Distinct angle domains are orthogonal if their associated hyperplanes intersect. 
\item An angle domain is orthogonal to a center domain if the associated hyperplane enters the associated horoball. 
\item An angle domain is orthogonal to a tree domain if the associated hyperplane enters the associated horoball and is not constant in the associated factor. 
\item A center domain is orthogonal to a tree domain if they are both associated to the same horoball.
\item Distinct tree domains are orthogonal if they are both associated to the same horoball.
\end{enumerate}

As always, two domains are asynchronous if neither is nested in the other and they are not orthogonal. The relations are summarized in \cref{fig:rels}, where a blank entry indicates the domains cannot satisfy the relation. 

\begin{figure}[h]
\begin{center}
\setlength{\extrarowheight}{20pt}
\resizebox{\textwidth}{!}{\begin{tabular}{c|c|c|c|c}
    \spc{top $\perp$  side}{top $\sqsubsetneq$ side}{top $\pitchfork$ side} & 
    \makecell{stratum domain \\ $\Shatcirc'$} &  
    \makecell{angle domain \\ $N^1(\Hhat')$} & 
    \makecell{center domain\\ $\That_{\hat{c}'}$} & 
    \makecell{tree domain $V'$ \\ $\{\text{hyperplanes constant}$ \\ $\text{in fixed factor of } \That_{\hat{c}'}\}$ } \\[20pt]\cline{1-5}
    %
    %
    %
    \makecell{stratum domain \\ $\Shatcirc$} & 
    \spc{}{$\Shat'\subsetneq \Shat$}{$\Shat'\not\subset \Shat \,\,\text{\&}\,\, \Shat\not\subset \Shat' $} & 
    \spc{$\Shat\subset \Hhat'$}{$\Shat \not\subset \Hhat' \,\,\text{\&}\,\, \Shat \cap\Hhat'\neq \emptyset$}{$\Shat \cap\Hhat'= \emptyset$} & 
    \spc{}{$\Shat\cap \That_{\hat{c}'} \neq \emptyset$}{$\Shat\cap \That_{\hat{c}'} = \emptyset$} & 
    \spc{}{$\Shat\cap \That_{\hat{c}'}\neq \emptyset\,\,\text{\&}\,\,\Shat \not\subset \bigcup V'$}{$\Shat\cap \That_{\hat{c}'}= \emptyset$\,\,\text{or}\,\, $\Shat\subset \bigcup V'$}  \\[20pt]\cline{1-5}
    %
    %
    %
    \makecell{angle domain \\ $N^1(\Hhat)$} & 
    \spc{$\Shat'\subset \Hhat$}{}{$\Hhat \cap\Shat'= \emptyset$} & 
    \spc{$\Hhat\cap \Hhat' \neq \emptyset \,\,\text{\&}\,\, \Hhat\neq \Hhat'$}{}{$\Hhat\cap \Hhat' = \emptyset$} & 
    \spc{$\Hhat \cap \That_{\hat{c}'}\neq \emptyset$}{}{$\Hhat \cap \That_{\hat{c}'}= \emptyset$} & 
    \spc{$ \Hhat \cap \That_{\hat{c}'}\neq\emptyset\,\,\text{\&}\,\, \Hhat \notin V'$}{}{$\Hhat \cap \That_{\hat{c}'}= \emptyset$} \\[20pt]\cline{1-5}
    %
    %
    %
    \makecell{center domain\\ $\That_{\hat{c}}$} &
    \spc{}{}{$\Shat'\cap \That_{\hat{c}} = \emptyset$} & 
    \spc{$\Hhat'\cap \That_{\hat{c}}\neq \emptyset$}{}{$\Hhat'\cap \That_{\hat{c}}= \emptyset$} & 
    \spc{}{}{$\hat{c}\neq \hat{c}'$}  & 
    \spc{$\hat{c}=\hat{c}'$}{}{$\hat{c}\neq \hat{c}'$} \\[20pt]\cline{1-5}
    %
    %
    %
    \makecell{tree domain $V$ \\ $\{\text{hyperplanes constant}$\\$\text{in fixed factor of } \That_{\hat{c}}\}$ }& 
    \spc{}{}{$\Shat'\cap \That_{\hat{c}}= \emptyset$\,\,\text{or}\,\, $\Shat'\subset \bigcup V$} & 
    \spc{$ \Hhat' \cap \That_{\hat{c}}\neq \emptyset \,\,\text{\&}\,\, \Hhat' \notin V$}{$\Hhat'\in V$}{$\Hhat' \cap \That_{\hat{c}}= \emptyset$} & 
    \spc{$\hat{c}=\hat{c}'$}{}{$\hat{c}\neq \hat{c}'$} & 
    \spc{$\hat{c}=\hat{c}' \,\,\text{\&}\,\, V\neq V'$}{}{$\hat{c}\neq \hat{c}'$}  \\[20pt]
\end{tabular}}
\end{center}
\caption{The definition of orthogonal, properly nested, and asynchronous.}
     \label[figure]{fig:rels}
\end{figure}

\begin{warning}\label{W:warning2}
We now repeat \cref{W:warning1} in more detail to emphasize that one may not think of the words ``nesting'' and ``orthogonality'' as naively describing the geometric objects associated to the domains. For example, nesting doesn't always correspond to containment and orthogonality doesn't always correspond to perpendicular intersections. Note in particular that an angle domain defined by a hyperplane is orthogonal to a stratum domain when the hyperplane contains the stratum. 
\end{warning}

\subsection{First observations on the relations} 

\begin{lemma}\label{L:BoundedChains}
Every chain $U_1\sqsubsetneq U_2 \sqsubsetneq \cdots \sqsubsetneq U_\ell$ that isn't properly contained in a larger chain has the following properties. 
\begin{enumerate}
\item $U_\ell=\Bhatcirc$.
\item The $U_i, i\geq 3$ are stratum domains.
\item $U_2$ is a stratum domain or a tree domain. 
\item If $U_2$ is a tree domain then $U_1$ is an angle domain. 
\end{enumerate}
\end{lemma}

\begin{proof}
The first claim is true because every domain is nested in $\Bhatcirc$. The other claims follow from \cref{R:limitednesting}.
\end{proof}

\begin{corollary}\label{C:PO}
Nesting is a partial order. 
\end{corollary}

\begin{proof}
Suppose $U\sqsubsetneq V$ and $V \sqsubsetneq W$. We need to show $U\sqsubsetneq W$. 

Note that $W$ must be a stratum domain, and $V$ must be a stratum domain or a tree domain. We leave the case where $V$ is a stratum domain to the reader and do the case when $V$ is a tree domain. In this case $U$ is an angle domain. We know 
\begin{enumerate}
\item the hyperplane associated to $U$ enters the horoball associated to $V$ and is constant in the factor associated to $V$, and 
\item the stratum associated to $W$ enters the horoball associated to $V$ and is not constant in the factor associated to $V$. 
\end{enumerate}
Inside of a horoball, hyperplanes are defined by setting a single coordinate equal to a constant, and the closures of strata are defined by setting some number of coordinates equal to a constant.  Thus we can conclude that the stratum closure $\overline{W}$ and the hyperplane associated to $U$ intersect and the former is not contained in the latter. Hence $U\sqsubsetneq W$ as desired. 
\end{proof}

\begin{corollary}\label{C:BoundedChains}
The maximal size $\ell$ of a chain $U_1\sqsubsetneq U_2 \sqsubsetneq \cdots \sqsubsetneq U_\ell$ is  at most $n+1$, and is equal to $n+1$ if and only if $\Gamma$ is non-uniform or there is a 0-dimensional stratum.
\end{corollary}

\begin{proof}
In the non-uniform case,
 \cref{CuspLimit} gives that there exists a 1-dimensional stratum $U_1$ that enters a horoball. We can thus find a center domain $U_0$ and stratum domains $U_1, \ldots, U_n$ so 
$$U_0\sqsubsetneq U_1 \sqsubsetneq \cdots \sqsubsetneq U_n,$$
proving $\ell\geq n+1$. 

If there is a 0-dimensional stratum, it is defined as the intersection of $n$ distinct hyperplanes, say $\Hhat_1, \Hhat_2, \ldots, \Hhat_n$. For $i=1, \ldots, n$ we can let $U_i$ be the generic stratum of $\cap_{j=1}^{n-i} \Hhat_j$, and we can let $U_0$ be the angle domain $N^1(\Hhat_n)$, again proving $\ell\geq n+1$. 

The remaining claims follow from \cref{L:BoundedChains}.
\end{proof}

The following statement is slightly complicated by our choice not to define stratum domains for $0$-dimensional strata. 

\begin{lemma}\label{L:MaxOrthogonalSets}
A set of domains is an orthogonal set  not contained in a larger orthogonal set if and only if it is one of the following. 
\begin{enumerate}
\item A collection of $n$ angle domains corresponding to $n$ distinct hyperplanes through a point. 
\item A stratum domain and the collection of angle domains corresponding to hyperplanes containing it. 
\item The $n$ cusp domains associated to a horoball, as well as the same with any subset of the tree domains each replaced by an angle domain nested in it. 
\end{enumerate}
\end{lemma}

\begin{proof}
A stratum domain is only orthogonal to angle domains corresponding to hyperplanes containing the stratum, and all these angle domains happen to be orthogonal to each other. 

Similarly a center domain is only orthogonal to the tree domains associated to the same horoball and the angle domains nested in these tree domains. Again these tree domains happen to be orthogonal to each other, and the same is true if some of them are replaced with an angle domain nested in them. 

This proves the claim if the tuple contains a stratum domain or a center domain. Anything orthogonal to a tree domain is orthogonal to the associated center domain, so this also proves it unless the tuple consists entirely of angle domains. Any orthogonal set of angle domains must correspond to hyperplanes with a common intersection, by \cref{L:TripleIntersection,R:OrthIntersection}, and one can add the generic stratum of this intersection to the orthogonal set as long as this intersection is not a point. 
\end{proof}

\begin{corollary}\label{C:MaxSizeOrth}
The maximum size of an orthogonal set of domains is at most $n$. It is equal to $n$ if and only if $\Gamma$ is non-uniform or there is a 0-dimensional stratum or a 1-dimensional stratum or $n=1$.  
\end{corollary}

\begin{remark}\label{R:ForABD}
Suppose $n>1$. Then the only domain not orthogonal to any other is the nest maximal domain. 
\end{remark}
%
%
%

\subsection{Nearly valid} 

We require two more lemmas before we can conclude that $\IndexSmall$ is a nearly valid index set.

\begin{lemma}\label{L:Axiom3a}
If $V \sqsubseteq W$ and $W\perp U$ then $V \perp U$.  
\end{lemma}

\begin{proof}
It suffices to assume $V\neq W$. Since $W$ has something properly nested in it, it must be a stratum domain or a tree domain. 

\bold{Case 1: $W$ is a stratum domain.} Then $U$ is an angle domain, and we leave it to the reader to check the result for each of the four possible types of $V$. 

\bold{Case 2: $W$ is a tree domain.} In this case $U$ is a center domain, or a tree domain corresponding to one of the other factors of the horoball, or an angle domain nested in one of those other tree domains, and $V$ is an angle domain. Again we leave the details to the reader. 
\end{proof}

\begin{lemma}\label{L:WeakContainers}
If $V \sqsubseteq W$ then the set $\cP_{V,W}$ of domains orthogonal to $V$ and nested in $W$  has at most finitely many nest maximal elements, all of which are orthogonal to each other and to $V$. 
\end{lemma}

\begin{proof}
We first prove this when $W=\Bhatcirc$ is the nest maximal element. In this case: 
\begin{enumerate}
\item If $V$ is a stratum domain, $\cP_{V,W}$ will be an orthogonal tuple of angle domains. 
\item\label{CASE2} If $V$ is a cusp domain, the elements of $\cP_{V,W}$ are the $n-1$ other cusp domains and, for those that are tree domains, the angle domains nested in them. So the maximal elements of $\cP_{V,W}$ together with $V$ itself will be the collection of $n$ orthogonal cusp domains associated to a single horoball. 
\item If $V$ is an angle domain, then $V$ is orthogonal only to the generic stratum of the hyperplane $V$ is associated to and the domains nested in that stratum. So $\cP_{V,W}$ has only a single nest maximal element, and it is orthogonal to $V$. 
\end{enumerate}

More generally, we must consider the case when $W$ is an arbitrary stratum domain and when it is a tree domain. We leave the stratum case to the reader.
The case when $W$ is a tree domain doesn't really occur, since only angle domains can be nested in a tree domain, and two angle domains nested in the same tree domain are never orthogonal: so $\cP_{V,W}$ is always empty in this case. 
\end{proof}

We now have near validity of the index set. 

\begin{proposition}\label{P:av}
$\IndexSmall$ with the relations defined above is a nearly valid index set. 
\end{proposition}

\begin{proof}
Orthogonality is symmetric and anti-reflexive by definition; nesting is a partial order by \cref{C:PO}, with a unique maximal element by \cref{L:BoundedChains}. The finite height axiom is true by \cref{C:BoundedChains}; and the coherence axiom by \cref{L:Axiom3a}. The weak containers axiom is true by \cref{L:WeakContainers} and \cref{C:MaxSizeOrth}. 
\end{proof}

\section{Active regions }\label{S:Boundedness}

In this section we establish the density, bounded projections, and structured redundancy assumptions of \cref{P:Criterion}. 

We define the active regions $\cA_U$ for \cref{P:Criterion} as follows. To start, define closed convex subsets $K_U$ of $\Bhat$ as follows. 
\begin{enumerate}
\item If $U$ is a stratum domain,  $K_U$ is the associated stratum closure. 
\item If $U$ is an angle domain,  $K_U$ is the associated hyperplane. 
\item If $U$ is a cusp domain,  $K_U$ is the associated horoball. 
\end{enumerate}
We then define $\cB_U$ to be the $10 n r_\theta$-neighborhood of $K_U$, and we define $\cA_U=\cB_U\cap \Bhatthick$. We also recall from \cref{R:extension,R:extensionstrata,R:extensiontree,R:extensioncenter,R:extensionangle} that $\Bhat^U$ satisfies $\Bhatthick \subset \Bhat^U \subset \Bhat$ and is the domain of $\pi_U$. 

\subsection{Density} 

We start by observing the following. 

\begin{lemma}\label{L:Density} 
The density assumption of \cref{P:Criterion} holds. 
\end{lemma}

\begin{proof}
It suffices to check this for maximal orthogonal sets, which \cref{L:MaxOrthogonalSets} gives are one of the following.
\begin{enumerate}
\item A collection of $n$ angle domains corresponding to $n$ distinct hyperplanes through a point. 
\item A stratum domain and the collection of angle domains of codimension 1 strata containing it. 
\item The $n$ cusp domains associated to a horoball, as well as the same with any subset of the tree domains each replaced by an angle domain nested in it. 
\end{enumerate}
In the first case, use \cref{R:NudgeToThick} to find a thick point near the point of intersection of the $n$-hyperplanes and then apply \cref{L:StabCobounded}. The other cases follow immediately by additionally applying the discussion in \cref{S:StratumDomains,S:CuspDomains}.
\end{proof}

\subsection{Bounded projections}
We now address the bounded projection assumption. We actually prove something slightly stronger: Rather than using $\cA_U = \cB_U \cap \Bhatthick$, we use the larger set $\cB_U \cap \Bhat^V$. 

\begin{proposition}\label{P:BoundedProjections}
 If $U \pitchfork V$ or $U\sqsubsetneq V$,  $\diam \pi_V(\cB_U \cap \Bhat^V)$ is uniformly bounded.  
\end{proposition}

To prove \cref{P:BoundedProjections}, we need the following, which combines a restatement of \cref{L:CATBGI} and an immediate consequence of \cref{L:CATBGI}. 

\begin{corollary}\label{C:CATBGI}
For all $\epsilon>0$ there exists a $\delta>0$ such that the following holds: Suppose $C$ is a closed convex subset of a CAT($-1$) space $X$ and $p_C$ is projection to $C$. Suppose $A$ is a convex set. 
\begin{enumerate}
\item\label{CATBGI1} If $A$ is disjoint from  the $\epsilon$-neighborhood of $C$, then $p_C(A)$ has diameter at most $\delta$, both in the ambient metric and in the induced path metric on the boundary of $C$. 
\item\label{CATBGI2} If $A$ is not disjoint from the $\epsilon$-neighborhood of $C$, then $p_C(A)$ is contained in the $\delta$-neighborhood of the intersection of $A$ and the $\epsilon$-neighborhood of $C$. 
\end{enumerate}
\end{corollary}

Throughout this section, angle domains need to be handled with special care, due to the modified nature of their projections. So we start with the following. 

\begin{lemma}\label{L:AnglePreBGI}
For all $\epsilon>0$ there exists $C>0$ such that the following holds. 
Suppose $x,y\in \Bhat^U$, where $U$ is an angle domain with corresponding hyperplane $\Hhat$, and the geodesic from $x$ to $y$ does not come within distance $\epsilon$ of $\Hhat$ and does not enter any horoball that $\Hhat$ enters. Then $d_U(\pi_U(x),\pi_U(y))\leq C$. 
\end{lemma}

\begin{proof}
First suppose $\pi_U(x)$ and $\pi_U(y)$ are equal to the naive projections; so if the nearest point projections of $x$ and $y$ lie in a horoball that $\Hhat$ enters, the entry point is reasonably close to $\Hhat$.  Then
\cref{L:EasyCaseOfAngleLipschitz} gives the result. 

Suppose the nearest point projection of $x$ to $\Hhat$ lies in a horoball $\That_{\hat{c}}$ that $\Hhat$ enters. \cref{C:CATBGI} gives that the nearest point projections of $x$ and $y$ to $\Hhat$ are close to each other, and a small elaboration of \cref{C:CATBGI} gives the same for the projections to $\That_{\hat{c}}$. 
If these closest point projections to $\That_{\hat{c}}$ are far from $\Hhat$, the result follows from the fact that $\pi_U$ is coarsely Lipschitz (\cref{L:AngleLipschitz}), because \cref{L:EntryPoint} gives that the entry points are close to each other.

\cref{L:NaiveProjectionReasonableNearHhat} gives that if  $z\in \partial\That_{\hat{c}}$ has bounded distance from $\Hhat$ then the naive and modified projections  of $z$ are bounded distance apart. Thus, in the remaining case it suffices to use the naive projection, and again  \cref{L:EasyCaseOfAngleLipschitz} gives the result. 
\end{proof}

\begin{proof}[Proof of \cref{P:BoundedProjections}]
For convenience, we reproduce a modified form of \cref{fig:rels} here as \cref{fig:relsbounded}, which indicates which cases must be handled.
\begin{figure}[h]
\begin{center}
\setlength{\extrarowheight}{15pt}
\resizebox{\textwidth}{!}{\begin{tabular}{c|c|c|c|c}
    \spcbounded{top $\perp$  side}{top $\sqsubsetneq$ side}{top $\pitchfork$ side} & 
    \makecell{stratum domain \\ $\Shatcirc'$} &  
    \makecell{angle domain \\ $N^1(\Hhat')$} & 
    \makecell{center domain\\ $\That_{\hat{c}'}$} & 
    \makecell{tree domain $W'$ \\ $\{\text{hyperplanes constant}$ \\ $\text{in fixed factor of } \That_{\hat{c}'}\}$ } \\[10pt]\cline{1-5}
    %
    %
    %
    \makecell{stratum domain \\ $\Shatcirc$} & 
    \spcbounded{}{\color{DarkOrchid}$\Shat'\subsetneq \Shat$}{\color{DarkOrchid}$\Shat'\not\subset \Shat \,\,\text{\&}\,\, \Shat\not\subset \Shat' $} & 
    \spcbounded{$\Shat\subset \Hhat'$}{\color{DarkOrchid}$\Shat \not\subset \Hhat' \,\,\text{\&}\,\, \Shat \cap\Hhat'\neq \emptyset$}{\color{ForestGreen}$\Shat \cap\Hhat'= \emptyset$} & 
    \spcbounded{}{\color{DarkOrchid}$\Shat\cap \That_{\hat{c}'} \neq \emptyset$}{\color{ForestGreen}$\Shat\cap \That_{\hat{c}'} = \emptyset$} & 
    \spcbounded{}{\color{DarkOrchid}$\Shat\cap \That_{\hat{c}'}\neq \emptyset\,\,\text{\&}\,\,\Shat \not\subset \bigcup W'$}{\color{ForestGreen}$\Shat\cap \That_{\hat{c}'}= \emptyset$\color{black}\,\,\text{or}\,\, \color{DarkOrchid}$\Shat\subset \bigcup W'$}  \\[10pt]\cline{1-5}
    %
    %
    %
    \makecell{angle domain \\ $N^1(\Hhat)$} & 
    \spcbounded{$\Shat'\subset \Hhat$}{}{\color{ForestGreen}$\Hhat \cap\Shat'= \emptyset$\color{black}} & 
    \spcbounded{$\Hhat\cap \Hhat' \neq \emptyset \,\,\text{\&}\,\, \Hhat\neq \Hhat'$}{}{\color{ForestGreen}$\Hhat\cap \Hhat' = \emptyset$} & 
    \spcbounded{$\Hhat \cap \That_{\hat{c}'}\neq \emptyset$}{}{\color{ForestGreen}$\Hhat \cap \That_{\hat{c}'}= \emptyset$} & 
    \spcbounded{$ \Hhat \cap \That_{\hat{c}'}\neq\emptyset\,\,\text{\&}\,\, \Hhat \notin W'$}{}{\color{ForestGreen}$\Hhat \cap \That_{\hat{c}'}= \emptyset$} \\[10pt]\cline{1-5}
    %
    %
    %
    \makecell{center domain\\ $\That_{\hat{c}}$} &
    \spcbounded{}{}{\color{ForestGreen}$\Shat'\cap \That_{\hat{c}} = \emptyset$} & 
    \spcbounded{\color{ForestGreen}$\Hhat'\cap \That_{\hat{c}}\neq \emptyset$}{}{\color{ForestGreen}$\Hhat'\cap \That_{\hat{c}}= \emptyset$} & 
    \spcbounded{}{}{\color{ForestGreen}$\hat{c}\neq \hat{c}'$}  & 
    \spcbounded{\color{ForestGreen}$\hat{c}=\hat{c}'$}{}{\color{ForestGreen}$\hat{c}\neq \hat{c}'$} \\[10pt]\cline{1-5}
    %
    %
    %
    \makecell{tree domain $W$ \\ $\{\text{hyperplanes constant}$\\$\text{in fixed factor of } \That_{\hat{c}}\}$ }& 
    \spcbounded{}{}{\color{ForestGreen}$\Shat'\cap \That_{\hat{c}}= \emptyset$\color{black}\,\,\text{or}\,\, \color{DarkOrchid} $\Shat'\subset \bigcup W$} & 
    \spcbounded{$ \Hhat' \cap \That_{\hat{c}}\neq \emptyset \,\,\text{\&}\,\, \Hhat' \notin W$}{\color{DarkOrchid}$\Hhat'\in W$}{\color{ForestGreen}$\Hhat' \cap \That_{\hat{c}}= \emptyset$} & 
    \spcbounded{$\hat{c}=\hat{c}'$}{}{\color{ForestGreen}$\hat{c}\neq \hat{c}'$} & 
    \spcbounded{$\hat{c}=\hat{c}' \,\,\text{\&}\,\, W\neq W'$}{}{\color{ForestGreen}$\hat{c}\neq \hat{c}'$}  \\[10pt]
\end{tabular}}
\end{center}
\caption{The definition of properly nested and asynchronous.}
     \label[figure]{fig:relsbounded}
\end{figure}
We consider cases where $V$ is a domain on the side and $U$ is a domain on the top and the conditions given in \cref{fig:relsbounded} hold, and we must show $\pi_V(\cB_U \cap \Bhat^V)$ is bounded. We do this in four cases, each of which starts off by establishing this claim. 

\bold{Claim:} Every point of $\cB_U \cap \Bhat^V$ has uniformly bounded distance in the intrinsic $\Bhat^V$ path metric to a point of $K_U \cap \Bhat^V$.

\bold{Case 1: $V$ is a stratum domain.} Let $\Shat$ be the associated stratum closure.  In this case $\Bhat^V=\Bhat$, and $K_U$ is one of: a horoball; or a stratum closure not containing $\Shat$. (Hyperplanes are stratum closures.) Since $\Bhat^V=\Bhat$ the claim is immediate in this case. 

By the claim, it suffices to show that $\pi_V(K_U)$ is uniformly bounded. Here $\pi_V$ is the composition of closest point projection to $\Shat$ followed by inclusion into the electrification. \cref{C:CATBGI} and \cref{R:HoroballCapStratProj,R:OrthIntersection} give that the closest point projection of $K_U$ to $\Shat$ is either uniformly bounded or uniformly close to being contained in a smaller stratum closure or a horoball intersected with $\Shat$. 
%
%
%
%
Since these regions of $\Shat$ are contained in regions that are electrified in $\cC(V)$,  it follows that $\pi_V(K_U)$ is uniformly bounded. 

\bold{Case 2: $V$ is a center domain.} Let $\That_{\chat}$ be the associated horoball.  In this case $\Bhat^V = \Bhat - \int(\That_{\chat})$, and $K_U$ is one of: a stratum closure disjoint from $\That_{\chat}$;  or a horoball disjoint from $\That_{\chat}$. The claim is immediate, since \cref{R:constants} provides more than enough separation from $\That_{\chat}$.

Keeping in mind \cref{R:constants}, \cref{C:CATBGI} gives that the closest point projection of $K_U$ to $\That_{\chat}$ is bounded. This gives the result because $\pi_V$ factors through that projection and $\pi_V$ is coarsely Lipschitz. 

\bold{Case 3: $V$ is a tree domain.} Let $\That_{\chat}$ be the associated horoball. In this case $\pi_V$ is defined on the complement of the interior of $\That_{\chat}$, and $K_U$ is one of: a  horoball disjoint from $\That_{\chat}$; a stratum closure disjoint from $\That_{\chat}$; or a stratum closure which is constant in the factor $\Chat_i$ associated to $V$.

For both the claim and the result, the first two of the possibilities for $K_U$ follow as in Case 2, so let us  consider the case of a stratum closure $\Shat'$ which is constant in $\Chat_i$. For the claim, this last case is not much harder, except that for points of $\cB_U \cap \Bhat^V$ very close to $\That_{\hat{c}}$ the geodesic to $K_U$, which is very short, may require a slight radial push outwards if it enters $\That_{\hat{c}}$. 

By \cref{R:HoroballCapStratProj}, the closest point projection to $\That_{\hat{c}}$ of $\Shat'-\int(\That_{\hat{c}})$ is equal to $\Shat' \cap \partial \That_{\hat{c}}$. Since $\Shat' \cap \partial \That_{\hat{c}}$ is constant in the relevant factor $\Chat_i$ the result follows. 

\bold{Case 4: $V$ is an angle domain.} Let $\Hhat$ be the associated hyperplane. In this case $\pi_V$ is defined on the set of points at least $r_\theta$ from $\Hhat$ and not in the interior of a horoball that $\Hhat$ enters, and $K_U$ is one of: a stratum closure disjoint from $\Hhat$; or a horoball that is disjoint from $\Hhat$. 

For the claim, the key subcase here is when $K_U$ is a stratum closure that enters a horoball that $\Hhat$ enters, and this subcase follows as in Case 3 since \cref{R:constants} gives a separation of at least $1000n r_\theta$ between $\Hhat$ and the stratum closure outside of horoballs. The other subcases follow as in Case 2. This proves the claim in this case. 

To prove the result, we divide into two subcases. 

\bold{Case 4a: No horoball is entered by both $K_U$ and $\Hhat$.} In this case we get the result by \cref{L:AnglePreBGI} and \cref{R:constants}.

\bold{Case 4b: Some horoball is entered by both $K_U$ and $\Hhat$.} In this case, $K_U$ is a stratum closure $\Shat$ that is disjoint from $\Hhat$ but enters a horoball $\That_{\hat{c}}$ that $\Hhat$ enters. Note that $\Shat$ cannot go into any other horoball that $\Hhat$ enters, since otherwise $\Shat$ and $\Hhat$ would have a geodesic in common \cite[Proposition 4.4.4]{DasSimmonsUrbanski}. 

Every point of $K_U - \That_{\hat{c}}$  can be joined to its closest point projection to $\That_{\hat{c}}$ by a geodesic. This closest point projection lies in $\Shat \cap \partial \That_{\hat{c}}$ by \cref{R:HoroballCapStratProj}, so the geodesic lies in the region where \cref{L:AnglePreBGI} applies. 
%
%
The intersection $\Shat \cap \partial \That_{\hat{c}}$ maps to a coarse point of $\cC(V)$ via the alternative projection, by \cref{L:rhoAngleTree}. This gives the result by \cref{C:AngleProjOnHoroBoundary}.
\end{proof}

We conclude with the auxiliary statement of the boundedness assumption. 

\begin{lemma}\label{L:NestIntersect}
If $U\sqsubsetneq V$ then $\cA_U \cap \cA_V\neq\emptyset.$
\end{lemma} 

\begin{proof}
This follows immediately from the definitions; see \cref{fig:relsbounded}.
\end{proof}

\subsection{Structured redundancy}\label{SS:SR}
We prove the two halves of the structured redundancy assumption separately.

\begin{proposition}\label{P:SR1} 
There exists a $B>0$ such that
if $U\pitchfork V$, for any $x\in \Bhatthick$ we have at least one of
$$d_U(\pi_U(x), \pi_U(\cA_V)) \leq B \quad\quad\text{or}\quad\quad d_V(\pi_V(x), \pi_V(\cA_U)) \leq B.$$
\end{proposition}

Roughly speaking, the proof follows a common strategy that considers a sort of first entry point. 

\begin{proof}
To prove this, start by defining $\cB_U^+$ to be $\cB_U$ if $U$ is a stratum or cusp domain, and to be $\cB_U$  union all the horoballs it intersects if $U$ is an angle domain.  

Let $z$ be any point of $\cA_U\cup \cA_V$. Let $y$ be the first point along the geodesic from $x$ to $z$ that is contained in $\cB_U^+\cup \cB_V^+$. Without loss of generality, assume $y\in \cB_U^+$. If $y$ is in \emph{both} $\cB_U^+$ and $\cB_V^+$, then without loss of generality assume that if $y$ is in $\cB_U \cup \cB_V$ then $y\in \cB_U$.

Note that the interior of the geodesic from $x$ to $y$ is disjoint from $\cB_U^+\cup \cB_V^+$, and so in particular the geodesic is in  $\Bhat^U\cap \Bhat^V$. 
By \cref{P:BoundedProjections}, it suffices to show $\pi_V(x)$ is close to $\pi_V(\cB_U \cap \Bhat^V)$ or $\pi_U(x)$ is close to $\pi_U(\cB_V \cap \Bhat^U)$.

\bold{Case 1: $y \in \cB_U.$} In this case $\pi_V(\cB_U \cap \Bhat^V)\approx \pi_V(x)$ follows from \cref{C:CATBGI,L:AnglePreBGI}.

\bold{Note:} If Case 1 does not apply, then since $y\notin \cB_U\cup \cB_V$ and $y\in \cB_U^+$ we have
\begin{enumerate}
\item $U$ is an angle domain and the corresponding  hyperplane enters a horoball $\That_{\hat{c}}$ such that $y\in \partial \That_{\hat{c}}$, and 
\item $V$ cannot be a cusp domain associated to $\That_{\hat{c}}$.
\end{enumerate}
We assume this in all the remaining cases. 

\bold{Case 2: $V$ is a cusp domain, or $V$ is a stratum or angle domain with $K_V$ disjoint from $\That_{\hat{c}}$.}
In this case \cref{R:constants} gives that the set $K_V$ is far from $\That_{\hat{c}}$. Since there are points of $\cB_U \cap \Bhat^V$ arbitrarily close to $\That_{\hat{c}}$, $\pi_V(\cB_U \cap \Bhat^V)\approx \pi_V(x)$ again follows from  \cref{C:CATBGI,L:AnglePreBGI,R:constants}.

\bold{Case 3: $V$ is a stratum domain whose associated stratum intersects $\That_{\hat{c}}$.} Since there are points of $\cB_U \cap \Bhat^V$ arbitrarily close to $\That_{\hat{c}}$, and since  $\pi_V(\That_{\hat{c}})$ is a coarse point, we get $\pi_V(\cB_U \cap \Bhat^V)\approx \pi_V(x)$ from  \cref{C:CATBGI}.

\bold{Case 4: $V$ is an angle domain whose hyperplane intersects $\That_{\hat{c}}$.} So the two hyperplanes, which must be disjoint, both enter the given horoball. We continue to make use of the entry point $y$ defined above.

We do a second version of the rough idea used above, but now in the factor $\Chat_i$ where the hyperplanes associated to $U$ and $V$ are constant. Let $q_U, q_V\in \Chat_i$ be those constant values, and let $y_i$ be the image of $y$ in $\Chat_i$. Pick $z'\in \{q_U, q_V\}$, and consider the geodesic in $\Chat_i$ from $y_i$ to $z'$. Let $y_i'$ be the first point of this geodesic in $\{q_U, q_V\}$, and without loss of generality assume $y_i
'=q_U$. 

\cref{L:AnglePreBGI}, \cref{P:TreeAngleBGI}, and \cref{C:AngleProjOnHoroBoundary}  then give that $\pi_V(\cB_U \cap \Bhat^V)\approx \pi_V(y)\approx \pi_V(x)$.
\end{proof}

\begin{proposition}\label{P:SR2}
There exists a $B>0$ such that the following holds. 
 If $U\sqsubsetneq V$, for any $x,y\in \Bhatthick$, if there is a geodesic from $\pi_V(x)$ to $ \pi_V(y)$ that has distance at least $B$ from $\pi_V(\cA_U)$, then $d_U(\pi_U(x), \pi_U(y))\leq B$.
\end{proposition}

\begin{proof}
We divide the proof into cases.

\bold{Observation 1:} If $W$ is a stratum domain with corresponding stratum $\Shatcirc$, $\pi_W$ is defined on $\Bhat$. For any $x,y$, the geodesic from $\pi_W(x)$ to $\pi_W(y)$ is Hausdorff close to $\pi_W$ of the geodesic in $\Bhat$ from $x$ to $y$, by \cref{C:CATBGI} and \cref{P:ElectHyp}.

\bold{Case 1: $V$ is a stratum domain and $U$ is a stratum domain.} In this case we have a stratum $\Shatcirc$ associated to $V$ and another stratum contained in the closure of $\Shatcirc$, and it suffices to apply observation 1 and \cref{C:CATBGI}.

\bold{Case 2: $V$ is a stratum domain and $U$ is a cusp domain.} In this case we have a stratum $\Shatcirc$ associated to $V$, and a horoball $\That_{\hat{c}}$ that $\Shatcirc$ enters (with an extra condition if $U$ is a tree domain). Suppose that $d_U(\pi_U(x), \pi_U(y))$ is large. It follows that the closest point projections of $x$ and $y$ to $\That_{\hat{c}}$ are far apart, and hence that the geodesic in $\Bhat$ from $x$ to $y$ enters $\That_{\hat{c}}$. Observation 1 now gives that  there is a geodesic from $\pi_V(x)$ to $ \pi_V(y)$ that passes close to $\pi_V(\cA_U)$.

\bold{Case 3: $V$ is a stratum domain and $U$ is an angle domain.} In this case we have a stratum $\Shatcirc$ associated to $V$ and a hyperplane $\Hhat$ associated to $U$, and $\Shat$ intersects but is not contained in $\Hhat$. Consider $x$ and $y$ whose images in $\cC(U) =N^1(\Hhat)^\sw$ are far apart. 

First suppose that the $\Bhat$ geodesic from $x$ to $y$ doesn't go into any horoball that $\Hhat$ goes into. In this case \cref{L:AnglePreBGI} shows that the geodesic from $x$ to $y$ must come close to $\Hhat$. This gives that the geodesic from $\pi_V(x)$ to $\pi_V(y)$ must come close to $\pi_V(\cA_U)$. 

Next suppose that the geodesic from $x$ to $y$ does go into a horoball $\That_{\hat{c}}$ that $\Hhat$ goes into. If $\Shat$ also goes into $\That_{\hat{c}}$, this directly gives that the geodesic from $\pi_V(x)$ to $\pi_V(y)$ must come close to $\pi_V(\cA_U)$, since  the horoball is electrified in $\cC(V)$. If $\Shat$ does not go into $\That_{\hat{c}}$, it suffices to note that the projection of $\That_{\hat{c}}$ to $\Shat$ is bounded and is uniformly bounded distance from $\pi_V(\cA_U)$. 

\bold{Case 4: $V$ is a tree domain and $U$ is an angle domain.} In this case the hyperplane $\Hhat$ associated to $U$ intersects the horoball $\That_{\hat{c}}$ associated to $V$, and $\Hhat$ is constant in the factor $\Chat_i$ associated to $V$. Assume that $\pi_U(x)$ and $\pi_U(y)$ are far apart. 

\bold{Case 4a:} Suppose there is a geodesic from either $x$ or $y$ to $\Hhat$ that doesn't go into $\That_{\hat{c}}$. Without loss of generality, say there is a geodesic from $x$ to a point $x'$ of $\Hhat$ that doesn't go into $\That_{\hat{c}}$. We get that $\pi_V(x)$ is about $\pi_V(x')$, which is uniformly close to $\pi_V(\cA_U)$. 

Thus, in the remaining subcases,  we can restrict to the cases where the closest point projections of $x$ and $y$ to $\Hhat$ lie in $\That_{\hat{c}}$, and a small elaboration of the argument above also allows us to restrict to the case where these closest point projections of $x$ and $y$ to $\Hhat$ are not close to the complement of $\That_{\hat{c}}$.

\bold{Case 4b:}
Suppose the $\Bhat$ geodesic from $x$ to $y$ doesn't go  into $\That_{\hat{c}}$. In this case, this geodesic must go close to $\Hhat$ or a horoball that $\Hhat$ enters other than $\That_{\hat{c}}$, by \cref{L:AnglePreBGI}. The projection of $\Hhat$ or any such horoball is coarsely equal to $\pi_V(\cA_U)$, so this gives that $\pi_V$ of this geodesic must come close to $\pi_V(\cA_U)$. Since we are assuming the geodesic doesn't go  into  $\That_{\hat{c}}$, this image under $\pi_V$ is coarsely constant, thus giving in a degenerate form the desired conclusion that the geodesic from $\pi_V(x)$ to $\pi_V(y)$ must come close to $\pi_V(\cA_U)$.

\bold{Case 4c:} Suppose the $\Bhat$ geodesic from $x$ to $y$ does go into $\That_{\hat{c}}$. If it enters or exits near $\Hhat$, the result follows from \cref{L:EntryPoint}, so we can assume otherwise. By supposition, $\pi_U$ is computed using the modified part of the definition, so \cref{P:TreeAngleBGI} gives the result.
\end{proof}

\section{Proof of  hierarchical hyperbolicity }\label{S:Final}

In this section we prove the enough projections  and progress localization assumptions  of \cref{P:Criterion}, and then we put it all together to prove the main results of this paper. 

\subsection{Enough projections}\label{SS:EP}
We start with the enough projections assumption.

\begin{proposition}\label{P:Uniqueness} 
For each $\kappa\geq 0$ there exists $\omega\geq0$ such that if $x, y\in \Bhatthick$ satisfy $d_{\Bhatthick}(x,y)\geq\omega$ then there exists $U$ such that 
$$d_U(\pi_U(x), \pi_U(y))\geq \kappa.$$
\end{proposition}

We isolate the largest part of the remaining required analysis as a lemma. 

\begin{lemma}\label{L:HardPartOfUniqueness}
For each $\kappa'\geq 0$ there exists $\omega'\geq0$ such that if $x, y\in \Bhatthick$ satisfy $d_{\Bhat}(x,y)\leq \kappa'$ and  $d_{U}(\pi_U(x),\pi_U(y))\leq \kappa'$ for all angle domains $U$, then $d_{\Bhatthick}(x,y)\leq\omega'$.
\end{lemma}

\begin{proof}
For expositional clarity we start with the case when $\Gamma$ is cocompact, or equivalently when there are no cusp domains.

Throughout this proof, we will refer to $\Vhat$ as in \cref{SS:Uniformity}, and we will always assume this $\Vhat$ belongs to a finite union of $\Gammahat$ orbits such that the $\Vhat_0$ associated to this finite union of $\Gammahat$ orbits covers $\Bhat$. (More generally they should cover the complement of the horoballs, but for the moment we assume there aren't any horoballs.) 

Consider the geodesic in $\Bhat$ from $x$ to $y$. Pick a sequence $$x=x_0, x_1, \ldots, x_\ell=y$$ of $O(d_{\Bhat}(x,y))$ points along this geodesic, in order, each distance at most $\xi/100$ from the next.  
For each $i$, find one of the $\Vhat_0$ from \cref{SS:Uniformity} that contains $x_i$, call it $\Vhat_0^{(i)}$. Recall these $\Vhat_0$ are balls of radius $r_i\geq\xi/100$ about some center point $p_i$ contained in every hyperplane that intersects $\Vhat$, and that the ball of radius $100 r_i$ about $p_i$ is contained in $\Vhat$. 

Using \cref{R:NudgeToThick} on the point $p_i$, find a point $x_i'$ whose distance to $p_i$ is much less than $\xi/100$. Using \cref{L:StabCobounded}, adjust it to a point $x''_i\in \Bhatthick$ with $d(p_i, x''_i)=d(p_i, x_i')$ and such that $\pi_U(x''_i)$ is uniformly close to $\pi_U(x)$ for all angle domains $U$ associated to  hyperplanes passing through $\Vhat$. The fact that $\pi_U$ is coarsely Lipschitz applied either to the  geodesic $[x,x_i]$ or the geodesic $[x_i, y]$, implies that same for all angle domains $U$ associated to  hyperplanes not passing through $\Vhat$.

We can assume $x_0'' = x$ and $x_\ell''=y$.  
The triangle inequality gives 
$$
d_{\Bhat}(p_i, x_{i+1}'') \leq  
d(p_i, x_i)+ d(x_i, x_{i+1}) + d(x_{i+1}, x_{i+1}'') \leq 2r_i+\xi/100+2r_{i+1}.
$$
We also get the same for $d_{\Bhat}(p_{i+1}, x_{i}'')$, so depending on which of $r_i, r_{i+1}$ is larger we get that either both $x_i''$ and $x_{i+1}''$ are in $\Vhat^{(i)}$ or both are in $\Vhat^{(i+1)}$. Thus the result follows from \cref{L:PreUniqueness} by the triangle inequality. 

We now sketch one of several ways of adapting this to the case when there are horoballs, leaving the details to the reader. We start with the $\Bhat$ geodesic joining $x$ to $y$. Consider the points $x=x_0,x_1, x_2, \ldots, x_k=y$ along this geodesic such that the $x_i, 0<i<k$ are exactly the entry and exit points of the geodesic into the 1-neighborhoods of horoballs. As above, for $0<i<k$, we can find $x_i''\in \Bhatthick$ very close to $x_i$, and such that $\pi_U(x''_i)$ is uniformly close to $\pi_U(x)$ for all angle domains $U$. (It is helpful to first note that a bounded length entry of a geodesic into a horoball cannot substantially change the angle projection associated to a hyperplane that that segment does not come close to. One can prove this by radially pushing to the boundary of the horoball and using that the angle projection is coarsely Lipschitz.)

It suffices to show that $d_{\Bhatthick}(x_i'', x_{i+1}'')$ is bounded for all $i$. 

In the first case, the geodesic from $x_i$ to $x_{i+1}$ stays in the thick part, and we argue exactly as in the cocompact case. 

In the second case, the geodesic goes into (or at least close to) a horoball. First, we replace the geodesic with a path on the boundary of the 1-neighborhood of the horoball that is a geodesic  in each factor $\Chat_i$. \cref{L:rhoAngleTree} gives that the visual angle at each singular point in $\Chat_i$ bounded in terms of $\kappa'$, so factor by factor we can push slightly away from the singular points to get a thick path of bounded length. 
\end{proof}

\begin{proof}[Proof of \cref{P:Uniqueness}]
Consider the geodesic $\gamma=[x,y]$ in $\Bhat$.

\bold{Case 1:} The distance between $x$ and $y$ is large in $\Bhat$.

If all the regions of $\Bhat$ that are electrified when defining $\cC(\Bhatcirc)$ have small projection onto $\gamma$, the closest point projection from $\Bhat$ to $\gamma$ extends to a coarsely Lipschitz map from the electrification of $\Bhat$ to the interval parameterizing $\gamma$, and thus the distance between $x$ and $y$ in $\Bhat$ is comparable to the distance in the electrification. Hence we can take $U=\Bhatcirc$.

Otherwise, consider an electrified region that is minimal under inclusion which has large projection to $\gamma$. This implies that $\gamma$ stays close to the electrified region for a long time, and hence the projection of $x$ and $y$ to the electrified region are far apart. If that region is a stratum, then repeating this argument in this stratum gives the result. If the region is a horoball the result follows from \cref{L:HoroballProductMetric}.

\bold{Case 2:} The distance between $x$ and $y$ is small in $\Bhat$. This case was handled in \cref{L:HardPartOfUniqueness}.
\end{proof}

\subsection{Progress localization}\label{SS:PL}
We now address progress localization, building on the analysis showing enough projections and the analysis in \cref{S:Boundedness}. 

\begin{proposition}\label{P:LargeLinks} 
For any $P\geq 0$ there exists  $N\geq 0$ and $R\geq 0$ such that for any $V$ and any $x,y\in \Bhatthick$ with $$d_V(\pi_V(x),\pi_V(y))\leq P$$ there is a set of at most $N$ domains $U_i\sqsubsetneq V$ such that if $T\sqsubsetneq V$ is not nested in or orthogonal to one of the $U_i$ then $d_T(\pi_T(x),\pi_T(y))\leq R$. 
\end{proposition}

\begin{lemma}\label{L:GettingToStratum}
There is a $D>0$ such that the following holds. 
Let $\Shat$ be a stratum closure, and let $x\in \Bhatthick$. There exists a set $\cE$ of at most $D$ center or angle domains such that the following holds. 

Consider the geodesic from $x$ to the closest point projection $p_{\Shat}(x)$ of $x$ to $\Shat$. If the geodesic $[x,p_{\Shat}(x)]$ goes into a horoball $\That_{\hat{c}}$ that $\Shat$ goes into, define $x'$ to be the first entry into such a horoball. Otherwise set $x'=p_{\Shat}(x)$. If $U$ is nested  in the stratum domain of $\Shat^\circ$, and either $\pi_U(x')$ is undefined or $$d_U(\pi_U(x), \pi_U(x'))\geq D,$$ then $U$ is an angle domain, and either $U$ is equal to one of the angle domains in $\cE$ or $U$ is orthogonal to one of the center domains in $\cE$. 
\end{lemma} 

Recall from \cref{R:extension} that $\pi_U(x')$ is undefined when $x' \notin \Bhat^U$.

\begin{proof}
It is helpful to keep \cref{fig:relsbounded} in mind. 

\bold{Case 1:} If $U$ is a stratum domain, then the corresponding stratum is contained in $\Shat$, so \cref{C:CATBGI} gives that $d_U(\pi_U(x), \pi_U(x'))$ is uniformly small.

\bold{Case 2:} If $U$ is a cusp domain, then $\Shat$ goes into the corresponding horoball, so \cref{C:CATBGI} gives that $d_U(\pi_U(x), \pi_U(x'))$ is uniformly small.

\bold{Case 3:} If $U$ is an angle domain corresponding to a hyperplane $\Hhat$, then $\Hhat$ must intersect $\Shat$ but must not contain $\Shat$. Assume $d_U(\pi_U(x), \pi_U(x'))$ is large, and apply \cref{L:AnglePreBGI} to get a point $x''$ on $[x,x']$ that is either within $r_\theta$ of $\Hhat$ or in a horoball that $\Hhat$ enters but $\Shat$ does not enter. By definition, $x$ and $x''$ have the same projection to $\Shat$. 

We now claim that $\Hhat\cap \Shat$  contains a point near $p_{\Shat}(x)$. To see this, first suppose that $x''$ is within $r_\theta$ of $\Hhat$. In this case the claim follows from the fact that  the projection of $\Hhat$ to $\Shat$ is $\Shat\cap \Hhat$ (\cref{R:OrthIntersection}) and the fact that projection to $\Shat$ is distance non-increasing, and the definition of ``near'' we get in this case can be taken to be ``within $r_\theta$''. It remains to suppose that $x''$ is in a horoball that $\Hhat$ enters but $\Shat$ does not. In this case the claim follows by additionally applying \cref{R:constants}, and the definition of ``near'' we get in this case is ``within 1''.
%
%

\bold{Case 3a:} If $x'\neq p_{\Shat}(x)$, then $p_{\Shat}(x)$ is in some horoball $\That_{\hat{c}}$. Let $\cE$ consist of only the center domain of $\That_{\hat{c}}$. Because $p_{\Shat}(x)$ is at most 1 away from $\Hhat\cap {\Shat}$, and we can assume (via $s_\theta \gg_{n,\cH, \Gamma} 1$ in \cref{R:constants}) that any hyperplane that goes close to a horoball must enter it, we get that $\Hhat$ intersects ${\Shat}\cap \That_{\hat{c}}$. This means that $U$ is orthogonal to the center domain of the horoball $\That_{\hat{c}}$ containing $p_{\Shat}(x)$. 

\bold{Case 3b:} We now assume that $x'=p_{\Shat}(x)$, so $p_{\Shat}(x)$ is not in a horoball. For this case, the following is helpful. 

\begin{remark}
Suppose $\Hhat$ is a hyperplane that enters the horoball of a cusp point $\hat{c}'$, and also that $\Hhat$ intersects a stratum closure $\Shat$, and $\Shat$ does not enter the horoball of $\hat{c}'$. Then, by \cref{R:OrthIntersection}, $\Hhat$ contains the projection $p_{\Shat}(\hat{c}')$. (The fact that $p_{\Shat}$ extends continuously to $\hat{c}'$ follows from a short comparison geometry argument and the fact that $p_{\Shat}$ is distance non-increasing.)
\end{remark}

Fix $L>0$ large enough so that the projection of any horoball onto another distance at least $L$ away has diameter at most $r_\theta$.  Let $y$ be the last point on $[x,x']$ that is not in a horoball and has distance at least $L$ from $x'$; or let $y=x$ if such a point does not exist. 

Let $\cE$ consist of 
 all angle domains of hyperplanes that $x'$ is within $r_\theta$ of but don't contain $\Shat$, and
 all angle domains associated to hyperplanes containing one of the points $p_{\Shat}(\hat{c}')$ but not containing $\Shat$, where the $\hat{c}'$ correspond to horoballs that $[y,x']$ enters.

The result now follows from \cref{L:AnglePreBGI}.
Indeed, suppose $\Hhat$ is a hyperplane that $[x,x']$ comes within $r_\theta$ of, or that enters a horoball that $[x,x']$ enters. If $[x,x']$ comes within $r_\theta$ of a hyperplane $\Hhat$ that intersects $\Shat$, then $\Hhat\cap {\Shat}$ is within $r_\theta$ of $x'$. Similarly if $[x,y]$ enters a horoball containing a hyperplane $\Hhat$ that intersects $\Shat$, then $\Hhat\cap {\Shat}$ must be extremely close to $x'$. Finally, if $[y,x']$ enters a horoball that $\Hhat$ enters, then $\Hhat$ contains one of the $p_{\Shat}(\hat{c}')$. 
\end{proof}

We also have a similar statement with the stratum replaced by a horoball. 

\begin{lemma}\label{L:GettingToHoroball}
There is a $D>0$ such that the following holds. 
Let $\That_{\hat{c}}$ be a horoball, and let $x\in \Bhatthick$. There exists a set $\cE$ of at most $D$  angle domains such that the following holds.  Consider the geodesic from $x$ to the closest point projection $x'=p_{\That_{\hat{c}}}(x)$ of $x$ to $\That_{\hat{c}}$.  If $U$ is properly nested in a cusp domain associated to $\That_{\hat{c}}$ and either $\pi_U(x')$ is not defined or $$d_U(\pi_U(x), \pi_U(x'))\geq D,$$ then $U$ is  equal to one of the angle domains in $\cE$.
\end{lemma}

\begin{proof}
The proof is similar to the proof of \cref{L:GettingToStratum} and is left to the reader. The only kind of domain that can be properly nested in a cusp domain is an angle domain, so one only needs to adapt Case 3. Since $x'$ in this case is always $p_{\That_{\hat{c}}}(x)$, one does not have an analogue of Case 3a; one only has 3b. 
\end{proof}

Finally we can prove the progress localization assumption. 

\begin{proof}[Proof of \cref{P:LargeLinks}]
There are only two types of $V$ that are not nest-minimal, namely stratum domains and tree domains. Our goal is to build a set of domains $\{U_i\}$ with the desired properties. 

\bold{Case 1: $V$ is a stratum domain.} Let $\Shat$ be the associated stratum closure.
\cref{L:GettingToStratum} gives us points $x'$, $y'$ that are either in $\Shat$ or in a horoball that $\Shat$ enters. Either way, these points define points $x'', y''\in \cC(V)$; if $x'\in \Shat$ then we have $x''=x'$ viewed as a point of $\cC(V)$, and otherwise $x''$ is the generic point of the horoball in whose boundary it lies. 

We now consider a geodesic (or path of approximately minimal length) $\alpha$ from $x''$ to $y''$ in $\cC(V)$. 
By definition this can contain at most $d_V(x'', y'')/2+2$ generic points, and all the subpaths in $\Shat$ of course have length bounded by $d_V(x'', y'')$. None of the subpaths can go more than depth 1 into a horoball, so there is no harm to assuming they are disjoint from horoballs, and we will do this. (The paths could be pushed out of horoballs.)

We construct the following set of domains $\{U_i\}$ consisting of the following.
\begin{enumerate}
\item The domains given by the use of \cref{L:GettingToStratum} above.
\item For each generic point of a horoball (intersected by $\Shat$) that $\alpha$ passes through, add the cusp domains associated to the horoball that are nested in $V$. 
\item For each generic point of a codimension 1 hyperplane (in $\Shat$) that $\alpha$ contains, add the corresponding stratum domain (when the corresponding stratum has positive dimension), and the angle domain corresponding to the hyperplane containing that codimension 1 locus but not containing all of $\Shat$. 
\item For each hyperplane not containing $\Shat$ that comes within $r_\theta$ of a bounded length subpath of $\alpha$ in $\Shat$, add the corresponding angle domain.
\end{enumerate}
This set $\{U_i\}$ of domains is of bounded size (depending on $P$), so it suffices now to show that all domains $T\sqsubsetneq V$ which are properly nested in $V$ but not nested in or orthogonal to one of the $\{U_i\}$ have that $d_T(\pi_T(x), \pi_T(y))$ is bounded. Note that if $T\sqsubsetneq V$ is not  nested in or orthogonal to $U_i$, then $T\pitchfork U_i$ or $U_i\sqsubsetneq T$. 

By construction of $\{U_i\}$, we have that $T$ cannot be a cusp domain associated to a horoball that $\alpha$ passes through, nor can it be an angle domain  associated to a hyperplane that enters such a horoball.

\bold{Case 1a: $T$ is a stratum domain or a cusp domain.} In the stratum domain case, $\Bhat^T = \Bhat$, and in the cusp domain $\Bhat^T$ is the complement of a horoball. We have that $\alpha$ consists of bounded subpaths in  $\Shat\cap\Bhat^T$ and visits to generic points. The bounded subpaths have uniformly bounded image in $\cC(T)$ because $\pi_T$ is coarsely Lipschitz on $\Bhat^T$. The start and end points of an excursion to a generic point map to nearby points of $\cC(T)$ by \cref{P:BoundedProjections}.

\bold{Case 1b: $T$ is an angle domain.} Let $\Hhat$ be the corresponding hyperplane. As mentioned above, $\Hhat$ does not enter any horoball that $\alpha$ enters. Additionally, we have that $\alpha$ does not come within $r_\theta$ of $\Hhat$. So $\alpha$ consists of subpaths in $\Shat \cap \Bhat^T$ and generic points of regions contained in $\Bhat^T$. The conclusion is the same as in Case 1a.

\bold{Case 2: $V$ is a tree domain.} Only angle domains can be nested in tree domains, and no angle domain can be nested in another. 

Let $x'$ and $y'$ be the projections to the horoball associated to the tree domain. \cref{L:GettingToHoroball} controls the passage from $x$ to $x'$ and from $y'$ to $y$, so it suffices to think about how $x'$ and $y'$ differ. The result follows from \cref{P:TreeAngleBGI} and is a variant of \cref{C:TreeAngleBGI}. 
\end{proof}

\subsection{Proof of  hierarchical hyperbolicity}
Putting everything together, we get hierarchical hyperbolicity for our geometric model of the group.

\begin{theorem}
$(\Bhatthick, \IndexSmall, \pi_\bullet)$ satisfies all the assumptions of \cref{P:Criterion}. In particular, $\Bhatthick$ is an HHS. 
\end{theorem}

\begin{proof}
In \cref{P:pre} we checked that $(\Bhatthick, \IndexSmall, \pi_\bullet)$ is a pre-HHS, in \cref{P:av} we checked that $\IndexSmall$ is a nearly valid index set, and in \cref{S:Boundedness} we defined the active regions, so it suffices to check the numbered  assumptions of \cref{P:Criterion}.
\begin{enumerate}
\item The enough projections assumption was checked in \cref{P:Uniqueness}. 
\item The density assumption was checked in \cref{L:Density}.
\item The bounded projections assumption was checked in \cref{P:BoundedProjections} and \cref{L:NestIntersect}.
\item The structured redundancy assumption was checked in \cref{P:SR1} and \cref{P:SR2}. 
\item The progress localization assumption was checked in \cref{P:LargeLinks}. 
\end{enumerate} 
Thus \cref{P:Criterion} gives the result. 
\end{proof}

We can now conclude. 

\begin{proof}[Proof of \cref{T:main}]
This now follows from \cref{P:CriterionGroups}, since all of our constructions are natural and canonical enough to get the required equivariance under isometries. The condition on finitely many orbits is satisfied thanks to \cref{L:MaxOrthogonalSets} and the fact that there are only finitely many orbits of horoballs and strata. 
\end{proof}

\subsection{The motivating application}\label{SS:VerificationForMcs} 
For completeness, we now briefly reference the specific results  that show that \cref{T:main} applies to give \cref{T:Cubic}.

Let $\cE$ be the ring of Eisenstein integers, and let $\Lambda = \cE^{4,1}$ be the free module $\cE^5$ equipped with the hermitian form 
$$h(x,y) = -x_0 \ol{y}_0 + x_1 \ol{y_1} + \cdots + x_4 \ol{y_4}.$$
Let $\Gamma\subset PU(4,1)$ be the projectivized automorphism group of $\Lambda$, 
%
%
and let $\cH$ be the union of the hyperplanes in  $B^4$ defined using the perps of the norm 1 elements of $\Lambda$. 
%
%

\cite[Theorem 2.20]{AllcockCarlsonToledoSurfaces} shows that $\cM_{cs}$ is isomorphic to $\Gamma \back (B^4-\cH)$.
\ref{LocFin}  is checked in \cite[Proof of Corollary 3.4]{AllcockAsphericity}, \ref{Orth} in \cite[Lemma 7.28]{AllcockCarlsonToledoSurfaces}, and \ref{Inv} is by definition. 
\cite[Theorem 7.21, Theorem 8.2]{AllcockCarlsonToledoSurfaces} shows that $\Gamma$ has only one orbit of cusps. 

One can verify \ref{CuspLimit} by noting that the null vector $(1,1,0,0,0)$ defines a cusp that is a limit of the 1-dimensional stratum defined by norm 1 elements $$(0,0,1,0,0), (0,0,0,1,0), \text{ and } (0,0,0,0,1).$$

\subsection{The Farrell-Jones Conjecture}\label{SS:FJ}

What \cite{DurhamMinskySistoCAT0} calls the Farrell-Jones Conjecture is what \cite{BartelsBestvina} calls the Farrell-Jones Conjecture with wreath products; 
%
%
it means that the wreath product with every finite group satisfies the K- and L-theoretic Farrell-Jones Conjectures with coefficients in any additive category. 
%
%
This version of the conjecture has good inheritance properties, many of which are summarized in \cite[Proposition 1.1]{GandiniRuping}. What is important here is that the class of groups that satisfy the Farrell-Jones Conjecture is stable under taking subgroups, finite direct products, central extensions with finitely generated kernel, 
%
%
%
and passing to finite index overgroups, and that hyperbolic groups satisfy the  Farrell-Jones Conjecture. We make use of these properties below without further comment. 

We recall from \cite[Theorem C]{DurhamMinskySistoCAT0} that if $\Gammahat$ is a colorable HHG, and if $\Stab(U)$ satisfies the Farrell-Jones Conjecture for every non-nest maximal domain $U$, then $\Gammahat$ satisfies the Farrell-Jones Conjecture. 
%
%
%
Colorability is a mild technical condition that is checked under the assumptions of \cref{T:main} in \cref{L:colorable}. 

The core of the proof is the following two observations.

\begin{lemma}\label{L:StabS}
For any stratum closure $\Shat\subset \Bhat$ of codimension $k$, $\Stab_{\Gammahat}(\Shat)$ has a finite index subgroup $\Gammahat_0$ that is a central $\bZ^k$-extension of a group as in \cref{T:main} for a ball of dimension $n-k$. 
\end{lemma}

In the degenerate case $k=n$ the stabilizer is virtually $\bZ^n$.

\begin{proof}
$\Gamma$ has a finite index torsion free subgroup $\Gamma'$, whose preimage $\Gammahat'$ is a finite index subgroup of $\Gammahat$. We set $\Gammahat_0$ to be the finite index subgroup of $\Gammahat' \cap \Stab_{\Gammahat}(\Shat)$ which does not permute the hyperplanes containing $\Shat$. This $\Gammahat_0$ acts on $\Shat$ with central kernel $\bZ^k$, where $k$ is the codimension of $\Shat$. The image of $\Gammahat_0$ in the isometry group of $\Shat$ is a group as in \cref{T:main} for a ball of dimension $n-k$. 
%
%
\end{proof}

\begin{lemma}\label{L:Stabc}
For any horoball $\That_{\hat{c}}$, the stabilizer $\Stab_{\Gammahat}(\That_{\hat{c}})$ has a finite index subgroup that is a  product of $\bZ$ and $n-1$ finitely generated non-abelian free groups. 
\end{lemma}

\begin{proof}
Although it is not required, we prefer to use the discussion in \cref{R:tauMotivation}. This remark shows that $\That_{\hat{c}}$ deformation retracts onto an $\bR$-bundle over a product of $n-1$ finite valence trees. After passing to a finite index subgroup, one can assume that the image acts on the product of the $n-1$ trees as a product of $n-1$ cocompact free group actions. 
%
%
The kernel of this action is a central copy of $\bZ$.
%
%

Because \cref{R:tauMotivation} gives that the monodromy of the $\bR$-bundle is trivial, one gets that the finite index subgroup is abstractly a subgroup of $\bR$ times the $n-1$ finitely generated free groups. When $Q$ is a product of finitely generated free groups any central extension $1\to \bZ \to H \to Q \to 1$ that embeds in $\bR\times Q$ in such a way that the maps to $Q$ agree must be the trivial extension $\bZ\times Q$. 
%
%
\end{proof}

\begin{proof}[Proof of \cref{T:FJ}]
We prove this by induction on the dimension $n$. For $n=1$, it is true because $\Gammahat$  has a finite index subgroup which is a free group or a surface group. Assume the result is known for all dimensions less than some $n>1$, and consider a $\Gammahat$ in dimension $n$. 

We now wish to apply \cite[Theorem C]{DurhamMinskySistoCAT0}. 
A minor subtlety is that $\IndexSmall$ is not the index set of the HHG structure on $\Gammahat$. The actual index set $\IndexBig$ consists of tuples of orthogonal domains in $\IndexSmall$. This means that the stabilizer of a domain in $\IndexBig$ is, up to finite index, an intersection of the stabilizers of a finite set of domains in $\IndexSmall$. 

It follows that the stabilizers of domains in $\IndexBig$, up to finite index, are all of the form $\Stab_{\Gammahat}(\Shat)$ or $\Stab_{\Gammahat}(\That_{\hat{c}})$ or $\Stab_{\Gammahat}(\Shat) \cap \Stab_{\Gammahat}(\That_{\hat{c}})$, where in this final case the stratum closure $\Shat$ must enter the horoball $\That_{\hat{c}}$. 

We see that
$\Stab_{\Gammahat}(\Shat)$ satisfies the Farrell-Jones Conjecture by \cref{L:StabS} and the induction hypothesis, and $\Stab_{\Gammahat}(\That_{\hat{c}})$ satisfies the Farrell-Jones Conjecture by \cref{L:Stabc}. That $\Stab_{\Gammahat}(\Shat) \cap \Stab_{\Gammahat}(\That_{\hat{c}})$ satisfies the Farrell-Jones Conjecture is then automatic because it is a subgroup of a group that satisfies the Farrell-Jones Conjecture. The result now follows from \cite[Theorem C]{DurhamMinskySistoCAT0}.
\end{proof}

\begin{remark}\label{R:CPR}
A related analysis shows that the HHG structures provided by \cref{T:main} have cobounded product regions in the sense of \cite[Definition 1.13]{Mangioni}.
%
%
%
\end{remark}

\appendix

\section{Proof of criteria for hierarchical hyperbolicity }\label{S:CritProof}

\subsection{The definition of an HHS} 

Here we follow the definition in \cite[Section 1.3]{BehrstockHagenSistoHHSII}. We have made minor changes to the definition, which do not change which spaces the definition applies to. These changes are  noted after the definition. We continue to use ``asynchronous'' as a replacement for the standard term ``transverse''.

\begin{definition}\label{D:HHS}
Let $(\cX, \IndexBig, \pi_\bullet)$  be a pre-HHS with  $\IndexBig$ a valid index set. We say that $(\cX, \IndexBig, \pi_\bullet)$ is a hierarchically hyperbolic space if the following hold for some constants $\kappa_0, \alpha\geq 0$, $E\geq \kappa_0$. 
\begin{enumerate}
\item\label[Haxiom]{O:N} \textbf{Nesting:} For each properly nested pair of domains $U\sqsubsetneq V$ there is a specified point $\rho^U_V\in \cC(V)$. 
\item\label[Haxiom]{O:T} \textbf{Transversality:} 
\begin{enumerate}
\item\label[HaxiomS]{O:TC:Behrstock}  If $V\pitchfork W$, then there are specified points $\rho^V_W \in \cC(W)$ and $\rho^W_V \in \cC(V)$ such that for all $x\in \cX$ 
$$\min\{\,d_W(\pi_W(x), \rho^V_W),\, d_V(\pi_V(x), \rho^W_V)\,\} \leq \kappa_0.$$
\item\label[HaxiomS]{O:TC:Rhos} If $U \sqsubsetneq V$ and both $\rho^U_W$ and $\rho^V_W$ are defined, then  
$d_W(\rho^U_W, \rho^V_W) \leq \kappa_0.$ 
\end{enumerate}
\item\label[Haxiom]{O:LL} \textbf{Large links:} For any $P\geq 0$ there exists  $N\geq 0$ such that for any $V$ and any $x,y\in \cX$  with $$d_V(\pi_V(x),\pi_V(y))\leq P$$  there is a set of at most $N$ domains $U_i\sqsubsetneq V$ such that if $T\sqsubsetneq V$ is not nested in  one of the $U_i$ then $d_T(\pi_T(x),\pi_T(y))\leq E$. 
\item\label[Haxiom]{O:BGI} \textbf{Bounded geodesic image:} For all $x, y \in \cX$ and $U\sqsubsetneq V$, if there is a geodesic in $\cC(V)$ from $\pi_V(x)$ to $\pi_V(y)$ that stays $E+2\delta$ away from $\rho^U_V$, then $$d_U(\pi_U(x), \pi_U(y))\leq E.$$
Here $\delta$ is such that all $\cC(U)$ are $\delta$-hyperbolic. 
\item\label[Haxiom]{O:R}  \textbf{Partial realization:}  For any  orthogonal family $\{V_j\}$ and points $p_j \in \cC(V_j)$, there exists a point $x\in \cX$ such that
$$d_{V_j}(\pi_{V_j}(x), p_j)\leq \alpha$$
for all $j$, and such that for each $j$ and each $V$ with  $\rho^{V_j}_V$ defined, 
$$d_V(\pi_V(x), \rho^{V_j}_V)\leq \alpha.$$ 
\item\label[Haxiom]{O:U} \textbf{Uniqueness:} For all $\kappa\geq 0$ there exists $\omega\geq 0$ such that if $d(x,y)\geq \omega$ then $$d_U(\pi_U(x), \pi_U(y))\geq \kappa$$ for some $U\in \IndexBig$. 
\end{enumerate}
\end{definition}

The minor changes we have made are the following: we have chosen to use points everywhere instead of sets of uniformly bounded diameter following the discussion in \cite[Section 3]{SistoWhatIs}; we have incorporated the coarse surjectivity assumption of \cite[Section 1.3]{BehrstockHagenSistoHHSII} into the partial realization axiom; and we have modified the large links axiom following \cite[Remark 1.5]{BehrstockHagenSistoHHSII} and \cite[Remark 2.10]{Russell}. 

Note that ``$\rho^U_V$ is defined'' means that either $U\sqsubsetneq V$ or $U\pitchfork V$.

\subsection{The large links axiom} 
Consider the following less uniform version of the large links axiom, where the uniform constant $E$ has been replaced with a constant $R$ that can depend on $P$. 

\begin{enumerate}
\item[(3')]\label[Haxiom]{O:LL2} \textbf{Less uniform large links:} For any $P\geq 0$ there exists  $N\geq 0$ and $R\geq 0$ such that for any $V$ and any $x,y\in \cX$  with $$d_V(\pi_V(x),\pi_V(y))\leq P$$  there is a set of at most $N$ domains $U_i\sqsubsetneq V$ such that if $T\sqsubsetneq V$ is not nested in  one of the $U_i$ then $d_T(\pi_T(x),\pi_T(y))\leq R$. 
\end{enumerate}

We remark that this change in uniformity is not significant. 

\begin{lemma}\label{L:LLuniformity}
In the presence of the other axioms, the less uniform large links axiom implies the usual large links axiom, possibly after increasing the constant $E$.  
\end{lemma} 

The proof does not use hyperbolicity of $\cC(U)$ when $U$ is nest minimal. 

\begin{proof}
Fix even constants $P_0\gg L \gg \max(E, \delta, \alpha)$. 
Apply the less uniform large links with $P_0$ to obtain an $N_0$ and an $R_0$ as above. We will now show that the usual large links axiom holds if $E$ is increased to $E'= 2E + R_0.$ 

Let $P>0$ be arbitrary, and suppose $d_V(\pi_V(x), \pi_V(y))\leq P$. If $d_V(\pi_V(x), \pi_V(y))\leq P_0$ there is nothing new to prove, so assume $d_V(\pi_V(x), \pi_V(y))> P_0$.

Fix a geodesic from $\pi_V(x)$ to $\pi_V(y)$. Pick points $p_0, \ldots, p_d$ on this geodesic, so $d_V(\pi_V(x), p_i)=i$ and $d_V(p_d, \pi_V(y))< 1$. 
Using partial realization, pick points $$x=x_0, x_1, \ldots, x_d = y$$ in $\cX$ so that $\pi_V(x_i)$ lies within $\alpha$ of $p_i$.

For each valid pair $i, i+L$ of indices, apply  the less uniform large links  to the points $x_i, x_{i+L}$ to get a set $\cU_i$ of at most $N_0$ domains properly nested in $V$. We claim that we can use the set $\cU = \bigcup_i \cU_i$, which has size at most $d N_0$, to prove the usual version of large links. We use $N= (P+1) N_0$, which is greater than $d N_0$ and does not depend on the choice of $x,y$. 

Indeed, suppose $T \sqsubsetneq V$ and $T$ is not nested in any of the domains of $\cU$. If the chosen geodesic from $\pi_V(x)$ to $\pi_V(y)$ does not come within $E+2\delta$ of $\rho^T_V$, the bounded geodesic image axiom shows that $d_T(\pi_T(x), \pi_T(y)) \leq E < E'$. So suppose that some $p_j$ is within $E+2\delta+1$ of $\rho^T_V$. 

We leave the case when $p_j$ is close to the beginning or end of the geodesic to the reader. Otherwise, let $i=j-L/2$, and consider $x_{i}$ and $x_{i+L}$. If $k\leq i$ or $k\geq i+L$, then the triangle inequality gives
$d_V(\pi_V(x_k), \rho^T_V)\gg E+2\delta.$
Hence the bounded geodesic image axiom gives that $\pi_T(x_i)$ is within $E$ of  $\pi_T(x)$  and $\pi_T(x_{i+L})$ is within $E$ of $\pi_T(y)$. This uses hyperbolicity of $\cC(V)$, to ensure that a geodesic from $\pi_V(x_i)$ to $\pi_V(x)$, and a geodesic from $\pi_V(x_{i+L})$ to $\pi_V(y)$, stays close to the appropriate segments of the chosen geodesic from $\pi_V(x)$ to $\pi_V(y)$. Since $T$ is not nested in any domain in $\cU_i$, we get that $d_T(\pi_T(x_i), \pi_T(x_{i+L}))\leq R_0$. So the triangle inequality gives 
$$d_T(\pi_T(x), \pi_T(y))\leq E + R_0 + E = E'$$
as desired. 
\end{proof}

\subsection{The proof of \cref{P:Criterion}.} 

We start by explaining how to add domains to a nearly valid index set to make it into a valid index set. This can be compared to, but is easier and less general than, the appendix of \cite{AbbottBehrstockDurham}. (We cannot use the appendix of \cite{AbbottBehrstockDurham} because of \cite[Remark 3.4]{AbbottBehrstockRussellStructure}.)

\begin{definition}\label{D:IndexBig}
If $\IndexSmall$ is a nearly valid index set, we define $\IndexBig$ to be the set of non-empty  orthogonal subsets of $\IndexSmall$. We say $\mbU, \mbV\in \IndexBig$ are orthogonal if every domain in $\mbU$ is orthogonal to every domain in $\mbV$, and we say $\mbU$ is nested in $\mbV$ if every domain in $\mbU$ is nested in some domain of $\mbV$. As always, two domains in $\IndexBig$ are asynchronous if they are not orthogonal and neither is nested in the other. 
\end{definition}

\begin{example}
Say $U_1\sqsubsetneq V$ and $V \perp U_2$ are domains in $\IndexSmall$. Then $\mbU=\{U_1, U_2\}$ is asynchronous to $\mbV=\{V\}$ (even though no domain in $\mbU$ is asynchronous to any domain in $\mbV$).  
\end{example}

We view $\IndexSmall$ as a subset of $\IndexBig$ via the inclusion $U\mapsto \{U\}$, and note that the orthogonality and nesting relations on $\IndexBig$ extend those on $\IndexSmall$. We will typically use regular letters like $U$ for elements of $\IndexSmall$ and bold letters like $\mbU$ for elements of $\IndexBig$, except that sometimes, keeping the inclusion in mind, we may write $U$ for $\{U\}$.

The point of this construction is the following. 

\begin{lemma}\label{L:Valid}
If $\IndexSmall$ is a nearly valid index set then $\IndexBig$ is a valid index set. 
\end{lemma}

\begin{proof}
It is immediate that the orthogonality relation on $\IndexBig$ is symmetric and anti-reflexive, and that the nesting relation is a partial order. If $\Sigma\in \IndexSmall$ is nest maximal, then $\{\Sigma\} \in \IndexBig$ is nest maximal. 


The finite height axiom follows from the corresponding statement for $\IndexSmall$ and the fact that orthogonal sets of domains in $\IndexSmall$ have  bounded size. The coherence axiom is left as an exercise, so it only remains to check the containers axiom. 

Suppose $\mbU=\{U_i\}$ is nested in $\mbV = \{V_j\}$. First observe that it suffices to construct containers when $\mbV=\{V\}$ has size 1, since containers for $\mbU$ in the various $V_j$ can be assembled into a container for $\mbU$ in $\mbV$. 

To build a container for $\{U_1, \ldots, U_r\}$ in $\{V\}$, start by applying the weak container axiom to the pair $U_1, V$, to find an orthogonal tuple of domains $\{C_1, \ldots, C_s\}$ properly nested in $V$. Note that each $U_i, i>1$ is nested in a unique $C_j$, so one can iterate the procedure. A stronger fact is proven in \cref{L:CleanContainers} so here we leave the details to the reader. 
\end{proof}

With this in hand, we can proceed to the proof. 
\begin{proof}[Proof of  \cref{P:Criterion}]
We let $\IndexBig$ be as in \cref{D:IndexBig}, so $\IndexBig$ is a valid index set by \cref{L:Valid}. 
For $\mbU\in \IndexBig$, we define $\cC(\mbU)$ and $\pi_\mbU$ as follows: If $\mbU=\{U\}$, then $\cC(\mbU)=\cC(U)$ and $\pi_\mbU=\pi_U$, and otherwise define $\cC(\mbU)$ to be a point and $\pi_\mbU$ to be the constant map. It now suffices to check the 6 axioms of \cref{D:HHS}. We will view $\IndexSmall$ as a subset of $\IndexBig$ in what follows. 

We will need points $\rho^\mbU_\mbV \in \cC(\mbV)$ when $\mbU \pitchfork \mbV$ or $\mbU \sqsubsetneq  \mbV$. If $\mbV\notin \IndexSmall$, then $\cC(\mbV)$ is a point, and $\rho^\mbU_\mbV$ is that point. So assume $\mbV = \{V\}$ is a singleton, and $\mbU=\{U_i\}$. To have definitions that are coarsely natural, we first note the following. 

\begin{lemma}\label{L:WellDefined}
Suppose $U_1\perp U_2$, and each $U_i$ is properly nested in or asynchronous to $V$. Then 
$\diam \pi_V(\cA_{U_1}\cup \cA_{U_2})\leq 2B.$
\end{lemma}

\begin{proof}[Proof of \cref{L:WellDefined}]
The density assumption gives $\cA_{U_1}\cap \cA_{U_2}\neq \emptyset$ so the result follows from the bounded projections assumption. 
\end{proof}

First suppose $\mbU=\{U_i\}$ is nested in $\mbV=\{V\}$. Define $\rho^\mbU_\mbV$ to be any point in $\pi_V(\cup \cA_{U_i})$. Since this set has diameter at most $2B$ by \cref{L:WellDefined}, this is coarsely well defined. 

Next, suppose $\mbU=\{U_i\}$ is asynchronous to $\mbV=\{V\}$. We first claim that at least one $U_i$ must be properly nested in or asynchronous to $V$. Otherwise, for each $i$ either $V\sqsubseteq U_i$ or $V\perp U_i$, and this would give the contradiction that either $\mbV \sqsubseteq \mbU$ (if $V\sqsubseteq U_i$ for some $i$) or $\mbV \perp \mbU$. We let  $S$ be the set of $i$ with $U_i$ properly nested in or asynchronous to $V$ and define $\rho^\mbU_\mbV$ to be any point in $\pi_V(\cup_{i\in S} \cA_{U_i})$. Since this set has diameter at most $2B$ by \cref{L:WellDefined}, this is coarsely well defined. We can now proceed.

\begin{proof}[Proof of the transversality axiom]
For part (a), we first note that the inequality is trivial unless $\mbU=\{U\}$ and $\mbV=\{V\}$ are both singletons, so assume this is the case. The result is now given by part of the structured redundancy assumption.

Part (b) follows from the final part of the bounded projections assumption. 
\end{proof}

\begin{proof}[Proof of the bounded geodesic image axiom]
For this to fail,  we must have that $d_\mbU(\pi_\mbU(x), \pi_\mbU(y))$ is large and $\rho^\mbU_\mbV$ is far from a geodesic from $\pi_\mbV(x)$ to $\pi_\mbV(y)$. So it suffices to consider the case when both domains $\mbU=\{U\}$, $\mbV=\{V\}$ are in $\IndexSmall$. The result is now given by part of the structured redundancy assumption. 
\end{proof}

\begin{proof}[Proof of the partial realization axiom]
This follows immediately from the density assumption. 
\end{proof}

\begin{proof}[Proof of the large links axiom] We prove the weaker version, which is sufficient by \cref{L:LLuniformity}. 

Fix $P\geq 0$. The progress localization assumption gives that there is a $N_{PL}$ and an $R>0$  such that for any $V\in \IndexSmall$ and any $x,y\in \cX$ with $d_V(\pi_V(x),\pi_V(y))\leq P$  there is a set of at most $N_{PL}$ domains $U_i\sqsubsetneq V$ such that if $T\sqsubsetneq V$ is not nested in or orthogonal to one of the $U_i$ then $d_T(\pi_T(x),\pi_T(y))\leq R$. 

We claim that the statement in large links holds with $N=\max(2N_{PL}, K)$, where $K$ is the maximum size of a set of pairwise orthogonal domains. Indeed, consider $\mbV$ and  $x,y\in \cX$ with $d_{\mbV}(\pi_{\mbV}(x),\pi_{\mbV}(y))\leq P$. 

First consider the case when $\mbV=\{V\}$ is a singleton. Consider the set of domains $U_i\sqsubsetneq V$ given by progress localization. Adding the container domains for the $U_i$ in $V$, as described in \cref{A:Container}, gives the required set of at most $N$ domains.

Next consider the case when $\mbV=\{V_1, \ldots, V_k\}$ is not a singleton. Then we  simply define $\{U_i\}$ to be the set of the $V_i$, now viewed as domains in $\IndexBig$ via the usual implicit map $V_i \mapsto \{V_i\}$.  
\end{proof}

This concludes the proof, because the uniqueness axiom is the same as the enough projections assumption.
\end{proof}

\subsection{The definition of an HHG}

Suppose we have a self-map $f: \cX \to \cX$ of an HHS, where we use the notation above for the HHS. Suppose $K$ is a constant. Suppose  we have a bijection $$f^\diamond : \mfS \to \mfS$$ preserving nesting, orthogonality, and transversality, and, for each $U\in \mfS$ we have an isometry $$f^*(U): \cC(U) \to \cC( f^\diamond(U)).$$ Suppose that, for all $x\in \cX$, 
$$d_{f^\diamond(U)}(\pi_{f^\diamond(U)}(f(x)), f^*(U) ( \pi_U(x))) \leq K,$$ 
and, whenever $U$ is nested in or asynchronous to $V$, then 
$$d_{f^\diamond(V)}(f^*(V) \rho^U_V, \rho^{f^\diamond(U)}_{f^\diamond(V)}) \leq K.$$
Then we say that $f$, together with data of the associated maps, is an automorphism of the HHS. 

There is a natural notion of composing automorphisms, which involves composing the associated data as well as the actual self-maps  $f$. We say that a group acts on an HHS if there is a homomorphism from the group to the group of automorphisms of the HHS. 

Following \cite[Section 1.7]{BehrstockHagenSistoHHSII}, to prove that a group is hierarchically hyperbolic, it suffices to prove that it acts properly and cocompactly by isometries on an HHS and that the induced action on $\mfS$ has finitely many orbits.
 (The actual definition allows for quasi-actions in addition to actions, and allows for  metrically proper and cobounded actions in addition to actions that are proper and cocompact by isometries.)

We note that \cite[Section 1.7]{BehrstockHagenSistoHHSII} does not use the variant axioms of \cite[Section 1.3]{BehrstockHagenSistoHHSII} that we follow, but the proof of the equivalence in \cite[Proposition 1.11]{BehrstockHagenSistoHHSII} can easily be adapted to bridge this difference. 

\subsection{Proof of \cref{P:CriterionGroups}} 

The proof is almost immediate. The group $G$ acts on the enlarged index set $\IndexBig$ via $$g^\diamond\{U_1, \ldots, U_k\} = \{g^\diamond U_1, \ldots, g^\diamond U_k\}.$$ For a domain $V$ in $\IndexBig- \IndexSmall$, for any $g\in G$ we have that both $\cC(V)$ and $\cC(g^\diamond V)$ are points, and the required isometry $g(V):\cC(V)\to \cC(g^\diamond V)$ is the constant map. 

The requirement that 
$$d_{f^\diamond(V)}(f^*(V) \rho^U_V, \rho^{f^\diamond(U)}_{f^\diamond(V)}) \leq K$$
follows from the definition of the $\rho$ points and \cref{L:WellDefined}.

\begin{remark}\label{R:characterization}
Although the criteria given are only sufficient conditions for hierarchical hyperbolicity, it seems straightforward to modify them to actually characterize hierarchical hyperbolicity. To show an appropriately modified criterion can apply to an arbitrary HHS, one might define the active regions to be the product regions of \cite{BehrstockHagenSistoHHSII}, and apply \cite[Proposition 4.24]{RussellSprianoTran}, which is a corrected form of a result in \cite{BehrstockHagenSistoHHSII}. If $U\sqsubsetneq V$, to obtain a coarser form of $\cA_U \cap \cA_V\neq\emptyset$, one might want to note that the product region of $U$ contains a point not far from a point of the product region of $V$, which (following a suggestion of Sisto) might be done by showing that for any point of the HHS the result of projecting (``gating'') it  to the product region of $U$ and then to the product region of $V$ is close to the result of applying the projections in the opposite order.
\end{remark}

\section{The combinatorics of the index set (not required) }\label{A:IndexSet}

In this appendix, which is not used in  the paper, we explore the combinatorial properties of the  sets $\IndexSmall$ and $\IndexBig$ together with the relations of nesting and orthogonality. For example, as noted in \cref{R:lattice}, the set $\IndexBig$ with the nesting partial order is, after adding a smallest element, a lattice (a  poset with greatest lower bounds and least upper bounds). This appendix shows that the HHG structure we produce is extremely similar to the usual HHG structure on mapping class groups, and allows us to upgrade \cref{T:main} as follows. 

\begin{theorem}\label{T:Upgrade}
The HHG structure produced in the proof of \cref{T:main} is colorable and has wedges, clean containers, and the orthogonals for non-split domains property. 
\end{theorem}

The definitions of these terms will be recalled below. 

\begin{proof}
Colorability is checked in \cref{L:colorable}; wedges in \cref{L:Wedges}; clean containers in \cref{L:CleanContainers}; and the orthogonals for non-split domains property in \cref{L:OFNSDP}.
\end{proof}

\begin{remark}
It follows from \cref{T:Upgrade} and \cite[Theorem 1]{HagenMangioniSisto} that our spaces are Combinatorial Hierarchically Hyperbolic Spaces (CHHS) as defined in \cite{BehrstockHagenMartinSisto}.
%
%
We believe it would require new ideas to  check this directly without reproducing much of our analysis as well as parts of the proof of \cite[Theorem 1.1]{HagenMangioniSisto}. Additionally, an approach based on \cite{BehrstockHagenMartinSisto} would not allow our extension to the relative setting in \cref{S:RHHG} (see \cite[Remark 3.17]{HagenMangioniSisto}), and would not provide as much explicit information on the HHS structure.
%
%
\end{remark}
%
%
%

Recall that $\IndexSmall$ is the set of stratum domains of positive dimensional strata, angle domains, center domains, and tree domains;  $\IndexBig$ is the set of subsets $\{U_1, \ldots, U_k\}$ of $\IndexSmall$ consisting of $k\geq 1$ pairwise orthogonal domains of $\IndexSmall$; and we can view $\IndexSmall$ as a subset of $\IndexBig$ via the inclusion $U\mapsto \{U\}$. 

This appendix assumes familiarity with previously introduced material, especially  \cref{fig:rels} and \cref{L:Valid}.

\subsection{Wedges}

Given a poset, the wedge of two elements is their greatest lower bound, if it exists. We will say that a poset has wedges if every pair of elements either has a wedge or there does not exist any element less than both of them. (This is the same as saying wedges always exist after adding a smallest element.) We will show that $\IndexBig$ has wedges, but first we need the following.

\begin{lemma}\label{L:WedgePrep}
Given $U, V \in \IndexSmall$, 
there is a unique subset $\{W_1, \ldots, W_k\}\subset \IndexSmall$ of pairwise orthogonal domains  such that
\begin{enumerate}
\item  all $W_i$ are nested in both $U$ and $V$, and
\item every domain nested in both $U$ and $V$ is nested in one of the $W_i$. 
\end{enumerate}
\end{lemma}
We call $\{W_1, \ldots, W_k\}$ a wedge set for $U$ and $V$. We allow $k=0$, and note that the wedge set is  empty exactly when there are no domains nested in both $U$ and $V$.

\begin{proof}
 It suffices to consider pairs $U,V$ with neither  nested in the other, since if $U \sqsubseteq V$ then $U$ is the wedge of $U$ and $V$. We also assume that there is something nested in both $U$ and $V$. We proceed in cases according to the type of $U$ and the type of $V$. There are only three cases. 

\bold{Case 1: stratum, stratum.} The domains nested in a stratum domain are: smaller stratum domains; angle domains of orthogonally intersecting hyperplanes; center domains of horoballs the stratum enters, and some tree domains of horoballs the stratum enters (those not equal to a point in the tree domain factor). 

\bold{Subcase 1a:} The closures of the two strata intersect. In this case, let $W$ be the stratum domain associated to the generic stratum of the intersection of the closures; the analysis below, using the assumption that there is something nested in both $U$ and $V$, will show that this intersection is not a point. We claim that $W$ is the wedge of $U$ and $V$, or in other words that $\{W\}$ satisfies the requirements of the lemma.

Now suppose that $R$ is a domain nested in $U$ and $V$. We want to show that $R$ is nested in $W$. Considering $U,V$ as strata, we can consider their closures $\ol{U}, \ol{V}$ and proceed as follows. 
\begin{enumerate}
\item If $R$ is a stratum domain, then $\ol{R} \subset \ol{U}$ and $\ol{R}\subset \ol{V}$ and it suffices to note that $\ol{R} \subset \ol{U}\cap \ol{V}$. 
\item If $R$ is an angle domain, then the associated hyperplane intersects both $\ol{U}$ and $\ol{V}$ orthogonally. We know $\ol{U}$ and $\ol{V}$ intersect orthogonally along $\ol{W}$. \cref{L:TripleIntersection,R:OrthIntersection} gives that the hyperplane associated to $R$ intersects $\ol{W}$. 
\item If $R$ is a center domain, then both $U$ and $V$ enter the associated horoball.  \cref{L:TripleIntersection} gives that $W$ also enters the horoball. 
\item If $R$ is a tree domain, then both $U$ and $V$ enter the associated horoball, and again \cref{L:TripleIntersection} gives that $W$ also enters the horoball. 
\end{enumerate}
In all cases one can deduce that $R$ is nested in $W$. 

\bold{Subcase 1b:} Both strata enter a horoball, but their closures $\ol{U}$ and $\ol{V}$ do not intersect. First note that there cannot be more than one horoball that both strata enter, since otherwise the geodesic between the associated cusps would show that $\ol{U}\cap \ol{V}\neq \emptyset$, giving a contradiction. 

Here are the domains nested in both $U$ and $V$:
\begin{enumerate}
\item the center domain of the cusp,
\item each tree domain of this cusp associated to a factor $\Chat_i$ where both $U$ and $V$ are not constant,  and 
\item angle domains nested in those tree domains.  
\end{enumerate} 

The wedge set of $U$ and $V$ is the center and tree domains above. 

\bold{Case 2: stratum, tree.} Only angle domains can be properly nested in tree domains, so there is an angle domain $N^1(\Hhat)$ nested in both $U$ and $V$. The $\Hhat$ intersects $\ol{U}$ orthogonally, and corresponds to a point in the tree corresponding to the tree domain $V$. 

\bold{Subcase 2a:} The stratum enters the horoball associated to $V$. In this case the stratum is given by fixing the point coordinate in some of the tree factors -- but, since $U$ is not a subset of $\Hhat$, not the factor corresponding to the given tree domain. In this case the given tree domain is actually nested in the given stratum domain. So this case does not occur because of our assumption that neither of $U,V$ is nested in the other. 

\bold{Subcase 2b:} The stratum $U$ does not enter the horoball associated to $V$. We know $N^1(\Hhat)$ is nested in $U$ and $V$. Suppose another angle domain $N^1(\Hhat')$  is nested in $U$ and $V$. We claim that $\Hhat$ and $\Hhat'$ cannot be disjoint. Otherwise we can contradict the fact that closest point projection is distance non-increasing, because $\Hhat\cap \ol{U}$ is the closest point projection of $\Hhat$ to $\ol{U}$ (\cref{R:OrthIntersection}) and similarly for $\Hhat'$, and because $\Hhat$ and $\Hhat'$ enter the same horoball and hence get arbitrarily close. 

Thus the collection of domains nested in $U$ and $V$ is a set of pairwise orthogonal angle domains, and this set is hence the wedge set.   

\bold{Case 3: tree, tree.} To have something nested in both it has to be an angle domain. Since we have assumed $U$ and $V$ are distinct, this means that the two tree domains are associated with different cusps. There is a unique complex line joining any two cusps, and any hyperplane entering both horoballs must contain this line. There is a unique hyperplane that both contains this line and is contained in the correct parallelism class at each cusp, and the wedge set should be the associated single angle domain. 
\end{proof}

\begin{lemma}\label{L:Wedges}
$\IndexBig$ has wedges.
\end{lemma}
\begin{proof}
Using  \cref{L:WedgePrep}, we claim the wedge of $\{V_i\}$ and $\{U_j\}$ is the union $\{W_k\}$ over $i,j$ of the wedge set of $V_i$ and $U_j$. 

First, we show that $\{W_k\}\in \IndexBig$. Suppose $W_{k_1}, W_{k_2}$ are in this $\{W_k\}$ and are distinct. We must show they are orthogonal.  Say $W_{k_1}$ is in the wedge set of $V_{i_1}$ and $U_{j_1}$ and $W_{k_2}$ is in the wedge set of $V_{i_2}$ and $U_{j_2}$. If $V_{i_1} \neq V_{i_2}$ or $U_{j_1} \neq U_{j_2}$  then the coherence condition in \cref{D:Valid} gives that $W_{k_1} \perp W_{k_2}$. In the remaining case we get $W_{k_1} \perp  W_{k_2}$ from the definition of a wedge set. 

Next, we show that $\{W_k\}$ is nested in both $\{V_i\}$ and $\{U_j\}$. Fix $k$, and say $W_k$ is part of the wedge tuple for  $V_i$ and $U_j$. So $W_k$ is nested in $V_i$ and $U_j$. Since this is true for all $k$ this gives the desired result. 

Finally, we consider a $\{Q_r\}\in \IndexBig$ that is nested in $\{V_i\}$ and $\{U_j\}$ and show $\{Q_r\}$ is nested in $\{W_k\}$. Fix $r$. We know $Q_r$ is nested in some $V_i$ and some $U_j$. So $Q_r$ is nested in an element of the wedge tuple of $V_i$ and $U_j$. Since this is true for all $r$ this gives the desired result. 
\end{proof}

\subsection{Clean containers}\label{SS:clean}

We now show the following. 

\begin{lemma}\label{L:CleanContainers}
Suppose $\{U_i\}, \{V_j\}\in \IndexBig$ and $\{U_i\}$ is nested in $\{V_j\}$. Suppose there is at least one domain nested in $\{V_j\}$ and orthogonal to $\{U_i\}$. Then there is a domain $\{W_k\}\in \IndexBig$ that is nested in $\{V_j\}$ and orthogonal to $\{U_i\}$, and such that every other such domain is nested in $\{W_k\}$. 
\end{lemma}

One thinks of $\{W_k\}$ as the orthogonal complement of $\{U_i\}$ in $\{V_j\}$. The statement of  \cref{L:CleanContainers} is what is known in the HHG literature as $\mfS$ having clean containers \cite[Definition 7.1]{AbbottBehrstockDurham}. 

\begin{proof}
First we prove this when $\{V_j\} = \{V\}$ is a singleton. The clean container in this case is the wedge (over $i$) of the orthogonal tuples produced by \cref{L:WeakContainers} for each $U_i\sqsubsetneq V$.  

With this having been proven, it suffices to note in general that the desired clean container is the orthogonal set obtained as the union over $j_0$ of the clean container for $\{U_i\}$ in the singleton $\{V_{j_0}\}$. 
\end{proof}

\begin{remark}\label{R:lattice}
As pointed out in \cite[Remark 10.21]{HagenMangioniSisto}, the existence of wedges together with the finite height assumption in \cref{D:Valid} implies the existence of least upper bounds (joins). Moreover, $\mfS$ with one smallest element added is a complete lattice. 
\end{remark}

\subsection{Orthogonals for non-split domains}

We start in the context of $\IndexSmall$. Recall that angle and center domains are always nest-minimal. The remaining domains have attached to them a certain very special set of domains that we call a halo. 
\begin{enumerate}
\item If $U$ is a stratum domain, define its halo $\partial U$ to be the set of angle domains corresponding to hyperplanes containing $U$. If $U$ has codimension $k$, this set has size $k$. 
\item If $U$ is a tree domain, define its halo $\partial U$ to be the set consisting only of the associated center domain. 
\end{enumerate}
In either case, we note that $\partial U$ consists of nest-minimal domains that are pairwise orthogonal. Except if $U$ is the nest maximal domain, $\partial U$ is non-empty when it is defined. We do not define the halo for angle and center domains.

Halos have the following two remarkable properties. 

\begin{lemma}\label{L:Halo}
Suppose $U$ is a tree domain or a stratum domain, and suppose $U\sqsubsetneq V$. Then at least one element of $\partial U$ is nested in $V$. 
\end{lemma}

\begin{proof}
If $U$ is a tree domain, then $V$ is a (particular kind of) stratum domain entering the horoball of $U$, so the associated center domain is nested in $V$. 

If $U$ is a stratum domain, then $V$ is a bigger stratum domain. There exists a hyperplane that contains $U$ but not $V$, and the associated angle domain is nested in $V$.  
\end{proof}

\begin{lemma}\label{L:Halo2}
Suppose $U$ is a tree domain or a stratum domain, and suppose that $H\in \partial U$ and $V\perp U$. Then $V=H$ or $V\perp H$. 
\end{lemma}

\begin{proof}
If $U$ is a tree domain, then $V$ is an angle, center, or tree domain, and $H$ is the center domain. If $U$ is a stratum domain, then $V$ is an angle domain and $H$ is an angle domain. Either way, the result follows from examining the orthogonality relation. 
\end{proof}

We now turn to $\IndexBig$. Recall from \cite[Definition 3.6]{HagenMangioniSisto} that a $\{U_i\}\in \IndexBig$ is split if there exists $\{V_i\}$ nested in $\{U_i\}$ such that every domain nested in $\{U_i\}$ is either orthogonal to $\{V_i\}$ or has $\{V_i\}$ nested in it. 
%
%

\begin{remark}\label{R:split}
If one of the $U_i$, say $U_1$, is nest-minimal, then $\{U_i\}$ is split, as seen by taking $\{V_i\} = \{U_1\}$. 
\end{remark}

Recall from \cite[Definition 3.9]{HagenMangioniSisto} that $\IndexBig$ has the orthogonals for non-split domains property if whenever $\{U_i\}\sqsubsetneq \{W_k\}$ and $\{U_i\}$ is not split, then there exists a domain nested in $\{W_k\}$ and orthogonal to $\{U_i\}$. 

\begin{lemma}\label{L:OFNSDP}
$\IndexBig$ has the orthogonals for non-split domains property. 
\end{lemma}

\begin{proof}
Suppose $\{U_i\}\sqsubsetneq \{W_k\}$ and $\{U_i\}$ is not split. If there is some $k_0$ such that $W_{k_0}$ does not have any of the $U_i$ nested in it, it suffices to note that the singleton $\{W_{k_0}\}$ is orthogonal to  $\{U_i\}$. Otherwise, without loss of generality $U_1 \sqsubsetneq W_1$. By \cref{R:split}, $U_1$ is not nest-minimal. \cref{L:Halo} gives a $Q\in \partial U_1$ that is nested in $W_1$. This $Q$ is nest-minimal so it cannot be equal to any of the $U_i$, and it suffices to note that the singleton $\{Q\}$ is orthogonal to $\{U_i\}$.
\end{proof}

\subsection{Colorability}

Recall that an HHG is called colorable if each domain can be colored by one of finitely many colors in such a way that domains of the same color are asynchronous and such that the action of the HHG on the set of domains induces an action of the HHG on the set of colors \cite{Hagen}. 

\begin{lemma}
To prove $\IndexBig$ is colorable it suffices to prove that $\IndexSmall$ is colorable. 
\end{lemma}

\begin{proof}
Suppose $\IndexSmall$ is colorable by a set of colors $\cC$. We then color $\IndexBig$ with a new set of colors equal to the power set $2^\cC$. The color of $\{U_i\}$ is the set of colors of the $U_i$. 

Suppose $\{U_i\} \neq  \{V_j\}$, and assume without loss of generality that $V_1$ is not equal to any of the $U_i$. 
If $\{U_i\}$ has the same color as $\{V_j\}$, then without loss of generality $U_1$ has the same color as $V_1$, which means that $U_1$ and $V_1$ are asynchronous. It follows that $\{U_i\}$ and $\{V_j\}$ are asynchronous.
\end{proof}

\begin{lemma}\label{L:HypColor}
The set of hyperplanes can be colored with finitely many colors such that any two hyperplanes of the same color are disjoint. 
\end{lemma}

For a more precise result on the moduli space of cubic surfaces, see  \cite[Lemma 7.30, Theorem 7.32]{AllcockCarlsonToledoSurfaces}.

\begin{proof}
A very general argument, given in \cite[page 113]{Bergeron} or \cite[Lemma 1.8]{BergeronHaglundWise}, gives that if $H\in \cH$, then $\Stab_\Gamma(H)$ is a separable subgroup of $\Gamma$.

As in \cite[Proof of Theorem 1.1]{DiCerboStover}, ideas of Scott \cite{Scott} can then be adapted to find a finite Galois cover of $\Gamma\back B$ where the image of $H$ does not intersect itself. (Some care is required in the non-uniform case: See \cite[Theorem 1.2]{GarlandRaghunathan}, or adapt arguments in  \cite{Bergeron} to the complex hyperbolic case.)
%
%
In the associated Galois cover the same holds for every hyperplane in $\Gamma H$. We can then take a common Galois cover of one such cover for every $\Gamma$ orbit in $\cH$. The result is a Galois cover of $\Gamma\back B$ where the image of every hyperplane in $\cH$ does not intersect itself. 
%
%

We then color each hyperplane by a set of colors equal to the images of elements of $\H$ in this cover. This gives the result. 
\end{proof}

\begin{corollary}
The set of stratum domains can be colored so any two of the same color have disjoint closures.
\end{corollary}

\begin{proof}
One can color a stratum by the set of colors of the hyperplanes containing it.
\end{proof}

\begin{lemma}
The set of cusp domains can be colored with finitely many colors such that any two  of the same color are asynchronous. 
\end{lemma}

\begin{proof}
All the center domains can be assigned the same new color. Color each tree domain, viewed as a parallelism class of hyperplanes, with the set of colors of the hyperplanes it contains. 
\end{proof}

We can now conclude.

\begin{lemma}\label{L:colorable}
 $\mfS$ is colorable.
\end{lemma}

\begin{proof}
It suffices to handle each type of domain separately, and these are handled by the previous lemmas. 
\end{proof}

\section{Relative hierarchical hyperbolicity (not required) }\label{S:RHHG}

For the definition of a relatively hierarchically hyperbolic space (RHHS) and group (RHHG), we recommend consulting \cite[Definition 1.4, Section 1.2.3]{BehrstockHagenSistoAsDim}; the original reference \cite[Definition 6.8]{BehrstockHagenSistoHHSII} is similar but does not discuss the group version explicitly. The definitions of a RHHS and RHHG are identical to those of an HHS and HHG, except that the $\cC(U)$ for $U$ nest-minimal are now allowed to be arbitrary geodesic metric spaces. (The projection $\pi_U$ should still be coarsely onto and coarsely Lipschitz.)

Minor modifications to our proof of \cref{T:main} prove that $\Gammahat$ is an RHHG even if \cref{CuspLimit} on cusps is dropped. \cref{A:OtherEx} provides a number of examples where this is of interest. 

\begin{theorem}
\label{T:mainrel}
Suppose $\H$ and $\Gammahat$ are as in \cref{T:main} except that \cref{CuspLimit} is not assumed to hold. Then $\Gammahat$ is a relatively hierarchically hyperbolic group (RHHG). 
\end{theorem}

Almost all the proof proceeds almost verbatim the same. The main changes concern cusp domains. In a sense there may be fewer tree domains than before, and so the metric spaces $\cC(U)$ attached to center domains $U$ now must sometimes be larger. Their role expands enough that it becomes potentially misleading to continue to call them center domains, and so we will call them Heisenberg domains instead. For the purposes of this appendix, a center domain can thus be thought of as a special case of a Heisenberg domain that arises when \cref{CuspLimit} in fact does hold for the associated cusp.

We now comment in more detail on the sections requiring changes. 

\bold{\cref{S:Criteria} and \cref{S:CritProof}:} One can modify \cref{P:Criterion} and \cref{P:CriterionGroups} simply by dropping the requirement that $\cC(U)$ must be hyperbolic when $U$ is nest-minimal, so they yield relative hierarchical hyperbolicity instead of hierarchical hyperbolicity. The proofs do not change; the only use of hyperbolicity as a tool in \cref{S:CritProof} is in \cref{L:LLuniformity}, which only uses it for non-nest-minimal domains. 

\bold{\cref{S:Swaddling}:} It is natural to work with swaddling in a slightly more general context. Rather than working only with the Lie group $G=\bR$, we now allow other connected Lie groups $G$ with left invariant metrics. Maps that previously were $\bR$-equivariant are now required to be $G$-equivariant. We assume that, for every $x,y,z$, the $G$-equivariant map $m_{z,x} \circ m_{y,z}\circ m_{x,y}$ is multiplication by an element of a fixed subgroup of the center of $G$ that is isomorphic, via a fixed isomorphism, to $\bR$. \cref{L:Swaddled,L:MetricCompatibilityOfSwaddling} and its proofs are valid in this more general context, and we do not require any later results from \cref{S:Swaddling} in this more general context. 

\begin{remark}
The centrality assumption avoids having to work around the fact that a self-map of
a $G$-torsor cannot be canonically identified with an element of $G$; it distinguishes only 
a conjugacy class.
 \end{remark}
%
%
%

\bold{\cref{S:CuspDomains}:} This is the main section requiring changes, which begin in \cref{SS:HyperplaneArrangement}. In comparison to \cref{LemAffineArrangement}, we now have horosphere minus arrangement mod center of the form
$$\bC^k \times \prod_{i=k+1}^{n-1}\Ccirc_i.$$
Unlike before, $k$ can be positive, and even equal to $n-1$ if the hyperplane arrangement doesn't go into the given horoball. The $\bC^k$ does not factor canonically so is best thought of as an abstract vector space of dimension $k$. Our first observation is to determine the replacement for the center of the Heisenberg group. 

\begin{lemma}\label{L:H2k1}
The connected component of the identity of the stabilizer of $\H\cap\partial T_c$ in the Heisenberg group $H^{2n-1}$ is a Heisenberg group $H^{2k+1}$ of dimension $2k+1$.
\end{lemma}

Note that when $k=0$, this recovers the center of $H^{2n-1}$. The proof is left to the reader, and is implicit in \cref{S:CuspDomains}. 

This $H^{2k+1}$ will replace the center $Z$ in most of our analysis. The space $\partial T_{c}$ has the structure of a bundle 
$$H^{2k+1} \to H^{2n-1} \to \prod_{i=k+1}^{n-1}\C_i.$$
It is this bundle and its variants that we use. The fibers are no longer real lines, but instead are torsors for $H^{2k+1}$. Given a path in $\prod_{i=k+1}^{n-1}\C_i$ we can lift it to a horizontal path in $H^{2n-1}$ in two steps. First, lift it to a  path in $\bC^k \times \prod_{i=k+1}^{n-1}\C_i$, simply by keeping the $\bC^k$ coordinate constant. Then, lift to a horizontal path in $H^{2n-1}$. This analysis shows that the monodromy for the bundle $H^{2k+1} \to H^{2n-1} \to \prod_{i=k+1}^{n-1}\C_i$ in fact lives in the center of $H^{2n-1}$, which is equal to the center of $H^{2k+1}$. The monodromy of a loop is again a sum of signed areas, exactly as in \cref{SS:HeisenbergGroup}.

\cref{SS:PullBackToHoroball} requires only notational changes. We get that $\partial\Thatcirc_{\hat{c}}$ is the pullback of
$$\partial \Tcirc_c\to\prod_{i=k+1}^{n-1}\Ccirc_i$$ to the universal cover
$\prod_{i=k+1}^{n-1}\Chatcirc_i$ of its base.  The discussion of tree domains in \cref{SS:TreeDomains} does not require modification, except to say there may now be fewer tree domains and to make straightforward modifications to \cref{LemAffineArrangement}. We get one tree domain for each parallelism class of hyperplanes entering the horoball. 

\begin{remark}\label{R:NotHHG}
If $k>0$, then $\Gammahat$ is not an HHG.  This follows from the Tits alternative
for HHGs, which holds in a strong form: every finitely generated subgroup either is
virtually abelian or contains a nonabelian free group \cite[Theorem 4.1]{DurhamHagenSistoCorrection}.
As in the proof of \cref{LemAffineArrangement}, 
one can show directly that $\Gamma\cap H^{2k+1}$ is a cocompact lattice in $H^{2k+1}$, hence nilpotent but not virtually abelian.
%
%
\end{remark}

The key changes are to \cref{SS:CenterDomains}. Before we had $n-1$ tree domains and one center domain per cusp. Now we have $n-1-k$ tree domains and one  Heisenberg domain. If $U$ is the Heisenberg domain, $\cC(U)$ will still be a swaddling of $\partial \That_{\hat{c}}\to  \prod_{i=k+1}^{n-1}\Chat_i$. The main difference is that now the swaddling is quasi-isometric not to a line but instead to $H^{2k+1}$. The proof of \cref{L:Thatswlip} is essentially unchanged but must be rephrased because the connection one-form is now valued in the Lie algebra of $H^{2k+1}$; compare for example to \cite[Section 5.10]{Hamilton}.

In \cref{SS:HoroballEP}, it is important that \cref{L:HoroballProductMetric} and its proof hold essentially unchanged.


\bold{\cref{S:NTO}:} \cref{C:BoundedChains,L:MaxOrthogonalSets,C:MaxSizeOrth} contain unnecessarily precise claims that now require straightforward modifications. For example, in \cref{C:BoundedChains} it is still true that the maximal size of a chain is at most $n+1$, but the (never used) criterion for equality need not hold if cusps need not contain 1-dimensional strata.
%
%
%
%

No other significant changes are required to the proof.

\begin{remark}\label{R:RelHyp}
$\Gammahat$ is only  relatively hyperbolic in very special situations. Compare to \cite{BehrstockCorneliaMosher, Russell, AbbottBehrstockRussell, BelegradekHruska}.  See also the discussion after \cite[Corollary 8.3]{BehrstockCorneliaMosher} for prior results implying that most mapping class groups are not relatively hyperbolic. 

Suppose $(G, \mfS)$ is an HHG for which the hyperbolic space of the nest-maximal domain is infinite diameter. It follows automatically that the Gromov boundary of this hyperbolic space is non-empty; see, for example,  \cite[Theorem 4.2]{AbbottBalasubramanyaOsin}. Suppose also that there is a $G$-invariant subset $\mfS_0\subset \mfS$, of size greater than 1, such that any two domains in $\mfS_0$ can be joined by a sequence of domains in $\mfS_0$, each orthogonal to the next, and such that the Gromov boundary of $\cC(U)$ is non-empty for each $U\in \mfS_0$. Then \cite[Theorem 6.12]{AbbottBehrstockRussell} gives that $G$ is thick of order 1 and hence, by   \cite[Corollary 7.9]{BehrstockCorneliaMosher}, $G$ is not relatively hyperbolic. 
%
%
This can be applied to show that the $\Gammahat$ associated with the moduli space of cubic surfaces is not relatively hyperbolic. Here one should let $\mfS_0$ be the set of all angle domains, and use the results mentioned in \cref{R:CombinatorialModel}. 
\end{remark}

\section{Other examples (not required) }\label{A:OtherEx}

We give many examples to which \cref{T:main} or \cref{T:mainrel} applies. 
Mainly, we review a number of known constructions that produce
arrangements of mutually orthogonal hyperplanes in the $n$-ball, invariant
under lattices in $PU(n,1)$.

\subsection{The definition of orthogonality.} 

Suppose $V$ is a complex vector space and $h$ is a hermitian form on $V$ of signature $(m,1)$. If $v$ is a vector in $V$ with $h(v,v)>0$, we define the corresponding hyperplane to be 
$$H_v = \{[w] \in \bP V : h(w,v)=0, h(w,w)<0\}.$$
Note that the condition $h(v,v)>0$ is exactly the condition required to make this a codimension 1 subspace of $\bC\bH^m$ (rather than being all of $\bC\bH^m$ or empty). We say $v$ is positive when $h(v,v)>0$. The first basic observation is the following. 

\begin{lemma}\label{L:BasicOrth}
Suppose $r, s\in V$ are non-colinear positive vectors. Then $H_{r} \cap H_s\neq\emptyset$ if and only if 
$$|h(r,s)|^2 < h(r,r) h(s,s),$$
and they intersect orthogonally if and only if $h(r,s)=0$. 
\end{lemma}

For convenience we recall the proof of the first claim. Compare to the proof of \cite[Lemma 7.28]{AllcockCarlsonToledoSurfaces}.

\begin{proof}
By definition, $H_{r} \cap H_s\neq\emptyset$ if and only if there is a negative vector orthogonal to both $r$ and $s$. Such a vector exists if and only if the restriction of $h$ to the span of $r$ and $s$ is positive definite, or in other words if and only if the matrix 
$$\begin{pmatrix} h(r,r) & h(r,s) \\ h(s,r) & h(s,s) \end{pmatrix}$$
is positive definite. Since $h(r,r)>0$, this is the case if and only if the determinant $h(r,r) h(s,s) - |h(r,s)|^2$ is positive. 
\end{proof}

\subsection{A unit norm criterion for orthogonality.} 
\label{SS:UnitNorm}

An arithmetic lattice ``of the simplest type'' in $PU(m,1)$ can be defined as follows. Let $E\subset \bC$ be a CM field, let $F\subset \bR$ be its totally real subfield, and let $\cO_E$ and $\cO_F$ be the rings of integers of $E$ and $F$. Let $V=E^{m+1}$, and let $$h: V\times V \to E$$ be a hermitian form of signature $(m,1)$.  Each real embedding $\sigma: F\to \bR$ gives a pair of complex conjugate maps $E\to \bC$, and for all $\sigma$ not the identity we assume the corresponding pair of Galois conjugates of $h$ have signature $(m+1,0)$. With these assumptions,
$$PU(\cO_E,h)\subset PU(V\otimes \bC, h)$$
is a lattice, and it is exactly lattices commensurable with such $PU(\cO_E,h)$ which are called ``of the simplest type''. 
\begin{remark}\label{R:SimplestTypeCocompact}
This lattice is cocompact if $F\neq \bQ$. 
\end{remark}
Note that arithmetic lattices not of the simplest type involve non-commutative central simple division algebras with involutions of the second kind. 
See for example the brief discussion in \cite[Section 3.2]{IsenrichPy}, or the discussion in \cite{EmeryStover,StoverToledo} as well as earlier sources.

In this common setup we have the following orthogonality criterion. 

\begin{lemma}\label{L:NormGap}
Suppose non-colinear positive vectors $r,s\in V$ have that $h(r,r)$ and $h(s,s)$ are units in $\cO_F$ and $h(r,s)$ is in $\cO_E$. Then $H_r$ and $H_s$ are orthogonal or disjoint. 
\end{lemma}

The main case to keep in mind is $h(r,r)=1=h(s,s)$.
%
%

\begin{proof}
Note that $h(r,s) h(s,r)= |h(r,s)|^2$ is in $F$. Suppose $H_r$ and $H_s$ intersect. Then \cref{L:BasicOrth} gives 
$$|h(r,s)|^2 < h(r,r) h(s,s).$$
For non-trivial Galois conjugates $\sigma$ as above, the fact that the Galois conjugate forms are definite, together with Cauchy-Schwarz, gives 
$$\sigma(h(r,s) h(s,r)) \leq \sigma(h(r,r)) \sigma(h(s,s)).$$
Appropriately multiplying all these together gives 
$$N_{F/\bQ}(h(r,s) h(s,r)) < N_{F/\bQ}(h(r,r)) N_{F/\bQ}(h(s,s)).$$
As a product of norms of units, the right side is~$\pm1$.
The left side is non-negative, since it is a product of non-negative terms $$\sigma(h(r,s) h(s,r))=\sigma(h(r,s)) \overline{\sigma(h(r,s))}.$$
Therefore
$N_{F/\bQ}(|h(r,s)|^2)=0$, and hence $h(r,s)=0$ as desired.  
\end{proof}

Already this can give infinitely many examples where \cref{T:main} or \cref{T:mainrel} applies. For example,  let $E$ be a cyclotomic field, let $F$ be its totally real subfield, let $a\in \cO_F$ be a negative number all of whose non-trivial conjugates are positive, and define 
$$h(z,w) = z_1 \overline{w}_1 + \cdots + z_m \overline{w}_m + a z_{m+1} \overline{w}_{m+1}.$$
Then \cref{L:NormGap} gives that the unit norm elements in $\cO_E^{m+1}\subset E^{m+1}$ define an orthogonal arrangement invariant under a lattice.  Whether \cref{T:main} applies depends on whether its hypothesis
\ref{CuspLimit} holds.  This is automatic when the lattice is cocompact, for example whenever $F\neq\bQ$. Compare to \cite[Example 2.13]{Fortman}. 

In the non-cocompact setting, it is difficult in general to get that every cusp contains a 1-dimensional stratum. Take for example the case when $\cO_E$ is the ring of Eisenstein integers, and consider the standard hermitian form 
$$h(z,w) = z_1 \overline{w}_1 + \cdots + z_m \overline{w}_m - z_{m+1} \overline{w}_{m+1}.$$
The lattice here is $PU(\cO_E,h)$ as above.
In general, cusps correspond to orbits of primitive isotropic lines,  and to each such line $\ell$ one can associate the positive definite lattice $\ell^\perp/\ell$. In this specific case this gives a bijection between cusps and positive definite hermitian self-dual lattices of rank $m-1$, and there is only one such when 
$m\leq 6$. See \cite[Section 7]{AllcockNew} for this correspondence, and \cite{Feit} for the
original classification of such lattices. Thus there is only one cusp when $m\leq 6$, and it contains a 1-dimensional stratum, but already for $m=7$ there is a second cusp that is  disjoint from the hyperplane arrangement. 

Consulting \cite[Remark 1.2, Examples 1.1, Proposition 3.2, Table 2, Corollary 5.3]{ArtebaniSarti} and using \cite[Theorem 7.1]{AllcockNew}, it is possible to see that the moduli spaces $$M_{3,0},M_{4,1},M_{5,2},M_{6,3},M_{7,4},M_{8,5}$$ of K3 surfaces with non-symplectic automorphisms of order 3 with no fixed curve of higher genus are all quotients of complex balls with hyperplane arrangements defined by the norm 1 elements with entries in the ring of Eisenstein integers.  In these cases the balls have dimensions $6, 5, 4, 3, 2,1$ respectively, and the hyperplane arrangements are orthogonal by \cref{L:NormGap}. See  \cite{MaOhashiTaki} for more on these spaces, and note in particular that $M_{5,2}$ is isomorphic to the moduli space of smooth complex cubic surfaces. The previous paragraph shows that each corresponding ball quotient has a single cusp, and that it lies in the closure of a dimension 1 stratum.

\subsection{Moduli spaces of abelian varieties with extra structure (unitary Shimura varieties of PEL type)}

Take $E$ and $F$ as in \cref{SS:UnitNorm}, and 
assume given a free $\cO_E$-module $\Lambda$ of rank $m+1$ and a hermitian form 
$$h : \Lambda \times \Lambda \to \cO_E.$$  As in \cref{SS:UnitNorm}, 
with $V=\Lambda\otimes_{\cO_E}E$, we assume
$h$ has signature $(m,1)$ and its other Galois conjugates have signature $(m+1,0)$. 
We define a short root to be any $r\in \Lambda$ with $h(r,r)=1$. \cref{L:NormGap} immediately implies that the hyperplane arrangement defined by the short roots is orthogonal, and so \cref{T:mainrel} (and sometimes \cref{T:main}) applies. 

This setup includes in particular the class of ``admissible hermitian lattices'' of \cite[Section 2.1]{Fortman}, which additionally places a restriction on the  different ideal  of  $\cO_E$.  In this setting, \cite[Section 11]{Fortman} shows that the ball quotient associated to $(\Lambda, h)$ is a moduli space of polarized $\cO_E$-linear abelian varieties, and the divisor defined by the short roots has a natural modular interpretation as the locus where an additional suitable morphism from a fixed CM torus exists. In this context the divisor is also known as the norm-one Kudla–Rapoport divisor. 
%
%
%

The case of quadratic imaginary fields $E$ is especially nice, and the condition on different ideals is automatic \cite[Example 2.12]{Fortman}. In that case the ball quotient minus the norm 1 divisor parametrizes abelian varieties with a suitable action of $\cO_E$, of the fixed signature and polarization type, but without a factor of the fixed CM elliptic curve $\bC/\cO_E$ as a polarized $\cO_E$ summand. 

\subsection{Orthogonality via root lattices} 
We now consider the following setup. Suppose $\Lambda$ is a lattice over~$\Z$ of signature $(N,2)$, and $g\in O(\Lambda)$ has  order $d>1$. Let $\zeta$ be a primitive $d$-th root of unity, and assume the eigenspace 
$$V=\ker(g-\zeta) \subset \Lambda \otimes_\bZ \bC$$
has signature $(m,1)$ for some $m$, with respect to the hermitian form $h(v,w) = (v, \overline{w})$. Then we can consider the complex hyperbolic space $\bC \bH^m$ defined by $V$, with the hyperplane arrangement defined by the ($g$-orbits of) norm $2$ elements of $\Lambda$ that define hyperplanes in $\bC\bH^m$.

\begin{proposition}\label{P:coprimeto30}
If $d$ is co-prime to 30, and $g$ does not have any non-zero fixed vectors, then this arrangement is orthogonal. 
\end{proposition} 

The assumption could be weakened, but here we prefer  a simple statement. The proof can be compared to \cite[Section 8]{DolgachevvanGeemenKondo} and is omitted. The key is to show the following: Let $Q$ be an irreducible ADE root lattice with an isometry $u$ of order $k$ co-prime to 30 and with no non-zero fixed vectors. Then $Q$ has type $A_{k-1}$, the isometry $u$ is a $k$-cycle in its Weyl group~$S_k$, and in particular every non-trivial eigenspace of $u$ on $Q\otimes \bC$ is one dimensional.

\begin{remark}
The case $d=7$ is the motivating example for this subsection. The moduli space of K3 surfaces with a non-symplectic automorphism of order 7 has two components in \cite[Theorem 9.5]{ArtebaniSartiTaki}, both of which are quotients of two-dimensional balls minus hyperplane arrangements which are orthogonal by \cref{P:coprimeto30}. The discussion at the beginning of \cite[Section 9]{ArtebaniSartiTaki} gives the setup.
%
%
%
%
In this case there are no cusps by \cref{R:SimplestTypeCocompact},
%
%
so \cref{T:main} applies. 
\end{remark}

\subsection{Orthogonality via commuting isometries} 

We start with the following well known observation, whose proof is left to the reader. Compare to \cite[Remark 2.3]{Deraux}. 

\begin{lemma}\label{L:commuting}
Suppose $\Gamma_0\back \bC\bH^m$ is a manifold and $H$ is a finite group of holomorphic isometries of $\Gamma_0\back \bC\bH^m$ such that for all $x\in \Gamma_0\back \bC\bH^m$ we have that the stabilizer of $x$ in $H$ is Abelian. Let $\cD \subset \Gamma_0\back \bC\bH^m$ be the union of the complex codimension 1 components of the fixed point sets of elements of $H$. Then the preimage $\cH$ of $\cD$ in $\bC\bH^m$ is an orthogonal arrangement. 
\end{lemma}

There are a few known line arrangements in $\bC\bP^2$ with finite (branched) Galois covers to which this applies. The complements of these line arrangements in $\bC\bP^2$ have the form  $\Gamma\back (\bC\bH^2-\cH)$ with $\cH$ orthogonal by \cref{L:commuting}. 
See \cite{Hirzebruch,BarthelHirzebruchHofer,Tretkoff}. Important examples include Klein's arrangement, studied for example in \cite{NarukiKlein}, and the Hesse arrangement, studied for example in \cite{Kaneko}. In both cases there is in fact more than one  complex hyperbolic metric known on the complement of the line arrangement, but in both cases there is (at least) one with $\Gamma_0$ cocompact where \cref{L:commuting} applies on a finite cover. See \cite[Definition 3.1, Tables 5.2 and 5.4, Theorem 6.4]{Tretkoff} as one possible starting point; and see \cite{DerauxKlein, CouwenbergHeckmanLooijenga} for some additional context. So we get that \cref{T:main} applies to the fundamental groups of complements in $\bC\bP^2$ of the Hesse and Klein arrangements. 
%
%

\subsection{Additional examples and special cases}\label{AA:Additional}

Here are some notable examples. 

\bold{Stable cubic threefolds:} A cubic threefold means a hypersurface in~$\C P^4$ defined
by a cubic equation.  Stability is a concept from Geometric Invariant Theory, that turns
out to be equivalent to every singularity having one of the types $A_1, A_2, A_3, A_4$
\cite{AllcockThreefolds, Yokoyama}.  
The moduli space of stable cubic threefolds,
or equivalently the moduli space of $K$-stable cubic threefolds \cite{LiuXu},
may be identified with the complex $10$-ball, minus a family of disjoint hyperplanes,
modulo a discrete group \cite[Theorems 3.7, 7.2]{AllcockCarlsonToledoThreefolds}, \cite{LooijengaSwierstra}.
Assumption \ref{CuspLimit} of \cref{T:main} 
fails since there are no 1-dimensional strata and the complex hyperbolic orbifold is not compact (see \cite[Theorem 4.10]{AllcockCarlsonToledoThreefolds}), but \cref{T:mainrel} applies. 
Caution: this identification with a ball quotient 
is an isomorphism of varieties, but not
one of orbifolds.  (Away from additional hyperplanes, that represent singular threefolds
and form a non-orthogonal arrangement,
it is also an orbifold isomorphism.)

\bold{Stable quartic plane curves:}  
This is similar to the previous example, with the $6$-ball in place of the $10$-ball
\cite{KondoGenus3}.
Now stability is equivalent to having only $A_1$ or $A_2$ singularities.
The
same caution applies, regarding orbifold structure along the hyperplanes representing
singular curves.
We remark that smooth quartic plane curves, up to projective equivalence, are the
same as non-hyperelliptic genus 3 curves, up to isomorphism.  The points of $B^6$,
whose removal leaves the stable locus, represent hyperelliptic genus~$3$ curves.  
So, by adding them back, the full
ball quotient 
is the moduli space of genus 3 curves with mild degenerations allowed.  

\bold{Moduli of 5 points, or smooth complex binary quintics:} Here the moduli space, also known as $\cM_{0,5}/S_5$, can be realized as the complement of an orthogonal arrangement in a compact ball quotient of dimension 2, so \cref{T:main} applies. This has been known in some form for a long time \cite{ShimuraOsaka, ShimuraAnnals, DeligneMostow}, but we recommend \cite[Section 4]{FortmanFive} as a reference that treats orthogonality explicitly. Up to finite covers, this moduli space is also the moduli space of smooth degree 4 del Pezzo surfaces \cite[Section 6]{HassettKreschTschinkel}, \cite{DolgachevRational}, and also a certain moduli space of K3 surfaces with order 5 non-symplectic automorphisms \cite{Kondo5, ArtebaniSartiTaki}.

\begin{remark}\label{R:M05both}
We also have that 
$$\cM_{0,5}\simeq (\bC \bP^1-\{0,1,\infty\})^2 - \Delta \simeq \Gamma(2)\times \Gamma(2) \back (\bH^2 \times \bH^2 - \cH),$$
where $\Delta$ is the diagonal and $\cH$ is its preimage and $\Gamma(2) \subset PSL(2,\bZ)$ is the principal level 2 congruence subgroup. Since $\bH^2 \times \bH^2$ is the $SO(2,n)$ symmetric space for $n=2$, this shows that $\cM_{0,5}$ fits both within the complex hyperbolic and type IV framework. 
\end{remark}

\bold{Moduli of 6 points, or smooth complex binary sextics:} Here the moduli space, also known as $\cM_{0,6}/S_6$, is the three-dimensional Eisenstein example presented above, so \cref{T:main} applies. Again this goes back at least to
\cite{DeligneMostow}, but see also \cite[Theorem 3]{AllcockCarlsonToledoSextics} and \cite[Section 12]{DolgachevKondoModuli} as a reference. 

\bold{Deligne-Mostow-Thurston examples and mapping class groups:} Deligne-Mostow and later Thurston gave a number of examples of ball quotients \cite{DeligneMostow, Thurston}. In the Deligne-Mostow point of view, they parameterize certain cyclic covers of the sphere, and in the Thurston point of view they parametrize certain cone metrics on the sphere. These ball quotients come equipped with a natural divisor corresponding to certain collisions between points. (Other collisions between points live ``at infinity'' on the ball quotient, or are not allowed.) 

We now refer to the table of 94 examples in \cite[Appendix]{Thurston}. In each, the divisor is orthogonal if, for every triple of numerators, their sum is at least as large as the sum of the remaining numerators.
The entries on Thurston's list with this property are: $1,2,8,9,42,43,47,48,72,73,74,75,77,78,83,84,87,89,94$. 
Each is a lattice in $PU(2,1)$, except the first entry, which is a lattice in $PU(3,1)$.
All are cocompact except $1,2,8,42,43,73$,
%
%
%
%
and Assumption \ref{CuspLimit} holds in all these cases. 
%
%
%
Thus, \cref{T:main} applies to all 19 examples listed above. 

The moduli of cubic surfaces is not a Deligne-Mostow-Thurston example \cite[Theorem 8.4]{AllcockNew} but it is commensurable with one \cite[Theorem 3]{DoranArxiv}, \cite[Section 3]{DolgachevvanGeemenKondo},  \cite[Section 3]{CasalainaMartinGrushevskyHulek}. This Deligne-Mostow-Thurston example is not one of the orthogonal ones listed above, and the commensurability result involves passing to a cover and removing only certain components of the divisor discussed above.  

We will not make any attempt to exhaustively list the coincidences and relationships among the examples we have presented, but it should be noted that examples number 1 and 9 on Thurston's list correspond to the cases of 6 and 5 points respectively on $\bC \bP^1$  discussed above, and we wish to emphasize that the full mapping class groups for the 5 and 6 times punctured sphere are valid examples of $\Gammahat$ to which \cref{T:main} applies. More generally, the other orthogonal examples on Thurston's list show that various partially labeled finite index subgroups of these two mapping class groups can be realized in different ways as a $\Gammahat$ to which \cref{T:main} applies.

\begin{remark}\label{R:NotMCG}
Our groups are similar enough to mapping class groups that we feel compelled to briefly sketch why, in the cubic surface case, $\Gammahat$ is not abstractly commensurable with a mapping class group.  By \cite[Theorem 1.15]{BehrstockHagenSistoQuasiflats} and \cref{C:MaxSizeOrth}, if this $\Gammahat$ were abstractly commensurable with a mapping class group of a genus $g$ surface with $k$ punctures, we'd have $3g-3+k=4$.  For such surfaces, the abstract commensurator is the extended mapping class group; see \cite{Ivanov,Korkmaz} and \cite[Theorem 6]{BehrstockMargalit}.  On the other hand, the Fermat cubic gives a finite subgroup $(\mathbb Z/3)^3\rtimes S_4<\Gammahat$ of order $648$ \cite[Theorem~9.5.8]{DolgachevBook}, and this subgroup injects into the abstract commensurator  of $\Gammahat$.  
%
%
But finite subgroups of the extended mapping class group have order at most $60$ in the cases $S_{0,7}$, $S_{1,4}$, and $S_{2,1}$ \cite{May,Kerckhoff}, so we can conclude that $\Gammahat$ is not abstractly commensurable with a mapping class group. This argument can be compared to arguments in \cite{ClayLeiningerMargalit, BrendleMargalit, Soroko}, and is based on a conversation with ChatGPT 5.5 Pro. 

Additionally invoking quasi-isometric rigidity of mapping class groups \cite{Hamenstaedt,BehrstockKleinerMinskyMosher,BehrstockHagenSistoQuasiflats}, it is possible to prove that the  $\Gammahat$ in the cubic surface case is not even quasi-isometric to a mapping class group. 
%
%
\end{remark}
%
%

\bibliographystyle{amsalpha}

\bibliography{bibliography}

\end{document}